\documentclass[12pt]{amsart} 

\usepackage{graphicx,amssymb,colordvi,textcomp,latexsym} 

\usepackage{tikz}
\usepackage{ulem}

\usepackage[title]{appendix}
\usepackage[showframe=false]{geometry}
\usepackage{changepage}

\usepackage{tikz}
\usetikzlibrary{arrows,shapes, positioning, matrix, patterns, decorations.pathmorphing, arrows.meta,automata}

\usepackage{color}
\usepackage{colortbl}
\usepackage{soul}

\usepackage{array}
\usepackage[hidelinks]{hyperref}
\usepackage{amsfonts}
\usepackage{amsmath}
\usepackage{amsthm}

\usepackage{adjustbox}
\usetikzlibrary{arrows,shapes, positioning, matrix, patterns, decorations.pathmorphing}

\newcommand{\R}{\mbox{$\mathbb{R}$}}

\newcommand{\NN}{\mathbf{N}}

\newtheorem{lemma}{Lemma}[section]

\newtheorem{prop}[lemma]{Proposition}
\newtheorem{thm}[lemma]{Theorem}
\newtheorem{cor}[lemma]{Corollary}

\theoremstyle{definition}
\newtheorem{Def}[lemma]{Definition}
\newtheorem{exam}[lemma]{Example}
\newtheorem{exams}[lemma]{Examples}

\theoremstyle{remark}
\newtheorem{rem}[lemma]{Remark}

\newcommand{\etal}{{\it et al.}}

\title{Synchrony and Anti-synchrony in Weighted Networks}
\author{Manuela Aguiar}
\address{Manuela Aguiar, Faculdade de Economia, Centro de Matem\'atica, Universidade do Porto,
Rua Dr Roberto Frias, 4200-464 Porto, Portugal.}
\email[Corresponding author]{maguiar@fep.up.pt}
\author{Ana Dias}
\address{Ana Dias, Departamento de Matem\'atica, Centro de Matem\'atica, Universidade do Porto,
Rua do Campo Alegre, 687, 4169-007 Porto, Portugal}
\email{apdias@fc.up.pt}

\date{\today}

\begin{document}

\maketitle

\begin{abstract}
We consider weighted coupled cell networks, that is networks where the interactions between any two cells have an associated weight that is a real valued number. Weighted networks are ubiquitous in real-world applications. We consider a dynamical systems perspective by associating to each network a set of continuous dynamical systems, the ones that respect the graph structure of the network. For weighted networks it is natural for the admissible coupled cell systems to have an additive input structure. We present a characterization of the synchrony subspaces and the anti-synchrony subspaces for a weighted network depending on the restrictions that are imposed to their admissible input-additive coupled cell systems. These subspaces are flow-invariant by  those systems and are generalized polydiagonal subspaces, that is, are characterized by conditions on the cell coordinates of the types $x_i = x_j$ and/or $x_k = -x_l$ and/or $x_m=0$. The existence and identification of the synchrony and anti-synchony subspaces for a weighted network are deeply relevant from the applications and dynamics point of view. Our characterization of the synchrony and anti-synchrony subspaces of a weighted network follows from our results where we give  necessary and sufficient conditions for a generalized polydiagonal to be left invariant by the adjacency matrix and/or the Laplacian matrix of the network.

\vspace{3mm}

\noindent  AMS classification scheme numbers: 05C50 05C22 05C90 34C15

\vspace{3mm}

\end{abstract}

\tableofcontents

\section{Introduction}

Networks are often used to model many applications in a huge set of scientific areas, see Arenas~\etal~\cite{ADMZ08} and references therein. From the dynamical systems perspective, an ultimate goal is to use properties of 
the {\it network}, as a graph object, to induce features for the associated {\it coupled cell systems}, the ones that respect the graph structure of the network. 
Here, we consider systems of ordinary differential 
equations. In the coupled cell systems formalism of Stewart, Golubitsky and collaborators~\cite{GSP03,GST05} and in the one of Field~\cite{F05}, 
the network connections have assigned nonnegative integer values. When the values associated with the connections can be any real number, then we  have 
{\it weighted networks}. See Aguiar, Dias and Ferreira~\cite{ADF17} and Aguiar and Dias~\cite{AD18}. In the context of weighted networks, it is common to assume that the coupled cell systems with structure consistent with the network have {\it additive input structure}, that is, the input to any cell is a sum of the {\it pairwise interactions} between the cell and the cells  connected to it, scaled by the weight of the connection. See, for example, Field~\cite{F15}, Bick and Field~\cite{BF17} and Newman~\cite{N10}. 

An important achievement in the two mentioned coupled cell systems formalisms is the characterization of the {\it synchrony spaces} for a network, the polydiagonals defined by equalities of cell coordinates, which are flow-invariant by any coupled cell system associated with  the network structure. Moreover, their existence and characterization relies only on the 
network structure. In fact, algorithms exist that determine the set of network synchrony spaces using solely the network adjacency matrix (or matrices in case there 
is more than one type of interactions between cells). See, for example, Aguiar and Dias~\cite{AD14}. Remark 2.11 of Aguiar and Dias~\cite{AD18} and Theorem 2.4 of Aguiar, Dias and Ferreira~\cite{ADF17} state that the same holds for weighted networks considering coupled cell systems with additive input structure.  In this work we consider a combination of the additive input structure of the coupled cell systems with restrictions on the internal dynamics and on the coupling functions and show that this  can lead to a drastic increase on the type of robust phenomena that the systems can exhibit. As it is widely known, the existence of robust flow invariant subspaces has a strong impact in the dynamics and favor the existence of non-generic dynamical behavior like robust heteroclinic cycles and networks and bifurcation phenomena. See, for example, Aguiar~\etal~\cite{AADF11}, Field~\cite{F15}, Golubitsky~\etal~\cite{GNS04} and Golubitsky and Lauterbach~\cite{GL09}.

Let $G$ be an $n$-cell  weighted network. Consider a coupled cell system  with additive input structure  associated with $G$ 
given by $\dot{x} = f(x)$, where $f = (f_1, \ldots, f_n)$ so that the equation  $\dot{x}_j = f_j(x)$ is associated with cell $j$ and it has the form: 
\begin{equation} 
\dot{x}_j=g(x_j) +\sum_{i=1}^n {w_{ji}h\left(x_j,x_i\right)}\, .
\label{eq:intro_EDOsystem}
\end{equation}
Here,   $g:\R \rightarrow \R$ and $h:\R \times \R \rightarrow \R$ are smooth functions and characterize the internal dynamics and the coupling function, respectively; 
each $w_{ji}\in \R$ is the value of the weight of the coupling strength of the directed edge from cell $i$ to cell $j$. If there is no directed edge from cell $i$ to cell $j$, the weight is assumed to be zero. Let $W_G = [w_{ij}]$ denote the $n \times n$ adjacency matrix of $G$. A polydiagonal space $\Delta$ is left invariant under any coupled cell system of the form (\ref{eq:intro_EDOsystem}) if and only if it is left invariant under $W_G$. See~\cite{GSP03, GST05, F05, ADF17,AD18}. 

In the literature, it is often common to assume, in addition to this additive input structure of the systems (\ref{eq:intro_EDOsystem}),  certain restrictions on the coupling function $h$. One usual restriction is 
\begin{equation}\label{eq:asump_h}
h(x,x) = 0,\, \forall x \in \R\, .
\end{equation}
Now observe  that $h$ satisfies the hypothesis (\ref{eq:asump_h}) if and only if  
\begin{equation}\label{eq:form_h}
h(x,y) = (x-y) h_1(x,y),\, \forall x,y \in \R,
\end{equation}
 for some smooth function $h_1:\, \R^2 \to \R$. Using this notation, equation (\ref{eq:intro_EDOsystem}) becomes
\begin{equation} 
\dot{x}_j=g(x_j) +\sum_{i=1}^n {w_{ji} (x_j - x_i) h_1\left(x_j,x_i\right)}\, .
\label{eq:3intro_EDOsystem}
\end{equation} 
Note that, for equations of the form (\ref{eq:3intro_EDOsystem}), we have the full synchronized space $\Delta_0 = \{ x:\, x_1 = x_2 = \cdots = x_n\}$, as $\Delta_0$ is left invariant under the flow for any choice of the functions $g$ and $h_1$.

Examples of coupled cell systems of the form (\ref{eq:3intro_EDOsystem}) are the {\it exo-difference-coupled cell systems} considered by Neuberger, Sieben, and Swift~\cite{NSS19} and the {\it diffusive networks} addressed for example by Poignard, Pade and Pereira~\cite{PPP19} where 
\begin{equation} \label{eq:dif}
h(x,y)  = H(x-y)
\end{equation}
for smooth $H:\, \R \to \R$ such that $H(0) = 0$. In  \cite{PPP19}  it is addressed the way the structure of the coupling structure of the graph affects the transverse stability of the full synchronized space $\Delta_0$. 

Note also that coupled cell systems where  cell equations have the form (\ref{eq:3intro_EDOsystem}) are consistent with the network obtained from $G$ by considering the entries $w_{jj}$ equal to zero, and so, a polydiagonal space $\Delta$ is left invariant under any coupled cell system of the form (\ref{eq:3intro_EDOsystem}) if and only if it is left invariant under the adjacency matrix obtained from $W_G$ by taking the diagonal entries equal to zero. Trivially, that is equivalent to checking if $\Delta$ is invariant under the Laplacian matrix $L_G$. Recall that, taking the adjacency matrix $W_G$ of the weighted network $G$, the {\it Laplacian matrix} associated with it is given by $L_G=D_G-W_G$, where $D_G$ is the diagonal matrix with the input valencies of the cells at the diagonal.

In Aguiar and Dias~\cite{AD18}, we consider synchronization for weighted networks and end with examples of flow-invariant subspaces whose definition includes, besides  cells coordinates that are equal, cells coordinates with the same magnitude but opposite sign.  Using the terminology of Neuberger, Sieben, and Swift~\cite{NSS19}, these are  {\it anti-synchrony subspaces}. In Neuberger, Sieben, and Swift~\cite{NSS19}, the authors consider four nested sets of difference-coupled systems, for $0-1$ undirected networks, by adding restrictions on the internal dynamics and coupling functions, and characterize the synchrony and anti-synchrony subspaces for a network with respect to those four subsets of admissible difference-coupled systems.
Anti-synchronization has been observed in different coupled cell systems scenarios, as coupled oscillators, Kim~\etal~\cite{KRKRP03}, neural networks,  Meng and Wang~\cite{MW07} and muti-agent systems, Hu and Zheng~\cite{HZ13}.

 Motivated  by the works of Aguiar and Dias~\cite{AD18} and Neuberger, Sieben, and Swift~\cite{NSS19}, we introduce the definition of generalized polydiagonal subspace as a space where the defining conditions, besides equalities of cell coordinates, can also include conditions such as $x_i = -x_j$ or $x_k =0$ and characterize the generalized polydiagonals that are left invariant by the adjacency matrix and/or by the Laplacian matrix of a weighted network. We then get the synchrony and anti-synchrony subspaces for different subclasses of the class of the input additive coupled cell systems of the network. Those subclasses are defined by considering restrictions on the internal dynamics and on the coupling functions, namely, to be odd, or even, or linear.  As more restrictions are imposed, more types of flow-invariant spaces can occur for any such systems.

We give next a very simple illustration of this by presenting a weighted network that has no robust synchrony spaces when taking general coupled cell systems with additive input structure but when a restriction is added on the coupling function, then the opposite occurs. 

\begin{figure}[!h]
\begin{tabular}{cc}
\begin{tikzpicture}
 [scale=.15,auto=left, node distance=1.5cm]
 \node[fill=magenta,style={circle,draw}] (n1) at (4,0) {\small{1}};
  \node[fill=magenta,style={circle,draw}] (n2) at (14,0) {\small{2}};
 \node[fill=white,style={circle,draw}] (n3) at (24,0)  {\small{3}};
 \draw[->, thick] (n1) edge[thick, bend left=30]  node  [near end, above=0.1pt] {{\tiny $2$}} (n2); 
 \draw[->, thick] (n2) edge[thick, bend left=30,]  node  [near end, below=0.1pt] {{\tiny $1$}} (n1); 
 \draw[->, thick] (n2) edge[thick] node  [near end, above=0.1pt] {{\tiny $3$}} (n3); 
\end{tikzpicture}  \qquad & \qquad 
\begin{tikzpicture}
 [scale=.15,auto=left, node distance=1.5cm]
 \node[fill=magenta,style={circle,draw}] (n1) at (4,0) {\small{$[1]$}};
  \node[fill=white,style={circle,draw}] (n3) at (14,0) {\small{$[3]$}};
 \draw[->, thick] (n1) edge[thick]  node  [near end, above=0.1pt] {{\tiny $3$}} (n3); 
 \end{tikzpicture} 
\end{tabular}
\caption{(Left) A three-cell network. (Right) A two-cell network.} \label{f:simples}
\end{figure}
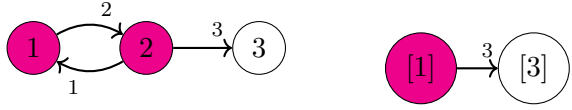

Coupled cell systems with additive input structure consistent with the three-cell network in Figure~\ref{f:simples} have the form:
$$
\left\{ 
\begin{array}{rcl}
\dot{x}_1 & = & g(x_1) +  h(x_1,x_2) \\
\dot{x}_2 & = & g(x_2) + 2 h(x_2,x_1) \\
\dot{x}_3 & = & g(x_3) + 3 h(x_3,x_2)
\end{array} \, .
\right.
$$
Note that for this network, there are no synchrony spaces. Assume now hypothesis (\ref{eq:form_h}). Thus we have the coupled cell system 
given by 
$$
\left\{ 
\begin{array}{rcl}
\dot{x}_1 & = & g(x_1) + (x_1 -x_2) h_1(x_1,x_2) \\
\dot{x}_2 & = & g(x_2) + 2 (x_2 - x_1) h_1(x_2,x_1) \\
\dot{x}_3 & = & g(x_3) + 3 (x_3 - x_1) h_1(x_3,x_2) 
\end{array} 
\right.
$$
and  $\Delta = \{ x:\, x_1 = x_2\}$ is left invariant under the flow of any such coupled cell system. Moreover, the restriction to $\Delta$ gives the 
coupled cell system 
$$
\left\{ 
\begin{array}{rcl}
\dot{x}_1 & = & g(x_1)  \\
\dot{x}_3 & = & g(x_3) + 3 (x_3 - x_1) h_1(x_3,x_1) 
\end{array} 
\right.
$$
which is consistent with the two-cell network in Figure~\ref{f:simples}. 

In this paper, we caracterize the set  of the synchrony and anti-synchrony subspaces of a general weighted network $G$, which correspond to the generalized polydiagonals invariant under the adjacency and/or Laplacian matrices of $G$. More precisely, the set of  synchrony and anti-synchrony subspaces of a general weighted network $G$ corresponding to the generalized polydiagonals that are flow-invariant by any coupled cell system with input additive structure that are linear-balanced, that is, those where the internal dynamics function is odd and the coupling function is odd and linear, are in correspondence with the generalized polydiagonals invariant under the network Laplacian matrix. See Proposition~\ref{prop:IGL}. The set of  synchrony and anti-synchrony subspaces of a general weighted network $G$ corresponding to the generalized polydiagonals that are flow-invariant by any coupled cell system with input additive structure that are even-odd-balanced, that is, those where the internal dynamics function is odd and the coupling function is even in the first variable and odd in the second variable, are in correspondence with the generalized polydiagonals invariant under the network adjacency matrix. See Proposition~\ref{prop:eoia}.

It follows so from our results that the caracterization of the set  of the synchrony and anti-synchrony subspaces of a general weighted network, for the above classes of coupled 
cell systems with input additive structure follows from the characterization of the generalized polydiagonals invariant under the adjacency and/or Laplacian matrices of the network. Using an extension of the results of Aguiar and Dias~\cite{AD14}, we characterize the  synchrony and anti-synchrony subspaces of a general weighted network by considering generalized polydiagonal subspaces and using the eigenvalue and eigenvector structures of the adjacency matrix and the Laplacian matrix. See Section~\ref{sec:algm}.

The paper is organized in the following way. In Section~\ref{sec:wn}, we establish notation and a few facts concerning weighted networks. In Section~\ref{sec_gen_tag} we introduce the definitions of generalized polydiagonal subspace and of the associated  tagged partition of the network set of cells. These include, as particular cases, the definitions of polydiagonal subspace and of the network set of cells associated partition. The characterization of the generalized polydiagonal subspaces that are left invariant by the adjacency matrix and/or the Laplacian matrix of a weighted network appears in Section~\ref{sec:gen_poly}. This is done through necessary and sufficient conditions on the blocks of any block structure of the weighted and Laplacian network matrices adapted to the generalized polydiagonal subspace and leads to the definition of several kinds of tagged partitions.
In Section~\ref{sec_CCNS}, we review coupled cell systems with additive input structure and, following the terminology in Neuberger, Sieben, and Swift~\cite{NSS19}, define subclasses of these coupled cell systems, namely, exo-input-additive, odd-input-additive, and linear-input-additive coupled cell systems.
We also define the class of even-odd-input-additive coupled cell systems.

 In Sections~\ref{sec_bal}-\ref{sec_eo_bal}, using the results obtained in Section~\ref{sec:gen_poly}, we characterize the synchrony and the anti-synchrony subspaces for weighted coupled cell networks depending on the additional restrictions imposed to their input-additive admissible coupled cell systems.
In Section~\ref{sec:algm}, we show, for the adjacency matrix and for the Laplacian matrix of a network, that the set of the generalized polydiagonals that are left invariant by the matrix is a lattice.  We show that the work in Aguiar and Dias~\cite{AD14} generalizes easily to the lattice of synchrony and anti-synchrony subspaces and how to apply the algorithm there to find these two lattices and, thus, the set of the synchrony and anti-synchrony subspaces of the network.
We endup with some conclusions in Section~\ref{sec:conclu}.

\section{Weighted networks}\label{sec:wn}

We consider {\it weighted networks}, that is, networks given by directed graphs where edges have associated weights, given by real values. 

If $G$ is an $n$-cell  weighted network, with set of cells $C=\{1,\ldots,n\}$, its  $n \times n$ {\it weighted adjacency matrix} is $W_G = [w_{ij}]_{1 \le i,j \le n}$, where $w_{ij}$ is the weight of the edge from cell $j$ to cell $i$ or zero if there is no such edge. The {\it input valency} of a cell $i \in C$, denoted by $v(i)$, is the sum of the weights of the edges directed to the cell $i$, that is, $v(i) = \sum_{j \in C} w_{ij}$. 

A network is said to be {\it regular} when all the network cells have the same input valency, that is, $v(i) = v(j)$, for all $i,j \in C$. When the network is regular, we have that its weighted matrix 
$W_G$ has constant row sum, say $v_W = \sum_{k=1}^n w_{ik}$, for $i= 1, \ldots, n$. In that case, we also say that $W_G$ is {\it regular} of {\it valency} $v_W$. 

\begin{Def} \normalfont 
Define the {\it row sum operator} ${\mathrm rs}:\, M_{s,t}(\R) \to M_{s,1}(\R)$ which maps an $s\times t$ matrix $M$ to the $s \times 1$ column matrix where the 
$i$-th entry is the sum of the entries of the $i$-th row of $M$, for $i=1, \ldots, s$.
\hfill $\Diamond$
\end{Def}

\begin{rem}\normalfont 
Note that $\left({\mathrm rs}\left( W_G\right)\right)_i = v(i)$, for $W_G$, the weighted adjacency matrix of a weighted network $G$ with set of cells $C$, and $v(i)$, the input valency of cell $i \in C$.

\hfill $\Diamond$
\end{rem}

We recall that, given an $n$-cell  weighted network $G$ with adjacency matrix $W_G=[w_{i j}]_{n\times n}$, the corresponding {\it Laplacian matrix} is given by $L_G=D_G-W_G$, where $D_G$ is the diagonal matrix with the principal diagonal given by the entries of ${\mathrm rs}\left( W_G\right)$, that is,  the input valencies of the cells at the diagonal. The Laplacian matrix associated with $G$ is the regular $n \times n$ matrix $L_G = [l_{ij}]_{n\times n}$ defined by: 
$$l_{ij} = 
\left\{ 
\begin{array}{l}
-w_{ij} \mbox{  if }  i\not=j; \\
v(i) - w_{ii} \mbox{  if }  i=j\, .
\end{array}
\right.
$$

\begin{rem}\normalfont \label{rmk:Laplacian}
The Laplacian matrix $L_G$ of a weighted network $G$ is regular with valency $0$. Thus, it can be seen as the weighted adjacency matrix of another weighted network,  which is regular as the input valency of each cell is zero. 
\hfill $\Diamond$
\end{rem}

\begin{Def}
Given a weighted network $G$, we denote by  $G_L$ the regular weighted network with  weighted adjacency matrix the Laplacian matrix $L_G$ of $G$. 
\hfill $\Diamond$
\end{Def}

\section{Generalized polydiagonals and tagged partitions}\label{sec_gen_tag}

A polydiagonal subspace $\Delta$ of $\R^n$ is a subspace of $\R^n$ characterized by equalities of the form $x_i = x_j$ where $x_i,x_j$ denote coordinates of  cells $i,j$. Such 
polydiagonal  can be associated with a partition of $C = \{ 1, \ldots, n\}$ into the disjoint union of a certain number of nonempty parts with union $C$ and such that $i,j$ belong to the same part if and only if $x_i = x_j$ is a condition in the definition of $\Delta$. In this section, we generalize the notion of polydiagonal subspace of $\R^n$ to include possibly equalities of the form $x_k = -x_l$ or $x_m = 0$.

\begin{Def} \label{def:tag_part} \normalfont   Let  $C = \{ 1, \ldots, n\}$ and $p,q,r$ be nonnegative integers, where $0\leq q\leq p \leq n$ and $r \in \{0,1\}$. \\
(i) A  {\it tagged partition} of $C$ determined by $p,q,r$ is a partition of $C$ into the disjoint union of $p+q+r$ parts,  $P_k$, for $k=1,\ldots,p$ if $p>0$, $\overline{P}_l$, for $l =1,\ldots,q$ if $q>0$ and $P_0$ if $r=1$. If $q > 0$, then  for $l =1,\ldots,q$, each part $\overline{P}_l$ is the {\it counterpart} of $P_l$. If $r =1$ then the part $P_0$ is called the {\it zero part}. If $r=0$ then there is no zero part. \\
(ii) A tagged partition with no counterparts and no zero part, that is, a tagged partition determined by $p,q,r$ where $q=0, r=0$,  it is called a {\it standard partition} of $C$ into disjoint $p$ parts, $P_1, \ldots, P_p$. \\
(iii) There is a unique tagged partition such that $p=0$ which we call the {\it null partition}. In that case, $q=0$ and $r=1$ and it corresponds to the partition of $C$ with only the zero part, that is, $P_0 = C$.
 \hfill $\Diamond$
\end{Def}
 
 We can associate with a tagged partition, a subspace of $\R^n$ in the following way:
 
\begin{Def} \label{def:gen_poly} \normalfont   
(i) Given a  tagged partition $\mathcal{P}$  of $C = \{ 1, \ldots, n\}$,  a {\it generalized polydiagonal  subspace} of $\R^n$ is a subspace of the form
{\small
$$
\hspace{-2mm} \begin{array}{l}
\Delta_{\mathcal{P}} =   \left\{ x \in \R^n:\  x_j = x_i  \ (x_j = -x_i) \mbox{ if  $j$ is in the same part of $i$ (if  $j$ is in the counterpart of $i$)}; \right. \\
  \left. \qquad   \qquad  \qquad  \ \ x_j=0  \mbox{ if $j$ is in the zero part}\right\}\, .
 \end{array}
$$
} \\
(ii) A {\it polydiagonal  subspace} of $\R^n$ is a particular case of a generalized polydiagonal subspace associated with a standard partition of $C$. \\
(iii) The {\it null subspace}  $\{ (0, \ldots, 0)\}$ of $\R^n$ is the generalized polydiagonal subspace associated with  the null partition of $C$.
\hfill $\Diamond$
\end{Def}

\begin{exam} \normalfont Consider the following tagged partitions of $C = \{1,2,3,4,5\}$:
$$
\begin{array}{l}
\mathcal{P}_1 = \left\{ P_1 = \{1,2\},\, P_2 = \{3\},\, \overline{P}_1 = \{4\}, P_0 = \{5\} \right\}, \\
\mathcal{P}_2 = \left\{ P_1 = \{1,2\},\, P_2 = \{3\},\, P_3 = \{5\},\,  \overline{P}_1 = \{4\} \right\}, \\
\mathcal{P}_3 = \left\{ P_1 = \{1,2\},\, P_2 = \{3\},\, P_3 = \{4\},\,  P_0 = \{5\} \right\}, \\
\mathcal{P}_4 = \left\{ P_1 = \{1,2\},\, P_2 = \{3\},\, P_3 = \{4,5\} \right\}\, . 
\end{array}
$$
The associated generalized polydiagonal subspaces are:
$$
\begin{array}{l}
\Delta_{\mathcal{P}_1} = \left\{ x \in \R^5:\, x_1 = x_2 = -x_4,\, x_5=0 \right\}, \ 
\Delta_{\mathcal{P}_2} = \left\{ x \in \R^5:\, x_1 = x_2\ = -x_4 \right\}, \\
\Delta_{\mathcal{P}_3} = \left\{ x \in \R^5:\, x_1 = x_2,\, x_5 = 0 \right\}, \ 
\Delta_{\mathcal{P}_4} = \left\{ x \in \R^5:\, x_1 = x_2,\, x_4 = x_5 \right\}\, . 
\end{array}
$$
\hfill $\Diamond$
\end{exam}

\begin{Def} \label{def:inputvalency} \normalfont   
Let $G$ be a coupled cell network with set of cells $C$ and $\mathcal{P}$ a tagged partition of $C$. We denote by 
$v_{P} (i)$ the {\it input valency of a cell $i$ relative to the part $P \in \mathcal{P}$}, given by
$$
v_{P} (i) = \sum_{j \in P} w_{ij}.
$$
\hfill $\Diamond$
\end{Def}

We call the {\it input relation} on $G$, the equivalence relation corresponding to the partition of the network set of cells where two cells are in the same part if and only if they have the same input valency.

\begin{Def} \label{def:partition} \normalfont   
Let $G$ be a weighted  network with set of cells $C=\{ 1, \ldots, n\}$ and  $\mathcal{P} = \{ P_1, P_2, \ldots, P_p, \overline{P}_1, \overline{P}_2, \ldots, \overline{P}_q, P_0\}$ a tagged partition of $C$. Cells in $C$ can be  renumbered, if necessary, so that the $n \times n$ adjacency matrix $W_G$  (resp. Laplacian matrix $L_G$) of $G$ have a block form where each submatrix of $W_G$ (resp. $L_G$) represents the edges between the cells of two parts of $\mathcal{P}$: 
\begin{equation}
\left( 
 \begin{array}{cccc|cccc|c}
Q_{11} & Q_{12} &  \cdots &Q_{1p} & 
R_{11} & R_{12} &  \cdots &R_{1q} & Z_{10} \\
 \vdots& \vdots & \cdots &\vdots & 
 \vdots& \vdots & \cdots &\vdots & \vdots \\
Q_{p1} & Q_{p2} & \cdots &Q_{pp} &
R_{p1} & R_{p2} &  \cdots &R_{pq} & Z_{p0} \\
& & & & & & & & \\
\hline 
& & & & & & & & \\
\overline{R}_{11} & \overline{R}_{12} &  \cdots &\overline{R}_{1p} & 
\overline{Q}_{11} & \overline{Q}_{12} &  \cdots &\overline{Q}_{1q} & \overline{Z}_{10} \\
 \vdots& \vdots & \cdots &\vdots & 
 \vdots& \vdots & \cdots &\vdots & \vdots \\
\overline{R}_{q1} & \overline{R}_{q2} & \cdots &\overline{R}_{qp} &
\overline{Q}_{q1} & \overline{Q}_{q2} &  \cdots &\overline{Q}_{qq} &\overline{Z}_{q0} \\
& & & & & & & & \\
\hline 
& & & & & & & & \\
Z_{01} & Z_{02} & \cdots &Z_{0p} &
\overline{Z}_{01} & \overline{Z}_{02} &  \cdots &\overline{Z}_{0q} & Z_{00} \\
 \end{array}
 \right)\, .
 \label{eq:oddbf}
 \end{equation}
 Thus, matrices $Q_{ij}$, $R_{ij}$ and $Z_{i0}$ represent the connections to part $P_i$ from parts $P_j$, $\overline{P}_j$ and $P_0$, respectively. Matrices $\overline{R}_{ij}$, $\overline{Q}_{ij}$ and $\overline{Z}_{i0}$ represent the connections to part $\overline{P}_i$ from parts $P_j$, $\overline{P}_j$, and $P_0$, respectively. Matrices $Z_{0j}$, $\overline{Z}_{0j}$ and $Z_{00}$ represent the connections to part $P_0$ from parts $P_j$, $\overline{P}_j$ and $P_0$, respectively. We say that the {\it enumeration of the network set of cells is adapted to the (tagged) partition} $\mathcal{P}$. 
\hfill $\Diamond$ 
\end{Def}

\begin{rem}\normalfont \label{rmk_block_match} 
The row sum operator can be applied to each block matrix in 
(\ref{eq:oddbf}) of Definition~\ref{def:partition}. If $B$ is a block matrix representing  the connections to part $L$ from part $H$, we have that 
$$
v_{H} (k) = \left( {\mathrm rs} \left( B \right) \right)_k , \quad k \in L\, .
$$
\hfill $\Diamond$
\end{rem}

\section{Generalized polydiagonals invariant by the adjacency matrix and/or by the Laplacian matrix of a weighted network}\label{sec:gen_poly}

\begin{prop} \label{prop:subset}
Let $G$ be a weighted coupled $n$-cell network with adjacency matrix $W_G$ and Laplacian matrix $L_G$. We have:\\
(i) The set of the generalized polydiagonal  subspaces that are invariant by the adjacency matrix $W_G$ coincides with the set of the generalized polydiagonal  subspaces that are invariant by the Laplacian matrix $L_G$ if and only if  if $G$ is regular.\\
(ii) If $G$ is not regular, the set of the polydiagonal  subspaces that are invariant by the adjacency matrix $W_G$ is strictly contained in the set of the polydiagonal  subspaces that are invariant by the Laplacian matrix $L_G$ .
\end{prop}

\begin{proof} 
(i) We have that $L_G = D_G - W_G$. Moreover, $G$ is a regular network with valency $v_w$ if and only if  $D_G = v_w I$, with $I$ the identity matrix of order $n$. In that case, we have that 
a space is invariant under $W_G$ if and only if it is invariant under $L_G$. In particular, that holds for   invariant generalized polydiagonal  spaces. If $G$ is not regular, then at least 
the diagonal space $\{ x:\, x_i = x_j, \mbox{ for all } i,j\}$ is invariant under $L_G$ but  not under $W_G$. Thus there is at least one generalized polydiagonal that is left invariant under $L_G$ but not under $W_G$. \\
(ii) Given a polydiagonal  subspace $\Delta$, consider the associated (standard) partition $\mathcal{P}$. That is, $\Delta=\Delta_{\footnotesize{\mathcal{P}}}$.
If $\Delta_{\footnotesize{\mathcal{P}}}$ is invariant by the adjacency matrix $W_G$, we have that any two cells $i,j$ in the same part have the same input valency $v (i) = v (j)$. It follows that the entries $ii$ and $jj$ of the diagonal matrix $D_G$ are equal and thus, that $\Delta_{\footnotesize{\mathcal{P}}}$ is also left invariant by $D_G$ and, consequently, by $L_G$. As already mentioned in the proof of (i), since the Laplacian matrix $L_G$ is regular, the polydiagonal subspace where all the variables are identified is always invariant by $L_G$ but it is not invariant by $W_G$, in the case where $G$ is not regular. Thus, when $G$ is not regular, the set of the polydiagonals invariant by $W_G$ is strictly contained in the set of the  polydiagonals invariant by $L_G$.
\end{proof}

\begin{rem}\normalfont 
Given Proposition~\ref{prop:subset} (ii) we can ask, when $G$ is not a regular network, if the set of the generalized polydiagonal  subspaces that are invariant by the adjacency matrix $W_G$ is contained in the set of the generalized polydiagonal  subspaces that are invariant by the Laplacian matrix $L_G$. The following example shows that there can be generalized polydiagonal  subspaces that are invariant by the adjacency matrix $W_G$ but not by the Laplacian matrix $L_G$. 
\hfill $\Diamond$
\end{rem}

\begin{exam} \label{ex:notequal}		
Let $G$ be the four-cell non-regular network in Figure~\ref{fig:notequal} with adjacency matrix 
$$
W_{G} =
	\left( 
 \begin{array}{cc|cc}
3 & 1 & 1& 1  \\
1 & 1 & 0& 0  \\
\hline 
0 & 0 & 5& -3  \\
4 & 2 & 5& 3 
 \end{array}
 \right) = 
 \displaystyle 
\left( 
 \begin{array}{c|c}
 Q_{11} & R_{11} \\
 \hline 
 \overline{R}_{11} & \overline{Q}_{11} 
 \end{array}
 \right)
 \, .
 $$	
 
 The generalized polydiagonal  subspace 
 $$\Delta_{\footnotesize{\mathcal{P}}}=\{ x \in \R^4: \ x_1 = x_2=-x_3=-x_4 \}
 $$
  is left invariant by the adjacency matrix $W_{G}$ but not by the Laplacian matrix
 $$
L_{G} =
	\left( 
 \begin{array}{cc|cc}
3 & -1 & -1& -1  \\
-1 & 1 & 0& 0  \\
\hline 
0 & 0 & -3& 3  \\
-4 & -2 & -5& 11 
 \end{array}
 \right) 
 \, .
 $$	
\hfill $\Diamond$
\end{exam}

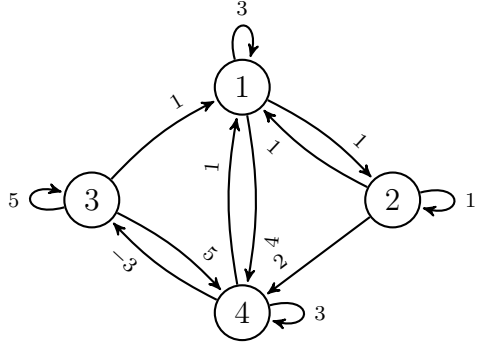
\begin{figure}[h!]
\begin{center}
\begin{tikzpicture}[->,>=stealth',shorten >=1pt,auto, node distance=1.5cm, thick,node/.style={circle,draw}]
                           \node[node]	(1) at (-5cm, 4cm)  [fill=white!60] {$1$};
			\node[node]	(4) at (-5cm, 1cm)  [fill=white!20] {$4$};
			\node[node]	(3) at (-7cm, 2.5cm) [fill=white] {$3$};
			\node[node]	(2) at (-3cm, 2.5cm)  [fill=white] {$2$};

                                                   \path
                                                   (1) edge [loop above] node {{\tiny $3$}} (1)   
				(1) edge [bend left=10, thick] node [pos=0.75, sloped, above] {{\tiny $1$}} (2)  
				(1) edge[bend left=10, thick] node [pos=0.75, sloped, below]  {{\tiny $4$}} (4) 
				 (2) edge [loop right] node {{\tiny $1$}} (2)                                                          
				(2) edge [bend left=10, thick] node  [pos=0.75, sloped, below]  {{\tiny $1$}} (1)
                                  (2) [->] edge node  [pos=0.75, sloped, above]  {{\tiny $2$}}   (4)  
                                  (3) edge [loop left] node {{\tiny $5$}} (3)
				(3) edge[bend left=10, thick]   node  [pos=0.75, sloped, above] {{\tiny $1$}} (1)
				(3) edge [bend left=10, thick] node [pos=0.75, sloped, above] {{\tiny $5$}} (4)
				(4) edge [bend left=10, thick]  node  [pos=0.7, sloped, above] {{\tiny $1$}} (1)
		                 (4) edge [loop right] node {{\tiny $3$}} (4)	
   			(4) edge [bend left=10, thick]  node  [pos=0.75, sloped, below]  {{\tiny $-3$}} (3);    
		\end{tikzpicture}	

		\end{center}
		\caption{The four-cell weighted network $G$ in Example~\ref{ex:notequal}.}
		\label{fig:notequal}
	\end{figure}

In the next result, we caracterize, for a weighted network, the generalized polydiagonals that are invariant under its adjacency matrix (resp. Laplacian matrix).

\begin{prop}\label{thm:mainLaplacian}
Let $G$ be a weighted network with set of cells $C = \{ 1, \ldots, n\}$ and $\Delta_{\mathcal{P}}$ a non-null generalized polydiagonal subspace of $\R^n$ where  $\mathcal{P}$ is a tagged partition $\{  P_1, P_2, \ldots, P_p, \overline{P}_1, \overline{P}_2, \ldots, \overline{P}_q, P_0 \}$ of $C$.  Consider an enumeration of  $C$ adapted to the partition $\mathcal{P}$. The adjacency matrix $W_G$ (resp. the  Laplacian matrix $L_G$) of $G$ leaves invariant the generalized polydiagonal $\Delta_{\mathcal{P}}$  if and only if the  block structure  (\ref{eq:oddbf}) of $W_G$  (resp.  $L_G$) satisfies the following conditions: \\
\ \\
When $q>0$:
\begin{equation}\label{eq:equal_Lap}
{\small \begin{array}{l}
\left\{
\begin{array}{ll}
{\mathrm rs}\left(Q_{ij}\right) - {\mathrm rs}\left( R_{ij}\right),\  {\mathrm rs}\left(\overline{Q}_{ij}\right) - {\mathrm rs}\left( \overline{R}_{ij}\right)  \mbox{ are regular of the same valency} & \left( 1 \leq i, j \leq q \right);\\
{\mathrm rs}\left(Q_{ij}\right), \ -{\mathrm rs}\left( \overline{R}_{ij}\right)  \mbox{ are regular of the same valency} & \left( 1 \leq i \leq q;\   q+1 \leq j \leq p \right); \\
{\mathrm rs}\left(Q_{ij}\right) - {\mathrm rs}\left( R_{ij}\right)  \mbox{ is regular} &\left(q+1 \leq  i \leq p,\  1 \leq  j \leq  q\right);\\
Q_{ij} \mbox{ is regular} & \left(q+1 \leq i, j \leq p \right); \\
\mbox{If } r=1 \mbox{ then }  {\mathrm rs}\left(Z_{0j}\right)  = {\mathrm rs}\left(\overline{Z}_{0j}\right) & \left( 1 \leq  j \leq q\right); \\
\qquad \qquad  \qquad {\mathrm rs}\left(Z_{0j}\right)  =  0 & \left( q+1 \leq  j \leq p \right).
\end{array}
\right.
\end{array}}
\end{equation}
\ \\
When $q=0$:
\begin{equation}\label{eq:equal_Lap_Z}      
\begin{array}{l}
\left\{
\begin{array}{ll}
Q_{ij} \mbox{ is regular} & \left( 1 \leq i,j \leq p \right);\\
\mbox{If } r=1 \mbox{ then } {\mathrm rs}\left(Z_{0j}\right)  =  0 & \left( 1 \leq j \leq p \right).
\end{array}
\right.
\end{array}
\end{equation}
\end{prop}

\begin{proof} 
Assume the tagged partition $\mathcal{P}$ has $q>0$. Denote by $X_i$, for $1\leq i\leq p$ (resp. $-X_i$, for $1\leq i \leq q$), the coordinates corresponding to the cells in part $P_i$ (resp. $\overline{P}_i$). Applying the matrix  $W_G$  (resp.  $L_G$) with block structure  (\ref{eq:oddbf}) to the column vector 
$X = \left(X_1, \ldots, X_q, X_{q+1}, \ldots, X_p, -X_1, \ldots, -X_q, {\bf 0} \right) \in \Delta_{\mathcal{P}}$, where cell coordinates corresponding to the zero part $P_0$ are set to zero (in case $r=1$), we obtain a column vector and corresponding properties in order to belong to $\Delta_{\mathcal{P}}$:\\
(i)  The components corresponding to the cells in the $q$ parts $P_1, \ldots, P_q$ are the entries of the column vector
$$
\begin{array}{l}
\left( 
 \begin{array}{cccc}
Q_{11} & Q_{12} &  \cdots &Q_{1q}  \\
 \vdots& \vdots & \cdots &\vdots  \\
Q_{q1} & Q_{q2} & \cdots &Q_{qq} 
\end{array}
\right)
\left( 
\begin{array}{c}
X_1\\
\vdots \\
X_q
\end{array}
\right) 
+  
\left( 
 \begin{array}{cccc}
R_{11} & R_{12} &  \cdots &R_{1q} \\
 \vdots& \vdots & \cdots &\vdots  \\
R_{q1} & R_{q2} & \cdots &R_{qq} 
\end{array}
\right)
\left( 
\begin{array}{c}
-X_1\\
\vdots \\
-X_q
\end{array}
\right) 
\\
 \\
+ 
\left( 
 \begin{array}{cccc}
Q_{1,q+1} & Q_{1,q+2} &  \cdots &Q_{1p} \\
 \vdots& \vdots & \cdots &\vdots  \\
Q_{q,q+1} & Q_{q,q+2} & \cdots &Q_{qp} 
\end{array}
\right)
\left( 
\begin{array}{c}
X_{q+1}\\
\vdots \\
X_p
\end{array}
\right); 
\end{array}
$$
\noindent The components corresponding to the cells in the counterparts $\overline{P}_1, \ldots, \overline{P}_q$ are the entries of the column vector: 
$$
\begin{array}{l}
\left( 
 \begin{array}{cccc}
\overline{R}_{11} & \overline{R}_{12} &  \cdots &\overline{R}_{1q} \\
 \vdots& \vdots & \cdots &\vdots  \\
\overline{R}_{q1}  & \overline{R}_{q2} & \cdots &\overline{R}_{qq} 
\end{array}
\right)
\left( 
\begin{array}{c}
X_1\\
\vdots \\
X_q
\end{array}
\right) 
+
\left( 
 \begin{array}{cccc}
 \overline{Q}_{11} & \overline{Q}_{12} &  \cdots & \overline{Q}_{1q} \\
 \vdots& \vdots & \cdots &\vdots  \\
\overline{Q}_{q1} &  \overline{Q}_{q2} & \cdots &\overline{Q}_{qq} 
\end{array}
\right)
\left( 
\begin{array}{c}
-X_1\\
\vdots \\
-X_q
\end{array}
\right) \\
\\
+
\left( 
 \begin{array}{cccc}
\overline{R}_{1,q+1} & \overline{R}_{1,q+2} &  \cdots &\overline{R}_{1p} \\
 \vdots& \vdots & \cdots &\vdots  \\
\overline{R}_{q,q+1} & \overline{R}_{q,q+2} & \cdots &\overline{R}_{qp} 
\end{array}
\right)
\left( 
\begin{array}{c}
X_{q+1}\\
\vdots \\
X_p
\end{array}
\right);
\end{array}
$$
We have so that  ${\mathrm rs}\left( Q_{ij} \right)-  {\mathrm rs}\left( R_{ij} \right) = {\mathrm rs} \left( \overline{Q}_{ij} \right) - {\mathrm rs} \left( \overline{R}_{ij} \right)$, for $1 \leq i,j \leq q$. Similarly, ${\mathrm rs}\left(  Q_{i,j} \right) = 
-{\mathrm rs} \left(  \overline{R}_{i,j} \right)$, for $1 \leq i \leq q$ and $q+1 \leq j \leq p$. Moreover, all these column vectors (of the form ${\mathrm rs} (M)$) have constant entries, that is, they are regular. \\
\noindent (ii) The components corresponding to the cells in the parts $P_{q+1}, \ldots, P_p$ are the entries of the column vector: 
$$
\begin{array}{l}
\left( 
 \begin{array}{cccc}
Q_{q+1,1} & Q_{q+1,2} &  \cdots &Q_{q+1,q}  \\
 \vdots& \vdots & \cdots &\vdots  \\
Q_{p1} & Q_{p2} & \cdots &Q_{pq} 
\end{array}
\right)
\left( 
\begin{array}{c}
X_1\\
\vdots \\
X_q
\end{array}
\right) 
+
\left( 
 \begin{array}{cccc}
R_{q+1,1} & R_{q+1, 2} &  \cdots &R_{q+1,q} \\
 \vdots& \vdots & \cdots &\vdots  \\
R_{p1} & R_{p2} & \cdots &Q_{pq} 
\end{array}
\right)
\left( 
\begin{array}{c}
-X_1\\
\vdots \\
-X_q
\end{array}
\right) \\
\\
+ 
\left( 
 \begin{array}{cccc}
Q_{q+1,q+1} & Q_{q+1,q+2} &  \cdots &Q_{q+1p} \\
 \vdots& \vdots & \cdots &\vdots  \\
Q_{p,q+1} & Q_{p,q+2} & \cdots &Q_{pp} 
\end{array}
\right)
\left( 
\begin{array}{c}
X_{q+1}\\
\vdots \\
X_p
\end{array}
\right);
\end{array} 
$$
\noindent Thus ${\mathrm rs} \left( Q_{ij} \right) - {\mathrm rs} \left( R_{ij} \right)$ is regular, for $q+1 \leq i \leq p$ and $1 \leq j \leq q$. Also, $Q_{i,j}$ is regular for 
$q+1 \leq i,j  \leq p$. \\
\noindent (iii Finally, the components corresponding to the cells in the part $P_0$ in case $r=1$ are the entries of the column vector: 
$$
\begin{array}{l}
\left( 
 \begin{array}{cccc}
Z_{01}  & Z_{02} &  \cdots &Z_{0q} 
\end{array}
\right)
\left( 
\begin{array}{c}
X_1\\
\vdots \\
X_q
\end{array}
\right) 
+
\left( 
 \begin{array}{cccc}
  \overline{Z}_{01} & \overline{Z}_{02} &  \cdots & \overline{Z}_{0q} 
\end{array}
\right)
\left( 
\begin{array}{c}
-X_1\\
\vdots \\
-X_q
\end{array}
\right) \\
\\ 
+
\left( 
 \begin{array}{cccc}
Z_{0,q+1} & Z_{0,q+2} &  \cdots &Z_{0p} 
\end{array}
\right)
\left( 
\begin{array}{c}
X_{q+1}\\
\vdots \\
X_p
\end{array}
\right)\, .
\end{array}
$$
Thus,  ${\mathrm rs}\left(Z_{0j}\right)  = {\mathrm rs}\left(\overline{Z}_{0j}\right)$ for  $1 \leq  j \leq q$ and ${\mathrm rs}\left(Z_{0j}\right)  =  0$ for $q+1 \leq  j \leq p$.  
We have so that $\Delta_{\mathcal{P}}$ is left invariant under $W_G$ (resp. $L_G$) if and only if conditions (\ref{eq:equal_Lap}) are verified. 

Now,  for tagged partitions where $q=0$, that is, there are no counterparts, we obtain conditions (\ref{eq:equal_Lap_Z}), since in that case, applying the matrix  $W_G$  (resp.  $L_G$) with block structure  (\ref{eq:oddbf}) to the collumn vector $X = \left(X_1, \ldots, \ldots, X_p, {\bf 0} \right) \in \Delta_{\mathcal{P}}$, where in case $r=1$, as before, cells corresponding to $P_0$ are set to zero, we obtain a column vector where:\\
(i)  The components corresponding to the cells in the $p$ parts $P_1, \ldots, P_p$ are the entries of the column vector
$$
{\tiny 
\left( 
 \begin{array}{cccc}
Q_{11}  & Q_{12} &  \cdots &Q_{1p}  \\
 \vdots& \vdots & \cdots &\vdots  \\
Q_{p1} & Q_{p2} & \cdots &Q_{pp}
\end{array}
\right)
\left( 
\begin{array}{c}
X_1\\
\vdots \\
X_p
\end{array}
\right)}; 
$$
\noindent (ii) The components corresponding to the cells in the part $P_0$ in case $r=1$ are the entries of the column vector: 
$$
{\tiny \left( 
 \begin{array}{cccc}
Z_{01} & Z_{02} &  \cdots &Z_{0p} 
\end{array}
\right)
\left( 
\begin{array}{c}
X_1\\
\vdots \\
X_p
\end{array}
\right)}\, .
$$
\end{proof}

\begin{exam} 		
Let $G$ be the five-cell non-regular network with adjacency matrix 
$$
W_{G} =
	\left( 
 \begin{array}{c|cc|c|c}
0 & -\frac{3}{2} & -\frac{3}{2} & 1 & \frac{23}{10}  \\
\hline 
-2 & 0 & 1 & 1 & 1  \\
-1 & 1 & 0 & 2 & 0  \\
\hline 
2 & 3 & 0 & 1 & \frac{11}{10}  \\
\hline 
1 & 1 & -1 & 1 & -3  
 \end{array}
 \right) = 
 \displaystyle 
\left( 
 \begin{array}{c|c|c|c}
 Q_{11} & Q_{12} & R_{11} & Z_{10} \\
 \hline 
  Q_{21} & Q_{22} & R_{21} & Z_{20} \\
 \hline 
 \overline{R}_{11} & \overline{R}_{12} & \overline{Q}_{11} & \overline{Z}_{10} \\
  \hline 
 Z_{01} & Z_{02} &   \overline{Z}_{01} & Z_{00} 
 \end{array}
 \right)
 \, .
 $$	
 
 Consider the generalized polydiagonal  subspace 
 $$\Delta_{\footnotesize{\mathcal{P}}}=\{ x \in \R^5: \ x_1 =-x_4,\ x_2=x_3,\ x_5=0\}
 $$
 for the tagged partition 
 $$\mathcal{P} = \{  P_1 =\{1\}, P_2= \{2,3\}, \overline{P}_1 =\{4\}, P_0 =\{5\} \}
  $$
  of $C$. Thus, $p=2$, $q=1$ and $r=1$. Note that the enumeration of the network set of cells is adapted to $\mathcal{P}$ providing 
the above  block structure of  $W_G$.  We have that:\\
 \begin{equation*}
\begin{array}{l}
\left\{
\begin{array}{ll}
{\mathrm rs}\left(Q_{11}\right) - {\mathrm rs}\left( R_{11}\right) = -{\mathrm rs}\left( \overline{R}_{11}\right)+{\mathrm rs}\left(\overline{Q}_{11}\right) =(-1);\\
{\mathrm rs}\left(Q_{12}\right) = -{\mathrm rs}\left( \overline{R}_{12}\right) =(-3); \\
{\mathrm rs}\left(Q_{21}\right) - {\mathrm rs}\left( R_{21}\right) \mbox{ is regular of valency } -3;\\
Q_{22} \mbox{ is regular of valency } 1; \\
{\mathrm rs}\left(Z_{01}\right)  = {\mathrm rs}\left(\overline{Z}_{01}\right) = (1); \\
{\mathrm rs}\left(Z_{02}\right)  =  (0).
\end{array}
\right.
\end{array}
\end{equation*} 
It follows, from Proposition~\ref{thm:mainLaplacian} that  $\Delta_{\footnotesize{\mathcal{P}}}$ is left invariant by the adjacency matrix $W_{G}$.
\hfill $\Diamond$
\end{exam}

The next corollary gives a characterization of the generalized polydiagonals that are invariant under the Laplacian matrix of a weighted network but in terms of its adjacency matrix.

\begin{cor}\label{cor:mainLaplacian}
Let $G$ be a weighted network with set of cells $C = \{ 1, \ldots, n\}$ and $\Delta_{\mathcal{P}}$ a generalized polydiagonal subspace of $\R^n$ for a tagged partition $\mathcal{P}$ of $C$. 
Consider an enumeration of  $C$ adapted to the partition $\mathcal{P}$. The Laplacian matrix $L_G$ leaves invariant the generalized polydiagonal $\Delta_{\mathcal{P}}$  if and only if the block structure  (\ref{eq:oddbf}) of  the adjacency matrix $W_G$ of $G$ satisfies the following conditions: \\

\noindent The matrices $Q_{ij},\, R_{ij},\, Z_{i0}$ and  $\overline{Q}_{ij}, \, \overline{R}_{ij},\ \overline{Z}_{i0}$ are related in the following way:\\
\noindent  When $q > 0$,\\  
\noindent (i) For each $i=1, \ldots, q$, and for $j=1, \ldots, q;\,  j\not=i$ the column vectors 
{\small $$ 
\left\{ \begin{array}{lr}
\displaystyle \sum_{j=1, j\not= i}^{p} {\mathrm rs}\left(Q_{ij}\right) + \sum_{j=1, j\not= i}^{q}   {\mathrm rs}\left(R_{ij}\right) +  2  {\mathrm rs}\left(R_{ii}\right)  + {\mathrm rs}\left( Z_{i0} \right) 
 & \\
& \mbox{ are regular of the same valency $r_i$;}\\

\displaystyle \sum_{j=1, j\not= i}^{q} {\mathrm rs}\left(\overline{Q}_{ij}\right) +    \sum_{j=1, j\not= i}^{p}    {\mathrm rs}\left(\overline{R}_{ij}\right) + 2  {\mathrm rs}\left(\overline{R}_{ii}\right)  
+ {\mathrm rs}\left( \overline{Z}_{i0} \right) &  
\end{array}\right. 
$$}
\noindent $
{\small 
\left\{
\begin{array}{lr}
\\
 - {\mathrm rs}(Q_{ij}) +  {\mathrm rs}(R_{ij}) & \\
  & \mbox{ are regular of the same valency $q_{ij}$}. \\
{\mathrm rs}\left(\overline{R}_{ij}\right) - {\mathrm rs}\left(\overline{Q}_{ij}\right) & 
\end{array}\right.}
$
\\
\\
  \noindent (ii)  For each $i \in \{1,\ldots,q\}$  and $j \in \{q+1,\ldots,p\}$,  \\
  \\
  \noindent 
  $ 
  \left\{
  \begin{array}{lr}
  - {\mathrm rs}\left(Q_{ij}\right)  & \\
  & \mbox{ are regular of the same valency } q_{ij}.\\
  {\mathrm rs}\left( \overline{R}_{ij}\right) & 
  \end{array}\right.
  $\\
  \\
  \noindent For all $q \in \NN_0$,  \\
 (iii) For each $i \in \{q+1,\ldots,p\}$ and $j \in \{1,\ldots,q\}$, 
 $$
 \begin{array}{l}
 \displaystyle \sum_{j=1, j\not= i}^{p} {\mathrm rs}\left(Q_{ij}\right) + \sum_{j=1}^{q}   {\mathrm rs}\left(R_{ij}\right) + {\mathrm rs}\left( Z_{i0} \right) \mbox{ is regular of valency $r_i$;}\\
 \\
  - {\mathrm rs}\left(Q_{ij}\right) + {\mathrm rs}\left( R_{ij}\right)  \mbox{ is regular of valency } q_{ij}.
  \end{array}
  $$
 
\noindent (iv) For each $i \in \{q+1,\ldots,p\}$ and $ j \in \{q+1,\ldots,p\},\ j\not=i$,   \\
  $$-Q_{ij} \mbox{ is regular of valency $q_{ij}$}
  \, .$$
    
\noindent The matrices $Z_{0j}$ and $\overline{Z}_{0j}$ satisfy:\\
(v) If $q > 0$, for each $j \in \{1,\ldots,q\}$,  we have: \\
$$
{\mathrm rs}(Z_{0j})  = {\mathrm rs}\left(\overline{Z}_{0j}\right)\, . \
$$

\noindent (vi) For all $q \in \NN_0$, for each $j \in \{q+1,\ldots,p\}$, \\
  $$
  Z_{0j}  \mbox{ is regular of valency zero.}
  $$
\end{cor}
\begin{proof}
 Let $G$ be a weighted network with set of cells $C = \{ 1, \ldots, n\}$, adjacency matrix $W_G$ and Laplacian matrix $L_G = D_G - W_G$. 
Given a generalized polydiagonal $\Delta_{\mathcal{P}}$, we consider the associated tagged partition $\mathcal{P}$ of $C$ determined by $p,q,r$ where $0 \leq q \leq p \leq n$ and $r \in \{0,1\}$. Denote the parts of $\mathcal{P}$ by $P_1, P_2, \ldots, P_p,$  the counterparts by $\overline{P}_1, \overline{P}_2, \ldots, \overline{P}_q$ and the zero part by $P_0$. Note that if $q=0$, then there are no counterparts and if $r=0$ then there is no zero part. Consider an enumeration of $C$ adapted to the partition $\mathcal{P}$ providing a block structure (\ref{eq:oddbf}) for $W_G$ and corresponding block structures for $L_G$ and $D_G$. For the blocks in $L_G$ and $D_G$ we use superscripts, respectively, $L$ and $D$. By Proposition~\ref{thm:mainLaplacian}, the generalized polydiagonal  $\Delta_{\mathcal{P}}$ is left invariant by the Laplacian matrix $L_G$  if and only if conditions (\ref{eq:equal_Lap}) and (\ref{eq:equal_Lap_Z}) are verified for the blocks $Q^L_{ij}$, $R^L_{ij}$, $Z^L_{0j}$, $\overline{Q}^L_{ij}$, $\overline{R}^L_{ij}$ and $\overline{Z}^L_{0j}$ of $L_G$.

For $i=j$, we have
$$
Q^L_{ii} = Q^D_{ii} - Q_{ii}, \quad R^L_{ii} =  - R_{ii}, \quad \overline{Q}^L_{ii} = \overline{Q}^D_{ii} - \overline{Q}_{ii}, \quad \overline{R}^L_{ii} =  - \overline{R}_{ii}.
$$
Thus, for $i=j$, the first condition in (\ref{eq:equal_Lap}) is equivalent to
$$
{\mathrm rs}\left(Q^D_{ii}\right) - {\mathrm rs}\left(Q_{ii}\right) + {\mathrm rs}\left(R_{ii}\right) = {\mathrm rs}\left(\overline{R}_{ii}\right) + {\mathrm rs}\left(\overline{Q}^D_{ii}\right)-{\mathrm rs}\left(\overline{Q}_{ii}\right),
$$
where the left and right columns of this equality are regular. Given that, 
{\tiny 
$$
{\mathrm rs}\left(Q^D_{ii}\right) = \sum_{j=1}^p {\mathrm rs}\left(Q_{ij}\right) +   \sum_{j=1}^q {\mathrm rs}\left(R_{ij}\right) 
+  {\mathrm rs}\left(Z_{i0}\right), \quad {\mathrm rs}\left(\overline{Q}^D_{ii}\right) = 
\sum_{j=1}^q {\mathrm rs}\left(\overline{Q}_{ij}\right) +  \sum_{j=1}^p  {\mathrm rs}\left(\overline{R}_{ij}\right) 
+  {\mathrm rs}\left(\overline{Z}_{i0}\right),
$$}
if $1\leq i \leq q$, the above equality simplifies to
{\tiny 
$$
 \sum_{j=1, j\not= i}^{p} {\mathrm rs}\left(Q_{ij}\right) +   \sum_{j=1, j\not= i}^{q} {\mathrm rs}\left(R_{ij}\right)  + 2 {\mathrm rs}\left(R_{ii}\right) +  {\mathrm rs}\left(Z_{i0}\right) =   
  \sum_{j=1, j\not= i}^{q} {\mathrm rs}\left(\overline{Q}_{ij}\right) +   \sum_{j=1, j\not= i}^{p} {\mathrm rs}\left(\overline{R}_{ij}\right) + 2 {\mathrm rs}\left(\overline{R}_{ii}\right) +  {\mathrm rs}\left(\overline{Z}_{i0}\right)\, .
$$
}
 Thus we obtain the first equality in condition (i) of Corollary~\ref{cor:mainLaplacian}. Moreover, for $i\ne j$, we have
$$
Q^L_{ij} = - Q_{ij}, \quad R^L_{ij} =  - R_{ij}, \quad \overline{Q}^L_{ij} =  - \overline{Q}_{ij}, \quad \overline{R}^L_{ij} =  - \overline{R}_{ij}.
$$
Thus, for $i\ne j$, the first condition in (\ref{eq:equal_Lap}) is equivalent to
$$
 - {\mathrm rs}\left(Q_{ij}\right) + {\mathrm rs}\left(R_{ij}\right) = {\mathrm rs}\left(\overline{R}_{ij}\right) -{\mathrm rs}\left(\overline{Q}_{ij}\right),
$$
where the left and right columns of this equality are regular (of the same valency).
Thus we obtain the second equality in condition (i) of Corollary~\ref{cor:mainLaplacian}. 

Finally, the remaining conditions in  (\ref{eq:equal_Lap}) and (\ref{eq:equal_Lap_Z}) of Proposition~\ref{thm:mainLaplacian} are equivalent to (ii)-(vi) of Corollary~\ref{cor:mainLaplacian}. 
 
\end{proof}

In Proposition~\ref{thm:mainLaplacian}, if we restrict to polydiagonal subspaces we get the following.

\begin{cor} \label{thm:mainpart_exo} 
 Let $G$ be a weighted network with set of cells $C = \{ 1, \ldots, n\}$ and $\Delta_{\mathcal{P}}$ a polydiagonal subspace of $\R^n$. Consider the associated (standard) partition $\mathcal{P}$ of $C$ with $p>0$ parts, say  $P_1, P_2, \ldots, P_p$, and take an enumeration of  $C$ adapted to the partition $\mathcal{P}$. \\
(i) The adjacency matrix $W_G$ of $G$ leaves invariant the  polydiagonal $\Delta_{\mathcal{P}}$  if and only if in the  block structure  (\ref{eq:oddbf}) of $W_G$ 
 every matrix $Q_{ij}$ is regular, for $i,j \in \{1, \ldots, p\}$. \\
 (ii) The Laplacian matrix $L_G$ of $G$ leaves invariant the  polydiagonal $\Delta_{\mathcal{P}}$  if and only if in the  block structure  (\ref{eq:oddbf}) of $W_G$ 
 every matrix $Q_{ij}$, with $i \ne j$, is regular, for $i,j \in \{1, \ldots, q\}$.
 \end{cor}

\begin{proof}The statement (i) follows directly from Proposition~\ref{thm:mainLaplacian}, considering that the partition $\mathcal{P}$ is standard, i.e., it is determined by $p>0$ and $q=r=0$, that is, there are no counterparts neither the zero part. To conclude (ii), note that, applying (i) to $L_G$, as $L_G$ is regular of valency zero, it follows that every matrix $Q^L_{ij}$ is regular, for all $i,j \in \{1, \ldots, q\}$ if and only if $Q^L_{ij} = -Q_{ij}$ is regular, for all $i,j \in \{1, \ldots, q\}$ with $i\not=j$.
\end{proof}

 \subsection*{Exo-balanced  and balanced standard partitions}
 
 We recall the concepts of balanced and exo-balanced (standard) partitions for weighted networks, as in Aguiar and Dias~\cite{AD18}. The concept of balanced partition was first introduced in the formalism of Golubitsky, Stewart and collaborators, where the network connections have associated nonnegative integer values and extended to the weighted formalism in Aguiar and Dias~\cite{AD18}.
 
 \begin{Def} \normalfont   \label{def:balanced}
Let $G$ be a weighted network with set of cells $C$ and a standard partition $\mathcal{P} = \left\{ P_1, P_2, \ldots, P_p \right\}$ of $C$.\\
(i) The partition $\mathcal{P}$  is said {\it exo-balanced} when the corresponding polydiagonal subspace $\Delta_{\mathcal{P}}$ is left invariant by the Laplacian matrix $L_G$ of $G$, that is, when 
$$v_{P} (i) = v_{P} (i')$$ 
for $[i]=[i']$ and for all $P \in \mathcal{P} \setminus \{ [i] \}$. We denote by $\mathcal{P}_{G,exo}$ the {\it set of exo-balanced (standard) partitions} of $G$. \\
(ii) The partition $\mathcal{P}$  is said {\it balanced} when the corresponding polydiagonal subspace $\Delta_{\mathcal{P}}$ is left invariant by the adjacency matrix $W_G$ of $G$, that is, when 
$$v_{P} (i) = v_{P} (i')$$ 
for $[i]=[i']$ and for all $P \in \mathcal{P}$. We denote by $\mathcal{P}_{G,bal}$ the {\it set of balanced (standard) partitions} of $G$.
\hfill $\Diamond$
\end{Def}

\begin{rem}\normalfont Recalling Remark~\ref{rmk_block_match} and using the notation of Corollary~\ref{thm:mainpart_exo} for the adjacency matrix $W_G$ of a weighted network $G$, we have that a standard partition  $\mathcal{P}$ is exo-balanced when every matrix $Q_{ij}$, with $i \ne j$, is regular, for $i,j \in \{1, \ldots, p\}$. Also, a standard partition  $\mathcal{P}$ is balanced when every matrix $Q_{ij}$ is regular, for all $i,j \in \{1, \ldots, p\}$. 
\hfill $\Diamond$
\end{rem}

\begin{rem}\normalfont \label{rmk:bal_G_L}
Let $G$ be a weighted network with weighted adjacency matrix $W_G$ and consider the network $G_L$ associated with the Laplacian matrix $L_G$ of $G$. We have that $\mathcal{P}_{G_L,bal} = \mathcal{P}_{G,exo}$.
\hfill $\Diamond$
\end{rem}

It follows from Proposition~\ref{prop:subset} the following relation between the set of the balanced and the set of the exo-balanced partitions of a network $G$, which is a generalization to weighted networks of Proposition 3.15 in Neuberger {\it et al.} \cite{NSS19}. 

\begin{cor} \label{prop:reg}
Let $G$ be a weighted network. We have,\\
(i) $\mathcal{P}_{G,bal} \subseteq \mathcal{P}_{G,exo}$; \\
(ii) $\mathcal{P}_{G,bal} = \mathcal{P}_{G,exo}$ if and only if $G$ is regular;\\
\end{cor}

It follows from Corollary~\ref{prop:reg} that, $\mathcal{P}_{G,exo} \setminus \mathcal{P}_{G,bal} \not= \emptyset$, if a network $G$ is not regular. We have then the following definition.

\begin{Def} \normalfont   
A standard partition in $\mathcal{P}_{G,exo} \setminus \mathcal{P}_{G,bal}$ is said to be {\it strictly exo-balanced}.
 \hfill $\Diamond$ 
\end{Def}

A standard partition $\mathcal{P}$ is so strict exo-balanced if and only if the subspace $\Delta_P$ is left invariant by the Laplacian matrix $L_G$ of $G$ but not by its adjacency matrix $W_G$.

\begin{exam}
Any weighted network has at least the exo-balanced standard partition corresponding to the trivial partition with only one part, the network set of cells. If the network is not regular, then the trivial (standard) partition is strictly exo-balanced.
\hfill $\Diamond$
\end{exam}

\begin{rem}\normalfont \label{rmk:bal_exo_L}
Let $G$ be a weighted network with weighted adjacency matrix $W_G$ and consider the network $G_L$ associated with the Laplacian matrix $L_G$ of $G$. 
By Remark~\ref{rmk:bal_G_L}, the set of strict exo-balanced partitions for $G$ is formed by the balanced partitions of $G_L$ which are not balanced for $G$, that is, $\mathcal{P}_{G_L,bal} \setminus \mathcal{P}_{G,bal}$.
\hfill $\Diamond$ 
\end{rem}

\subsubsection*{Quotient networks for balanced and exo-balanced standard partitions}

We recall the concept of quotient network  on balanced standard partitions of networks in the formalism of Golubitsky, Stewart and collaborators, where the network connections have associated nonnegative integer values. These concepts  are also valid and extend trivially to the weighted formalism as stated  in Aguiar and Dias~\cite{AD18}.

Following Section 2 of  \cite{AD18}, given a balanced standard partition $\mathcal{P}$ of the set of cells of a weighted network $G$, the  associated {\it quotient network} $G_{\footnotesize{\mathcal{P}}}$ is the weighted  network defined in the following way: each cell in $G_{\footnotesize{\mathcal{P}}}$  corresponds to a part in $\mathcal{P}$; denoting by $[i]$ the part in $\mathcal{P}$ containing $i$, there is an edge from $[j]$ directed to $[i]$ if and only if there exists in $G$ an edge directed from $j'$ to $i'$, with $j' \in [j]$ and $i' \in [i]$. Moreover, the weight of  the edge directed from $[j]$ to $[i]$ is $v_{[j]} (i) $. That is, if the balanced partition  $\mathcal{P}$ has $p$ parts and $W_{G_{\footnotesize{\mathcal{P}}}}=[q_{[i],[j]}]_{p \times p}$ is the weighted adjacency matrix of the quotient network $G_{\footnotesize{\mathcal{P}}}$, we have $q_{[i],[j]} =  v_{[j]} (i)$. The network $G$ is said to be a {\it lift} of $G_{\footnotesize{\mathcal{P}}}$ by a balanced partition. 

From Definition~\ref{def:balanced} (ii) and Corollary~\ref{thm:mainpart_exo}  (i), we have:

\begin{prop}
\label{prop:quotient_bal}
Let $G$ be a weighted network and $W_G$ the corresponding weighted adjacency matrix. Let $\mathcal{P}$ be a balanced standard 
 partition of the set of cells of $G$ with parts $P_1, \ldots, P_p$ and assume an enumeration of  the network set of cells adapted to the partition $\mathcal{P}$ providing a  block structure (\ref{eq:oddbf}). The adjacency matrix of the quotient network $G_{\footnotesize{\mathcal{P}}}$ is the $p \times p$ matrix $W_{G_{\footnotesize{\mathcal{P}}}} = [q_{ij}]$ with $q_{ij} = v_{Q_{ij}}$.
\end{prop}

Given a strict exo-balanced standard partition $\mathcal{P}$ on the set of cells of a weighted network $G$, we have that  $\mathcal{P}$ is balanced for $G_{-L}$ by Remark~\ref{rmk:bal_G_L}. It follows that we can take the quotient network of $G_{-L}$ by $\mathcal{P}$,  as defined above, where the $ij$ entry is $v_{[j]} (i)$ if $i \not= j$. We define:

\begin{Def} \normalfont   \label{def:quo:_exo}
Let $G$ be a coupled cell network with set of cells $C$.  Let $W=[w_{i,j}]_{n\times n}$ be the weighted adjacency matrix of $G$, $\mathcal{P}$ a strict exo-balanced standard partition on  $C$ with $p$ parts and $Q_{-L}$ the weighted quotient network of $G_{-L}$ by the balanced partition $\mathcal{P}$. Then, we define the {\it quotient of $G$ by $\mathcal{P}$} to be the network $Q_{\mathcal{P}}$ with adjacency matrix $[q_{i j}]_{1\leq i,j \leq p}$ obtained from the adjacency matrix of $Q_{-L}$ by setting to zero the diagonal entries: 
 $$
 q_{i j} = \left\{
 \begin{array}{ll}
 0, & \mbox{ if } [i] = [j] \\
 \\
 v_{[j]} (i), &  \mbox{ if } [i] \ne [j]
 \end{array}
 \right. \, .
 $$
\hfill $\Diamond$ 
\end{Def}

\begin{exam}
In Figure~\ref{f:um} we show a six-cell network $G$ and the standard partition $\mathcal{P} = \left\{ [1] = \{1,2,3\},\, [4] = \{4,5\},\, [6] = \{6\} \right\}$ of its set of cells. Note that $\mathcal{P}$ is not balanced for $G$ but it is balanced for $H \equiv G_{-L}$. Thus $\mathcal{P}$ is strictly exo-balanced for $G$.  On the right of Figure~\ref{f:um} we show the corresponding quotient networks as defined above. 

\hfill $\Diamond$ 
\end{exam}

\begin{figure}[!h]
\begin{tabular}{cc}
\begin{tikzpicture}
 [scale=.15,auto=left, node distance=1.5cm]
 \node[fill=magenta,style={circle,draw}] (n1) at (4,0) {\small{1}};
 \node[fill=magenta,style={circle,draw}] (n2) at (4,-6) {\small{2}};
  \node[fill=magenta,style={circle,draw}] (n3) at (14,0) {\small{3}};
 \node[fill=white,style={circle,draw}] (n4) at (14,-6)  {\small{4}};
 \node[fill=white,style={circle,draw}] (n5) at (24,0)  {\small{5}};
 \node[fill=green,style={circle,draw}] (n6) at (24,-6)  {\small{6}};
 \draw[->, thick] (n1) edge[thick]  node  [near end, above=0.1pt] {{\tiny $1$}} (n3); 
 \draw[->, thick] (n2) edge[thick]  node  [above=0.1pt] {{\tiny $2$}} (n3); 
 \draw[->, thick] (n2) edge[thick] node  [near end, above=0.1pt] {{\tiny $1$}} (n4); 
\draw[->, thick] (n3) edge[thick] node  [near end, above=1pt] {{\tiny $1$}} (n5); 
\draw[->, thick] (n4) edge[thick] node  [near end, above=1pt]  {{\tiny $2$}}  (n6); 
\end{tikzpicture}  & 
\begin{tikzpicture}
 [scale=.15,auto=left, node distance=1.5cm]
 \node[fill=magenta,style={circle,draw}] (n2) at (4,-6) {$\small{[1]}$};
 \node[fill=white,style={circle,draw}] (n4) at (14,-6)  {$\small{[4]}$};
 \node[fill=green,style={circle,draw}] (n6) at (24,-6)  {$\small{[6]}$};
 \draw[->, thick] (n2) edge[thick] node  [near end, above=0.1pt] {{\tiny $1$}} (n4); 
\draw[->, thick] (n4) edge[thick] node  [near end, above=1pt]  {{\tiny $2$}}  (n6); 
\end{tikzpicture}
\\
$G$ & $Q_{{\mathcal P}}$ \\
\\
\begin{tikzpicture}
 [scale=.15,auto=left, node distance=1.5cm]
 \node[fill=magenta,style={circle,draw}] (n1) at (4,0) {\small{1}};
 \node[fill=magenta,style={circle,draw}] (n2) at (4,-6) {\small{2}};
  \node[fill=magenta,style={circle,draw}] (n3) at (14,0) {\small{3}};
 \node[fill=white,style={circle,draw}] (n4) at (14,-6)  {\small{4}};
 \node[fill=white,style={circle,draw}] (n5) at (24,0)  {\small{5}};
 \node[fill=green,style={circle,draw}] (n6) at (24,-6)  {\small{6}};
 \draw[->, thick] (n1) edge[thick]  node  [near end, above=0.1pt] {{\tiny $1$}} (n3); 
 \draw[->, thick] (n2) edge[thick]  node  [above=0.1pt] {{\tiny $2$}} (n3); 
 \draw[->, thick] (n2) edge[thick] node  [near end, above=0.1pt] {{\tiny $1$}} (n4); 
\draw[->, thick] (n3) edge[thick] node  [near end, above=1pt] {{\tiny$1$}} (n5); 
\draw[->, thick] (n4) edge[thick] node  [near end, above=1pt]  {{\tiny $2$}}  (n6); 
\draw[->, thick] (n3) edge[loop, out=20, in=70, looseness=5] node [above left, pos=.45]  {{\tiny $-3$}}  (n3); 
\draw[->, thick] (n4) edge[loop, out=210, in=260, looseness=5] node [below right, pos=.45]  {{\tiny $-1$}}  (n4); 
\draw[->, thick] (n5) edge[loop, out=20, in=70, looseness=5] node [above left, pos=.45]  {{\tiny $-1$}}  (n5); 
\draw[->, thick] (n6) edge[loop, out=210, in=260, looseness=5] node [below right, pos=.45]  {{\tiny $-2$}}  (n6); 
\end{tikzpicture}  &
\begin{tikzpicture}
 [scale=.15,auto=left, node distance=1.5cm]
 \node[fill=magenta,style={circle,draw}] (n2) at (4,-6) {$\small{[1]}$};
 \node[fill=white,style={circle,draw}] (n4) at (14,-6)  {$\small{[4]}$};
 \node[fill=green,style={circle,draw}] (n6) at (24,-6)  {$\small{[6]}$};
 \draw[->, thick] (n2) edge[thick] node  [near end, above=0.1pt] {{\tiny $1$}} (n4); 
\draw[->, thick] (n4) edge[thick] node  [near end, above=1pt]  {{\tiny $2$}}  (n6); 
\draw[->, thick] (n4) edge[loop, out=210, in=260, looseness=5] node [below right, pos=.45]  {{\tiny $-1$}}  (n4); 
\draw[->, thick] (n6) edge[loop, out=210, in=260, looseness=5] node [below right, pos=.45]  {{\tiny $-2$}}  (n6); 
\end{tikzpicture}  \\
 \\
$H \equiv  G_{-L}$ & $H_{{\mathcal P}}$ 
 \end{tabular}
\caption{Two six-cell networks $G$ and $G_{-L}$ and a partition $\mathcal{P} = \left\{ [1] = \{1,2,3\},\, [4] = \{4,5\},\, [6] = \{6\} \right\}$ of their sets of cells. (Top) The partition $\mathcal{P}$ is exo-balanced but not balanced for $G$. The network $Q_{{\mathcal P}}$ is the three-cell quotient network of $G$ by the exo-balanced partition $\mathcal{P}$. (Bottom) The partition $\mathcal{P}$ is balanced for $H \equiv  G_{-L}$. The network $H_{{\mathcal P}}$ is the quotient network of $H$ by  $\mathcal{P}$.}
\label{f:um}
\end{figure}
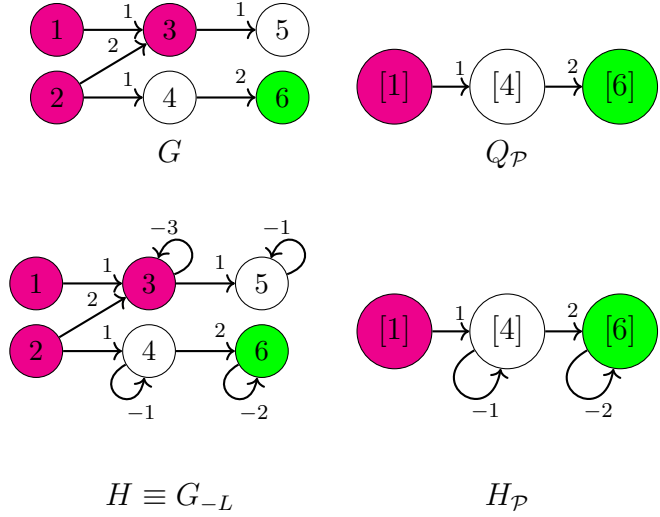

\subsection*{Linear-balanced, even-odd-balanced and odd-balanced tagged partitions} 

We define next, for general weighted networks, the concepts of linear-balanced, even-odd-balanced and odd-balanced partitions. We use here the terminology of linear-balanced and odd-balanced partitions in Definitions 4.17 and  4.6 of \cite{NSS19}, respectively,  for the class of undirected networks $G$. 

\begin{Def} \normalfont   \label{def:linear}
Let $G$ be a weighted network with set of cells $C$.  A 
non-standard
tagged partition $\mathcal{P} = \left\{ P_1, P_2, \ldots, P_p, \overline{P}_1, \overline{P}_2, \ldots, \overline{P}_q, P_0\right\}$ of $C$ is said {\it linear-balanced} (resp. {\it even-odd-balanced})  if the corresponding generalized polydiagonal subspace $\Delta_{\mathcal{P}}$ is left invariant by the Laplacian matrix $L_G$ (resp. adjacency matrix $W_G$) of $G$. We denote by $\mathcal{P}_{G,lin}$ the set of linear-balanced partitions of $G$ and by $\mathcal{P}_{G,eo}$  the set of even-odd-balanced partitions of $G$
\hfill $\Diamond$
\end{Def}

\begin{rem}\normalfont
By Proposition~\ref{prop:subset}, for a regular network $G$, we have $\mathcal{P}_{G,lin} = \mathcal{P}_{G,eo}$.

\hfill $\Diamond$
\end{rem}

\begin{Def} \normalfont   \label{def:odd}
Let $G$ be a weighted network with set of cells $C$. A 
non-standard
tagged partition $\mathcal{P} = \left\{ P_1, P_2, \ldots, P_p, \overline{P}_1, \overline{P}_2, \ldots, \overline{P}_q, P_0\right\}$ of $C$ is said {\it odd-balanced}  if, given an enumeration of  the network set of cells adapted to the partition $\mathcal{P}$ providing a  block structure (\ref{eq:oddbf}) for the adjacency matrix $W_G$, we have\\
(a)  All the blocks, excluding the blocks $Q_{ii}$, $\overline{Q}_{ii}$, $Z_{0j},\, \overline{Z}_{0j}$ and $Z_{00}$, are  regular.\\
(b)  If $q >0$, for $1 \le i,j \le q $, each pair of blocks of the type $Q_{ij},\, \overline{Q}_{ij}$, for $i\not=j$,  $R_{ij},\, \overline{R}_{ij}$ and $Z_{i0},\, \overline{Z}_{i0}$ have the same valency. \\
(c) If $q>0$ and $r=1$, the blocks $Z_{0j},\, \overline{Z}_{0j}$ satisfy ${\mathrm rs}(Z_{0j}) = {\mathrm rs}\left(\overline{Z}_{0j}\right)$ for  $j \in \{1,\ldots,q\}$. \\
(d) If $r=1$,  ${\mathrm rs}\left(Z_{0j}\right)  =  0$,  for $j \in \{q+1,\ldots,p\}$. \\
We denote by $\mathcal{P}_{G,odd}$ the set of odd-balanced partitions of $G$ 
\hfill $\Diamond$
\end{Def}

\begin{rem}\normalfont
In the definition of linear-balanced  and odd-balanced tagged partitions, the blocks $Q_{ii},\,\overline{Q}_{ii}$ for all $i$ and $Z_{00}$ have no restrictions.
\hfill $\Diamond$
\end{rem}

\begin{rem}\normalfont 
(i) Given a weighted network $G$, the conditions in Definition~\ref{def:odd} of odd-balanced tagged partition imply the conditions in Proposition~\ref{thm:mainLaplacian} for the corresponding generalized polydiagonal subspace $\Delta_{\mathcal{P}}$ to be left invariant by the Laplacian matrix $L_G$. Thus we have $\mathcal{P}_{G,odd} \subseteq \mathcal{P}_{G,lin}$. \\
(ii) A linear-balanced partition of a network set of cells does not have to be odd-balanced, as we show in Example~\ref{exs:linear}. \\
(iii) An odd-balanced partition may not be  even-odd-balanced and an even-odd-balanced partition may not be odd-balanced, as we show in Examples~\ref{ex:odd-eo}  and \ref{ex:eo-odd}, respectively.
\hfill $\Diamond$
\end{rem}

\begin{figure}[h!]
\begin{center}
{\tiny 
\begin{tabular}{cc}
\begin{tikzpicture}[->,>=stealth',shorten >=1pt,auto, node distance=1.5cm, thick,node/.style={circle,draw}]
                           \node[node]	(6) at (-6cm, 3cm)  [fill=black!70] {$6$};
			\node[node]	(3) at (-6cm, 1cm)  [fill=black!20] {$3$};
			\node[node]	(1) at (-7cm, 2cm) [fill=white] {$1$};
			\node[node]	(4) at (-5cm, 2cm)  [fill=black!50] {$4$};
			\node[node]	(5) at (-3.5cm, 2cm)  [fill=black!50] {$5$};
			\node[node]	(2) at (-2cm, 2cm) [fill=white] {$2$};
                                                   \path
                                                            
				(4) edge node {} (5)
				(5) edge node {} (4)
				(5) edge node {} (2)
				(2) edge node {} (5)
				(3) edge node {} (4)
				(4) edge node {} (3)
				(1) edge node {} (3)
				(3) edge node {} (1)
				(1) edge node {} (6)
				(6) edge node {} (1)
				(6) edge node {} (4)
				(4) edge node {} (6);    
		\end{tikzpicture}	\qquad & \qquad 
\begin{tikzpicture}[->,>=stealth',shorten >=1pt,auto, node distance=1.5cm, thick,node/.style={circle,draw}]
                           \node[node]	(6) at (-6cm, 3cm)  [fill=black!70] {$6$};
			\node[node]	(5) at (-6cm, 1cm)  [fill=black!70] {$5$};
			\node[node]	(1) at (-7cm, 2cm) [fill=white] {$1$};
			\node[node]	(3) at (-5cm, 2cm)  [fill=black!20] {$3$};
			\node[node]	(4) at (-3.5cm, 2cm)  [fill=black!20] {$4$};
			\node[node]	(2) at (-2cm, 2cm) [fill=white] {$2$};
                                                   \path
                                                            
				(4) edge node {} (2)
				(2) edge node {} (4)
				(3) edge node {} (4)
				(4) edge node {} (3)
				(5) edge node {} (3)
				(3) edge node {} (5)
				(6) edge node {} (3)
				(3) edge node {} (6)
				(1) edge node {} (6)
				(6) edge node {} (1)
				(1) edge node {} (5)
				(5) edge node {} (1);    
		\end{tikzpicture}	
		\end{tabular}}
		\end{center}
		\caption{A six-cell bidirectional network $G$. Two linear-balanced tagged partitions which are not odd-balanced.  
		(Left) $\mathcal{P} = \left\{ P_1 = \{1,2\},\, P_2 = \{ 3\},\, \overline{P}_1 = \{4,5\},\, \overline{P}_2 = \{ 6\} \right\}$;   
		(Right) $\mathcal{P} = \{ P_1 =\{1,2\}, \, \overline{P}_1 = \{ 3, 4\},\, P_0 =\{5,6\}\}$.}
		\label{f:2ndsixSwift}
	\end{figure}
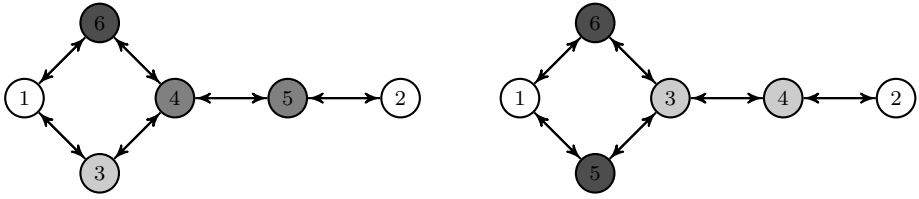

\begin{exams} \label{exs:linear} \normalfont 		
Take the two  isomorphic six-cell networks in Figure~\ref{f:2ndsixSwift} which correspond to the  six-cell  bidirectional network in Figure 9 of \cite{NSS19}.  \\
(i) Consider  the network on the left of  Figure~\ref{f:2ndsixSwift}  and take the tagged partition of its set of cells $\mathcal{P} = \{ P_1 =\{1,2\}, \, P_2 = \{ 3\},\, \overline{P}_1 =\{4,5\}, \, \overline{P}_2 = \{ 6\} \}$. The adjacency matrix has  block form: 
$$
W_{G} =
	\left( 
 \begin{array}{cc|c|cc|c}
0 & 0 & 1& 0 & 0& 1 \\
0 & 0 & 0& 0 & 1& 0 \\
\hline 
1 & 0 & 0& 1 & 0& 0 \\
\hline 
0 & 0 & 1& 0 & 1 & 1 \\
0 & 1 & 0& 1 & 0& 0 \\
\hline 
1 & 0 & 0& 1 & 0& 0 
 \end{array}
 \right) = 
 \displaystyle 
\left( 
 \begin{array}{cc|cc}
 Q_{11} & Q_{12} & R_{11} & R_{12}\\
 Q_{21} & Q_{22} & R_{21} & R_{22}\\
 \hline 
 \overline{R}_{11} & \overline{R}_{12} & \overline{Q}_{11} & \overline{Q}_{12} \\
 \overline{R}_{21} & \overline{R}_{22} & \overline{Q}_{21} & \overline{Q}_{22}
 \end{array}
 \right)
 \, .
 $$	
 We have that $\mathcal{P}$ is linear-balanced. By Corollary~\ref{cor:mainLaplacian}, 
 this follows from the equalities: 
 $$
 {\tiny 
 \begin{array}{l} 
 \left( 
\begin{array}{l}
2\\
 2
  \end{array}
  \right) 
  = 
    {\mathrm rs}(Q_{12}) +  {\mathrm rs}(R_{12})  + 
2  {\mathrm rs}(R_{11})  
= 
 {\mathrm rs}\left(\overline{Q}_{12}\right) +  {\mathrm rs}\left(\overline{R}_{12}\right)  + 
2  {\mathrm rs}\left(\overline{R}_{11}\right)  =   
\left( 
\begin{array}{l}
1\\
 0
  \end{array}
  \right) + 
  \left( 
\begin{array}{l}
1\\
 0
  \end{array}
  \right) +
 2 \left( 
\begin{array}{l}
0\\
 1
  \end{array}
  \right);
  \ \\
   \ \\
 \left( 
\begin{array}{l}
0\\
 0
  \end{array}
  \right) 
  = 
   - {\mathrm rs}(Q_{12}) +  {\mathrm rs}(R_{12}) = -  \left( 
\begin{array}{l}
1\\
 0
  \end{array}
  \right) + 
   \left( 
\begin{array}{l}
1\\
 0
  \end{array}
  \right) 
  =   
 {\mathrm rs}\left(\overline{R}_{12}\right) -  {\mathrm rs}\left(\overline{Q}_{12}\right) = 
  \left( 
\begin{array}{l}
1\\
 0
  \end{array}
  \right) - 
   \left( 
\begin{array}{l}
1\\
 0
  \end{array}
  \right); 
  \ \\
  \ \\
  \left( 
\begin{array}{l}
2
  \end{array}
  \right)   
  =   {\mathrm rs}(Q_{21}) +  {\mathrm rs}(R_{21})  + 
2  {\mathrm rs}(R_{22}) 
=  
 {\mathrm rs}\left(\overline{Q}_{21}\right) +  {\mathrm rs}\left(\overline{R}_{21}\right)  + 
2  {\mathrm rs}\left(\overline{R}_{22}\right) 
=  
\left( 
\begin{array}{l}
1
  \end{array}
  \right) + 
  \left( 
\begin{array}{l}
1
  \end{array}
  \right) +
 2 \left( 
\begin{array}{l}
0
  \end{array}
  \right);\\
  \ \\
  \ \\ 
\left( 
\begin{array}{l}
0
  \end{array}
  \right)  
  =  
  - {\mathrm rs}(Q_{21}) +  {\mathrm rs}(R_{21}) = -  \left( 
\begin{array}{l}
1
  \end{array}
  \right) + 
   \left( 
\begin{array}{l}
1
  \end{array}
  \right) 
   =  
 {\mathrm rs}\left(\overline{R}_{21}\right) -  {\mathrm rs}\left(\overline{Q}_{21}\right) = 
  \left( 
\begin{array}{l}
1
  \end{array}
  \right) - 
   \left( 
\begin{array}{l}
1
  \end{array}
  \right)\, .
  \end{array}
  }
  $$	
  The tagged partition $\mathcal{P}$ is not odd-balanced as, for example, the block $R_{11}$ is not regular.\\ 
\noindent (ii) Consider  the network on the right of  Figure~\ref{f:2ndsixSwift}  and take the tagged partition of its set of cells $\mathcal{P} = \{ P_1 =\{1,2\}, \, \overline{P}_1 = \{ 3, 4\},\, P_0 =\{5,6\}\}$.  
The adjacency matrix has  block form: 
$$
W_{G} =
	\left( 
 \begin{array}{cc|cc|cc}
0 & 0 & 0& 0 & 1 & 1 \\
0 & 0 & 0& 1 & 0& 0 \\
\hline 
0 & 0 & 0& 1 & 1& 1 \\
0 & 1 & 1 & 0 & 0& 0 \\
\hline 
1 & 0 & 1& 0 & 0 & 0 \\
1 & 0 & 1 &0 & 0& 0 
 \end{array}
 \right) = 
 \left( 
 \begin{array}{c|c|c}
 Q_{11}& R_{11} & Z_{10}\\
 \hline 
 \overline{R}_{11} & \overline{Q}_{11} & \overline{Z}_{10} \\
 \hline 
 Z_{01} & \overline{Z}_{01} & Z_{00}
 \end{array}
 \right)
 \, .
 $$	
 By Corollary~\ref{cor:mainLaplacian}, we have that $\mathcal{P}$ is linear-balanced  given the following equalities: 
 $$ 
 {\small 
 \begin{array}{l}
 \left( 
\begin{array}{l}
2\\
 2
  \end{array}
  \right) 
  =    
2  {\mathrm rs}(R_{11}) + {\mathrm rs}(Z_{10}) 
=   
2  {\mathrm rs}\left(\overline{R}_{11}\right) + {\mathrm rs}\left(\overline{Z}_{10}\right)  
=  
2\left( 
\begin{array}{l}
0\\
 1
  \end{array}
  \right) + 
  \left( 
\begin{array}{l}
2\\
 0
  \end{array}
  \right); \\ 
  {\mathrm rs}(Z_{01})  
  =   
  {\mathrm rs}\left(\overline{Z}_{01}\right) = \left( 
\begin{array}{l}
1\\
 1
  \end{array}
  \right) \, .
  \end{array} }
  $$	
   The tagged partition $\mathcal{P}$ is not odd-balanced as, for example, the block $R_{11}$ is not regular. 
  \hfill $\Diamond$ 
  \end{exams}

 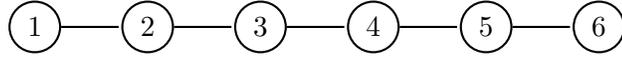
\begin{figure}[ht!]
\begin{center}
\vspace{-4mm}
\hspace{-4mm}
	{\small \begin{tikzpicture}	[-,>=stealth',shorten >=1pt,auto, node distance=1.5cm, thick,node/.style={circle,draw}]
	                 \node[node]	(1) at (-75mm, 4cm)  [fill=white] {$1$};
			\node[node]	(2) at (-6cm, 4cm)  [fill=white] {$2$};
			\node[node] 	(3) at (-45mm, 4cm)  [fill=white] {$3$};
		        \node[node]      (4) at (-30mm, 4cm)  [fill=white] {$4$};
			\node[node]	(5) at (-15mm, 4cm)  [fill=white] {$5$};
			\node[node]	(6) at (0mm, 4cm)  [fill=white] {$6$};			

                                                   \path
				(1) edge node {} (2)
				(2) edge node {} (3)
				(3) edge node {} (4)
				(4) edge node {} (5)
				(5) edge node {} (6)
								;
		\end{tikzpicture}}
		\caption{The network in Figure 6 (iii)  of Neuberger\etal~\cite{NSS19}.}
		\label{fig:6cnetwork_2}
		\end{center}
\end{figure}

\begin{exam} \label{ex:odd-eo} \normalfont 	  
Consider the network $G$ in Figure~\ref{fig:6cnetwork_2} which corresponds to the network in Figure 6 (iii)  of Neuberger\etal~\cite{NSS19}. 
The tagged partition $\mathcal{P} = \{ P_1 =\{1,6\}, \, \overline{P}_1 =\{3,4\} , \, P_0 =\{2,5\} \}$ of the set of cells of $G$ is odd-balanced but not even-odd-balanced. 

Considering the ordering $1,6,3,4,2,5$ of the cells of $G$ adapted to the tagged partition $\mathcal{P}$, the adjacency matrix of $G$ has the following block structure:
$$
W_G =\left(
\begin{array} {cc|cc|cc}
0  & 0 &  0 &  0 &  1 &  0 \\    
0  & 0  & 0  & 0 &  0 &  1 \\  
\hline  
0 &  0  & 0  & 1 & 1 &  0   \\
0  & 0 &  1 &  0 &  0 &  1    \\
\hline
1 &   0  & 1 &  0 &  0 &  0    \\
0  & 1  & 0  & 1  & 0 &  0
\end{array}
\right)
=
\left(
\begin{array} {c|c|c}
Q_{11} & R_{11} & Z_{10}\\
\hline
\overline{R}_{11} & \overline{Q}_{11} & \overline{Z}_{10}\\
\hline
Z_{01} & \overline{Z}_{01} & Z_{00}
\end{array}
\right).
$$
The blocks $R_{11},\, \overline{R}_{11}$ are regular of the same valency $0$ and $Z_{10},\, \overline{Z}_{10}$ are regular of same valency $1$. Moreover, ${\mathrm rs}\left(Z_{01} \right)={\mathrm rs}\left(\overline{Z}_{01} \right) 
=
\left(
\begin{array} {c}
1\\
1
\end{array}
\right).
$ 
Thus, by Definition~\ref{def:odd}, the partition $\mathcal{P}$ is odd-balanced. 

However, by Definition~\ref{def:linear}, $\mathcal{P}$  is not even-odd-balanced as ${\mathrm rs}\left(Q_{11} \right) - {\mathrm rs}\left(R_{11} \right) =
\left(
\begin{array} {c}
0\\
0
\end{array}
\right)
\ne {\mathrm rs}\left(\overline{Q}_{11} \right) - {\mathrm rs}\left(\overline{R}_{11} \right) =
\left(
\begin{array} {c}
1\\
1
\end{array}
\right)
$ and, thus, fails the first condition in (4.7) of Proposition~\ref{thm:mainLaplacian}.
 \hfill $\Diamond$ 
  \end{exam}

\begin{exam} \label{ex:eo-odd} \normalfont 	
Consider the weighted network $G$ with set of cells $\{1,\ldots,8\}$ and adjacency matrix
 $$
W_G =\left(
\begin{array} {cc|cc|cc|cc}
3  & 2 &  2 &  1 &  1 &  1 & 1 & 0 \\    
3  & 1 &  2 &  2 &  1 &  0 & 1 & 1 \\    
\hline  
2  & 2 &  2 &  2 &  1 &  1 & 1 & 0 \\    
1  & 1 &  3 &  3 &  0 &  0 & 2 & 1 \\    
\hline  
-1  & 0 &  2 &  2 &  1 &  1 & 4 & 2 \\    
0  & 1 &  2 &  0 &  2 &  2 & 2 & 2 \\    
\hline  
0  & 1 &  1 &  1 &  3 &  0 & 5 & 0 \\    
0  & 1 &  0 &  2 &  2 &  1 & 1 & 4 \
\end{array}
\right)
=
\left(
\begin{array} {c|c|c|c}
Q_{11} & Q_{12}  & R_{11} & R_{12}\\
\hline
Q_{21} & Q_{22}  & R_{21} & R_{22}\\
\hline
\overline{R}_{11} & \overline{R}_{12}  & \overline{Q}_{11} & \overline{Q}_{12}\\
\hline
\overline{R}_{21} & \overline{R}_{22}  & \overline{Q}_{21} & \overline{Q}_{22}
\end{array}
\right),
$$
which is regular of valency $11$, and consider the tagged partition $\mathcal{P} = \{ P_1 =\{1,2\}, \, P_2 =\{3,4\}, \, \overline{P}_1 =\{5,6\} , \, \overline{P}_2 =\{7,8\} \}$ of the set of cells of $G$.

By Proposition~\ref{thm:mainLaplacian} and  Definition~\ref{def:linear}, $\mathcal{P}$  is even-odd-balanced (linear balanced) as, for $1 \le i,j \le 2$, we have that ${\mathrm rs}\left(Q_{ij} \right) - {\mathrm rs}\left(R_{ij} \right)$ and 
${\mathrm rs}\left(\overline{Q}_{ij} \right) - {\mathrm rs}\left(\overline{R}_{ij} \right)$ are regular of the same valency. Clearly, $\mathcal{P}$ is not odd-balanced as, for example, the block $Q_{12}$ is not regular.
 \hfill $\Diamond$ 
  \end{exam}

In \cite{NSS19}, for the particular class of networks with symmetric $(0,1)$-adjacency matrices (undirected graphs), Neuberger {\it et al.} show in Proposition 5.6 that  in an odd-balanced partition each part $P_r$ and its counterpart $\overline{P}_r$ have the same number of cells. Moreover, they conjecture that this is also true for the linear-balanced partitions. The next example shows that Proposition 5.6 (and, thus, Conjecture 5.3) in \cite{NSS19} does not hold for general weighted networks.

\begin{exam} \label{ex:odd_dif_num} 
Consider the four-cell weighted network $G$ with set of cells $\{1,2,3,4\}$ and adjacency matrix  
$$
W_{G} =
	\left( 
 \begin{array}{c|cc|c}
 0     &   \frac{1}{2}  & \frac{1}{2}  & \frac{6}{5} \\
 \hline 
 1    &                0   & 0               &     \frac{6}{5}   \\
 1   &                0   &               0  &      \frac{6}{5}   \\
 \hline 
 1   &                0   & 1                &     0  
 \end{array}
 \right) =  
 \left( 
 \begin{array}{c|c|c}
 Q_{11} & R_{11} & Z_{10}\\
 \hline
 \overline{R}_{11} & \overline{Q}_{11} & \overline{Z}_{10} \\
 \hline 
 Z_{01} & \overline{Z}_{01} & Z_{00}
 \end{array}
 \right)
 \, .
 $$
Take the tagged partition 
$${\mathcal P} = \left\{ P_1=\{1\}, \overline{P}_1=\{ 2,3\}, P_0=\{ 4\} \right\}$$
and note that the enumeration of the network set of cells is adapted to this partition. 
As $R_{11}, \overline{R}_{11}$ are regular of the same valency, $Z_{10}, \overline{Z}_{10}$ are regular of the same valency and 
${\mathrm rs}\left( Z_{01} \right) = {\mathrm rs} \left( \overline{Z}_{01}\right)$, we have that $\mathcal{P}$  
 is odd-balanced for $G$. Note that $\# P_1 \ne \# \overline{P}_1$. 
\hfill $\Diamond$
\end{exam}

In the following remark, we consider tagged partitions where $r=0$, that is, there is no zero part $P_0$, and we relate the concepts of exo-balanced and odd-balanced partitions given their definitions in Definitions~\ref{def:balanced} (i) and~\ref{def:odd}, respectively. 
Observe that, by definition, an odd-balanced partition is, in particular, a non-standard tagged partition and an exo-balanced partition is standard. So, 
in the next remark we relate the two concepts  of exo-balanced and odd-balanced partitions of a partition  $\mathcal{P}$, with non zero part, by interpreting the $q >0$ counterparts when 
referring to $\mathcal{P}$ as odd-balanced and as independent parts if  interpreting $\mathcal{P}$  as a standard partition.

\begin{rem}\normalfont 
Let $G$ be a weighted network with set of cells $C$ and adjacency matrix $W_G$.  Consider a 
non-standard 
tagged partition $\mathcal{P} =\{ P_1, P_2, \ldots, P_p,$  $\overline{P}_1, \overline{P}_2, \ldots, \overline{P}_q \}$ of $C$ and consider an ordering of the  cells adapted to $\mathcal{P}$ so that $W_G$ has a block form (\ref{eq:oddbf}). 
Set $P_{p+1} = \overline{P}_1, \ldots, P_{p+q} = \overline{P}_q$.
We have:\\
(i) If  $\mathcal{P}$  is odd-balanced  then 
the standard partition $\{ P_1, P_2, \ldots, P_p, P_{p+1}, \ldots, P_{p+q}\}$
  is exo-balanced. The converse is not true. \\
(ii) Assume 
the standard partition $\{ P_1, P_2, \ldots, P_p, P_{p+1}, \ldots, P_{p+q}\}$
 is exo-balanced. Then $\mathcal{P}$ is odd-balanced if and only if  for all $i \not=j$, 
with $1 \leq i,j \leq q$,
 the regular matrices, $Q_{ij}, \overline{Q}_{ij}$  have the same valency, and  for all 
$1 \leq i, j \leq q$,
 the regular matrices $R_{ij}, \overline{R}_{ij}$ have the same valency.
\hfill $\Diamond$
\end{rem}

\subsubsection*{Quotient networks for odd-balanced, linear-balanced and even-odd-balanced partitions}
We define next the concepts of quotient networks for weighted networks by odd-balanced, linear-balanced and even-odd-balanced partitions. In the next section, we show an  application of these concepts to coupled cell systems with additive linear input. 

\begin{Def}\label{def:quo_odd}
Let $G$ be a network with set of cells $C = \{1, \ldots, n\}$ and  $\mathcal{P}$ a 
non-standard
tagged partition of $C$ with $p+q+1$ parts, $P_1, P_2, \ldots, P_p, \overline{P}_1, \overline{P}_2, \ldots, \overline{P}_q, P_0$ and recall Definitions~\ref{def:linear} and~\ref{def:odd}. \\
(i) If $\mathcal{P}$ is odd-balanced, we define the {\it symbolic quotient} of $G$ by the odd-balanced partition $\mathcal{P}$ to be the $(p+q+1)$-cell network, where the cells are the parts of $\mathcal{P}$ and the edges are defined in the following way. 
For $i, j_1=1, \ldots, p$ and $j_2=1, \dots, q$ 
such that $i\not=j_1$, there are directed edges from $P_{j_1}$ ($\overline{P}_{j_2}$) to $P_i$ with weight the valency of $Q_{i,,j_1}$ ($R_{i,j_2}$).
For $i=1, \ldots, p$,  there are directed edges from $P_0$ to $P_i$ with the valency of $Z_{i0}$. \\
(ii)  If $\mathcal{P}$ is linear balanced, take an enumeration of the network set of cells adapted to the tagged partition $\mathcal{P}$  so that the  the adjacency matrix $W_G$ of $G$ has a  block structure (\ref{eq:oddbf}) satisfying the conditions (i)-(v)  of Corollary~\ref{cor:mainLaplacian}. We call the following matrix the adjacency matrix of 
the {\it symbolic quotient} of $G$ by the linear balanced partition  $\mathcal{P}$:
\ \\
\begin{equation}
\left( 
 \begin{array}{cccc|c}
0 & q_{12} &  \cdots &q_{1p} & r_{1} \\
 \vdots& \vdots & \cdots &\vdots & \vdots \\
q_{p1} & q_{p2} & \cdots & 0 & r_{p} 
 \end{array}
 \right)\, .
 \label{eq:quolin}
 \end{equation}
 (iii) If $\mathcal{P}$ is even-odd-balanced, take an enumeration of the network set of cells adapted to the tagged partition $\mathcal{P}$  so that the  the adjacency matrix $W_G$ of $G$ has a  block structure (\ref{eq:oddbf}) satisfying the conditions (\ref{eq:equal_Lap})-(\ref{eq:equal_Lap_Z})  of Proposition~\ref{thm:mainLaplacian}. Denoting by $q_{ij}$ the valency of 
 ${\mathrm rs}\left(Q_{ij}\right) - {\mathrm rs}\left( R_{ij}\right)$ for $1 \leq i \leq p,\ 1 \leq  j \leq q$, of  ${\mathrm rs}\left(Q_{ij}\right)$ if $1 \leq i \leq p;\   q+1 \leq j \leq p$, then we call the following matrix the adjacency matrix of 
the {\it symbolic quotient} of $G$ by the even-odd-balanced partition  $\mathcal{P}$:
\ \\
\begin{equation}
\left( 
 \begin{array}{cccc}
q_{11} & q_{12} &  \cdots &q_{1p}  \\
 \vdots& \vdots & \cdots &\vdots  \\
q_{p1} & q_{p2} & \cdots & q_{pp}  
 \end{array}
 \right)\, .
 \label{eq:quoeo}
 \end{equation}
\hfill $\Diamond$
\end{Def}

\begin{exam}\normalfont \label{ex:simples}
Take the three-cell network $G$ at the left of Figure~\ref{fig:odd3cell} and the tagged partition 
$$\mathcal{P} = \left\{P_1 = \{1\},\, \overline{P}_1 = \{2,3\}  \right\}\, .$$ 
The adjacency matrix of $G$ has the block form 
$$\left( 
 \begin{array}{c|cc}
 0&1&1\\
 \hline 
 2&0&0\\
 2&0&0
 \end{array}
 \right) = 
 \left( 
 \begin{array}{c|c}
 Q_{11}&R_{11}\\
 \hline 
 \overline{R}_{11}& \overline{Q}_{11}
 \end{array}
 \right),$$
where the block $R_{11} = (1 \, 1)$ is regular of valency $2$, that is, it has row sum $2$ and the block $\overline{R}_{11} = (2 \, 2)^t$ is regular of valency $2$, since the entry of each row is $2$. The partition  $\mathcal{P}$ is exo-balanced since $R_{11}$ and $\overline{R}_{11}$ are regular. The quotient network is network $Q_2$ in Figure~\ref{fig:odd3cell}. 
As $R_{11}$ and $\overline{R}_{11}$ are regular of the same valency, we have that $\mathcal{P}$ is odd-balanced. The symbolic quotient network is network $Q_1$ in Figure~\ref{fig:odd3cell}. If the entries of the block $ \overline{R}_{11}$ were $3$, instead of $2$, then the partition $\mathcal{P}$ would also be exo-balanced but not odd-balanced. 
\hfill $\Diamond$ 
\end{exam}

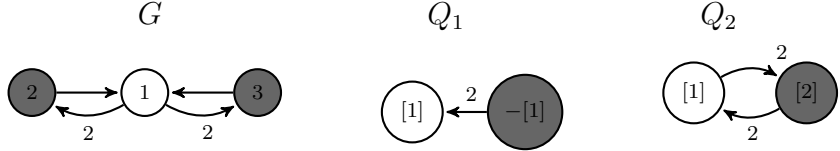
\begin{figure}[h!]
\begin{center}
\begin{tabular}{ccc}
$G$ & $Q_1$ & $Q_2$ \\
 \begin{tikzpicture}[->,>=stealth',shorten >=1pt,auto, node distance=1.5cm, thick,node/.style={circle,draw}]
			\node[node]	(2) at (-5cm, 2cm)  [fill=black!60] {{\tiny $2$}};
			\node[node]	(1) at (-35mm, 2cm)  [fill=white] {{\tiny $1$}};
			\node[node]	(3) at (-20mm, 2cm) [fill=black!60] {{\tiny $3$}};
                                                   \path
                                                            
				(2) edge node {} (1)
				(3) edge node {} (1)
				 (1) [->] edge[bend left=30, thick]  node {{\tiny $2$}}    (2)    
				(1) [->] edge[bend right=30, thick]  node [below right, pos=.35]  {{\tiny $2$}}    (3);    
		\end{tikzpicture}	\qquad & \qquad 
		\begin{tikzpicture}[->,>=stealth',shorten >=1pt,auto, node distance=1.5cm, thick,node/.style={circle,draw}]
			\node[node]	(1) at (-35mm, 2cm)  [fill=white] {{\tiny $[1]$}};
			\node[node]	(2) at (-20mm, 2cm) [fill=black!60] {{\tiny $-[1]$}};
                                                   \path                 
				(2) edge node [above right, pos=.50, near end]  {{\tiny $2$}} (1);
		\end{tikzpicture}	\qquad & \qquad 
	\begin{tikzpicture}[->,>=stealth',shorten >=1pt,auto, node distance=1.5cm, thick,node/.style={circle,draw}]
			\node[node]	(1) at (-35mm, 2cm)  [fill=white] {{\tiny $[1]$}};
			\node[node]	(2) at (-20mm, 2cm) [fill=black!60] {{\tiny $[2]$}};
                                                   \path
				(2) edge[bend left=30, thick]  node [below right, near end]  {{\tiny $2$}} (1) 
				(1) edge[bend left=30, thick]  node [above right, near end]  {{\tiny $2$}} (2);
		\end{tikzpicture}		
		
		\end{tabular}
		\end{center}
		\caption{A three-cell regular network $G$ of valency two. The tagged partition $\mathcal{P} = \left\{[1] = P_1 = \{1\},\, [2] = \overline{P}_1 = \{2,3\} \equiv -[1] \right\}$ is odd-balanced  and exo-balanced for $G$. On the center we see the symbolic quotient network $Q_1$ of $G$ by the odd-balanced tagged partition $\mathcal{P}$. On the right, $Q_2$ is the quotient network of $G$ by the exo-balanced partition $\mathcal{P}$.}
		\label{fig:odd3cell}
	\end{figure}

\begin{exam}\normalfont 
Take the six-cell network $G$ in Figure~\ref{f:sixSwift} and  the tagged partition of the network set of cells $\mathcal{P} = \{ [1] = P_1 =\{1\}, \, -[1] = \overline{P}_1 = \{ 2\},\, [3] = P_0 =\{3,4,5,6\}\}$. This network is the bidirectional network in Figure 9 of \cite{NSS19}.  The adjacency matrix of $G$ has  the following block form: 
$$
W_{G} =
	\left( 
 \begin{array}{c|c|cccc}
0 & 0 & 1& 1 & 0 & 0 \\
\hline 
0 & 0 & 1& 1 & 0 & 0 \\
\hline 
1 & 1 & 0& 0 & 0& 0 \\
1 & 1 & 0& 0 & 1 & 0 \\
0 & 0 & 0& 1 & 0& 1 \\
0 & 0 & 0& 0 & 1 & 0 
 \end{array}
 \right) = 
 \left( 
 \begin{array}{c|c|c}
 Q_{11} & R_{11} & Z_{10}\\
 \hline 
 \overline{R}_{11} & \overline{Q}_{11} & \overline{Z}_{10} \\
 \hline 
 Z_{01} & \overline{Z}_{01} & Z_{00}
 \end{array}
 \right)
 \, .
 $$	
 Note that, both $R_{11}$ and $\overline{R}_{11}$ are regular of valency $0$ and $Z_{10}$ and $\overline{Z}_{10}$ are regular of valency $2$. Moreover, 
 $
 {\mathrm rs}(Z_{01}) = {\mathrm rs}\left(\overline{Z}_{01}\right) = (1 \,  1 \,   0 \,   0)^t$. We have that $\mathcal{P}$ is odd-balanced.  See in Figure~\ref{f:sixSwift} the symbolic quotient of $G$ by the odd-balanced partition $\mathcal{P}$. However, $\mathcal{P}$ is not exo-balanced precisely because of the above equality: cells $3,4$ of the class $P_0$ receive one input from cells in the class $P_1$ (resp. $\overline{P}_1$)), whereas cells $5,6$ receive no inputs from cells in the class $P_1$ (resp. $\overline{P}_1$). 
 \hfill $\Diamond$ 
\end{exam}

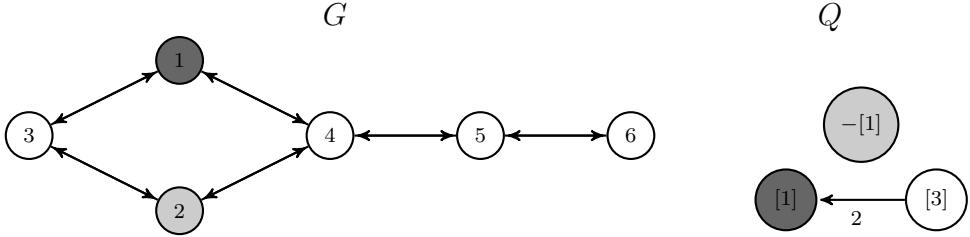
\begin{figure}[h!]
\begin{center}
\begin{tabular}{cc}
$G$ & $Q$  \\
\begin{tikzpicture}[->,>=stealth',shorten >=1pt,auto, node distance=1.5cm, thick,node/.style={circle,draw}]
                           \node[node]	(1) at (-5cm, 3cm)  [fill=black!60] {{\tiny $1$}};
			\node[node]	(2) at (-5cm, 1cm)  [fill=black!20] {{\tiny $2$}};
			\node[node]	(3) at (-7cm, 2cm) [fill=white] {{\tiny $3$}};
			\node[node]	(4) at (-3cm, 2cm)  [fill=white] {{\tiny $4$}};
			\node[node]	(5) at (-1cm, 2cm)  [fill=white] {{\tiny $5$}};
			\node[node]	(6) at (1cm, 2cm) [fill=white] {{\tiny $6$}};
                                                   \path
                                                            
				(4) edge node {} (5)
				(5) edge node {} (4)
				(5) edge node {} (6)
				(6) edge node {} (5)
				(3) edge node {} (1)
				(1) edge node {} (3)
				(3) edge node {} (2)
				(2) edge node {} (3)
				(1) edge node {} (4)
				(4) edge node {} (1)
				(2) edge node {} (4)
				(4) edge node {} (2);    
		\end{tikzpicture}	\qquad & \qquad 
\begin{tikzpicture}[->,>=stealth',shorten >=1pt,auto, node distance=1.5cm, thick,node/.style={circle,draw}]
			\node[node]	(1) at (-40mm, 2cm)  [fill=black!60] {{\tiny $[1]$}};
			\node[node]	(3) at (-20mm, 2cm) [fill=white] {{\tiny $[3]$}};
			\node[node]	(2) at (-30mm, 3cm)  [fill=black!20] {{\tiny $-[1]$}};
                                                   \path
				(3) edge[thick]  node [below right, near end]  {{\tiny $2$}} (1);
		\end{tikzpicture}		
		
		\end{tabular}
		\end{center}
		\caption{A six-cell bidirectional network $G$. The tagged partition $\mathcal{P} = \left\{[1]= P_1 = \{1\},\, -[1]= \overline{P}_1 = \{2\},\, [3] = P_0 = \{ 3,4,5,6\} \right\}$ is odd-balanced for $G$ but not exo-balanced.  On the right, $Q$ is the symbolic quotient network of $G$ by the odd-balanced partition $\mathcal{P}$.}
		\label{f:sixSwift}
	\end{figure}

\section{Coupled cell systems with additive input structure}\label{sec_CCNS}

Let $G$ be an $n$-cell network with weighted adjacency matrix $W_G$. We consider the cells of $G$  as individual dynamical systems, given by ordinary differential equations. We assume that the cells are all of the same type, that is, have the same phase space and the same internal dynamics. The dynamical systems that we associate to $G$  are such that  the couplings between the cells, the way they influence the dynamical evolution of each other, are determined by the edges of $G$ and corresponding weights. These are called {\it coupled cell systems}. More precisely, we take  a cell to be a system of ordinary differential equations and we consider coupled cell systems with {\it additive input structure}~\cite{F15,BF17}. Let $C=\{1, \ldots,n\}$ be the set of cells of $G$ where each cell $c$ has phase space $P_c=\R^k$.  A coupled cell system  with {\it additive input structure} is given by $\dot{x} = f(x)$, where $f = (f_1, \ldots, f_n)$ so that the equation  $\dot{x}_j = f_j(x)$ is associated with cell $j$ and it has the form: 
\begin{equation} 
\dot{x}_j=g(x_j) +\sum_{i=1}^n {w_{ji}h\left(x_j,x_i\right)} \quad \left( j=1, \ldots, n\right)
\label{eq:EDOsystem}
\end{equation}
where   $g:\R^k \rightarrow \R^k$ and $h:\R^k \times \R^k \rightarrow \R^k$ are smooth functions; each $w_{ji}\in \R$ is the value of the weight of the coupling strength from cell $i$ to cell $j$. The function $g$ characterizes the {\it internal dynamics} and the function $h$ is the \textit{coupling function}.  Systems of ordinary differential equations where cells are governed by equations of the form  (\ref{eq:EDOsystem}) are said to be $G$-{\it admissible} as they encode the network structure of $G$. 

\begin{rem} \normalfont 
The difference-coupled vector fields considered in Neuberger {\it et al.}~\cite{NSS19} are a particular class of input additive coupled cell systems where the coupling matrix is a symmetric $(0,1)$-matrix and $h(u,v) = \tilde h(u-v)$, for some function $\tilde h$. 
\hfill $\Diamond$
\end{rem}

\subsection{Additive coupled cell systems and additional restrictions}

We present next four subclasses of coupled cell systems with additive input structure associated with general weighted networks where restrictions are imposed on the internal dynamics and coupling functions. The first three subclasses are  an extension to general weighted networks of Definition 2.2 in Neuberger {\it et al.}~\cite{NSS19} for $0-1$ undirected networks. 

\begin{Def} \normalfont   
Let $G$ be an $n$-cell coupled cell network with weighted adjacency matrix $W_G$. Given a choice of cells phase spaces, take an input additive coupled cell system admissible by $G$ as  defined by (\ref{eq:EDOsystem}) where $w_{ji}$ is the entry $ji$ of $W_G$. Denote by $I_G$ the set of these coupled cell systems admissible by $G$ and define the following subsets: \\
(i) $I_{G,0} = \{ f \in I_G |\ h(x,y) = 0 \mbox{ if } x=y\}$ is the set of {\it exo-input-additive coupled cell systems}.\\
(ii) $I_{G,odd} = \{f \in I_{G,0} |\ g, h \mbox{ are odd} \}$ is the set of {\it odd-input-additive coupled cell systems}.\\
(iii) $I_{G,l} = \{ f \in I_{G,0}|\ g \mbox{ is odd and } h \mbox{ is linear}\}$ is the set of {\it linear-input-additive coupled cell systems}.\\
(iv) $I_{G,eo} = \{ f \in I_{G}|\ g \mbox{ is odd; }  h \mbox{ is even in $x$ and odd in $y$}\}$ is the set of {\it even-odd-input-additive coupled cell systems}.
 \hfill $\Diamond$
\end{Def}

\begin{rem} \normalfont (i) For any choice of cells phase spaces, we have that $I_{G,0}$ is a proper subspace of $I_G$. It follows in particular that it is natural to predict the existence of  subspaces that are flow-invariant under any coupled cell system for $G$ with additive input structure (for any choice of cell phase spaces) in $I_{G,0}$ which will not have that property in $I_G$. This issue is addressed in the next section.\\ 
(ii) Note that, in (\ref{eq:EDOsystem}), if the coupling function $h$ is linear then $h(-x,-y) = -h(x,y)$ for all $x,y$. Thus we have the following inclusions: $I_{G,l} \subseteq I_{G,odd} \subseteq I_{G,0} \subset I_G$. Moreover, the conditions defining  $I_{G,l}$ imply that the linear coupling function $h$ satisfies $h(x,-x) = 2h(x,0)$, for all $x \in \R^k$, as $h(x,-x) = h(x,0) + h(0,-x) = h(x,0) - h(0,x)$ and from $h(x,x) = 0$, we have that $h(0,x) = -h(x,0)$. 
\hfill $\Diamond$
\end{rem}

We describe now the general form of the smooth coupling functions taking the restrictions of $I_{G,0}, I_{G,odd}, I_{G,l}$ 
and $I_{G,eo}$.

\begin{prop} \label{prop:gen_form_h}
Take $a = (a_1, \ldots, a_k) \in \R^k$ and $b = (b_1, \ldots, b_k) \in \R^k$. The coupling function $h$ in (\ref{eq:EDOsystem}) has the following form:\\
(i) If $f$ in $I_{G,0}$ then 
$$
h(a,b) = (a_1-b_1) l_1(a,b) + \cdots + (a_k-b_k) l_k(a,b), 
$$
where for $j=1, \ldots, k$, the function $l_j:\, \R^k \times \R^k \to \R^k$ is smooth. \\
(ii) If $f$ in $I_{G,odd}$ then 
$$
h(a,b) = (a_1-b_1) m_1(a_1^2, \ldots, a_k^2, b_1^2, \ldots, b_k^2) + \cdots + (a_k-b_k) m_k(a_1^2, \ldots, a_k^2, b_1^2, \ldots, b_k^2), 
$$
where for $j=1, \ldots, k$, the function $m_j:\, \R^k \times \R^k \to \R^k$ is smooth. \\
(iii) If $f$ in $I_{G,l}$ then 
$$
h(a,b) = (a_1-b_1) (a_{11}, \ldots, a_{1k}) + \cdots + (a_k-b_k) (a_{k1}, \ldots, a_{kk})
$$
where $a_{ij} \in \R$ for all $i,j=1, \ldots, k$.  \\
(iv) If $f$ in $I_{G,eo}$ then 
$$
h(a,b) = b_1 m_1(a_1^2, \ldots, a_k^2, b_1^2, \ldots, b_k^2) + \cdots + b_k m_k(a_1^2, \ldots, a_k^2, b_1^2, \ldots, b_k^2), 
$$
where for $j=1, \ldots, k$, the function $m_j:\, \R^k \times \R^k \to \R^k$ is smooth. \\
\end{prop}

\begin{proof}
The proof of (i) follows from an adaptation of Lemma~3.1 in Chapter II  of \cite{GS85} (see \cite[Exercise II 3.3]{GS85}). Trivially, (ii) and (iii) follow from (i).  
The proof of (iv) follows trivially from the symmetry of $h$, that is, $h(a,b)$ must be even in $a$ and odd in $b$.
\end{proof}

\begin{prop} \label{prop:reg_nreg_L_W}
Let $G$ be an $n$-cell weighted network with adjacency matrix $W_G$ and Laplacian matrix $L_G$. In  (\ref{eq:EDOsystem}), assume  $k=1$, that is, assume the cell phase spaces to be $\R$. \\
(i) The linear subspace of the linear vector fields in $I_{G,0},\ I_{G, odd},\ I_{G, l}$ is  $< \mathrm{id}_n,\, L_G>$ 
and  in $I_{G,eo}$ is  $<\mathrm{id}_n,\, W_G>$.\\
(ii) If $G$ is regular, then we have that the linear subspace of the linear vector fields in $I_{G},\, I_{G,0},\ I_{G, odd},\ I_{G, l}, I_{G,eo}$ is  $<\mathrm{id}_n,\, W_G>  = < \mathrm{id}_n,\, L_G>$.
\end{prop}

\begin{proof} (i) If $f \in  I_{G,0}$ is linear, that is, both $g,h$ are  linear and $h(x,x) =0$ for all $x \in \R$, we have $f(x) = \alpha x $ and $h(x,y) = \beta (x-y) $ for all $x,y \in \R$. 
This follows trivially from Proposition~\ref{prop:gen_form_h} (iii). Thus 
$$
\begin{array}{rcl}
f_i (x_1, \ldots, x_n) & = & \alpha x_i + \beta \sum_{j\not=i; j=1}^n w_{ij} (x_i -x_j) =  \alpha x_i + \beta \sum_{ j=1}^n w_{ij} (x_i -x_j)\\
& = &  \alpha x_i + \beta x_i \sum_{j=1}^n w_{ij} - \beta \sum_{j=1}^n w_{ij} x_j  \\
& = &  \alpha x_i + \beta \left( v(i) x_i - \sum_{j=1}^n w_{ij} x_j \right) =  \alpha x_i + \beta (L_Gx)_i\, .
 \end{array}
 $$
That is,  $f = \alpha \mbox{id}_n + \beta L_G$. Moreover, any linear map on $\R^n$ of the type $\alpha \mbox{id}_n + \beta L_G$  belongs to $I_{G,odd}$ and $I_{G,l}$. \\
Let $f \in  I_{G,eo}$ be linear, that is, both $g,h$ are  linear where $g$ is odd and $h(x,y)$ is even in $x$ and odd in $y$. Thus $g(x) = \alpha x$ and trivially, from Proposition~\ref{prop:gen_form_h} (iv), it follows that $h(x,y) = \beta y$. It follows that 
$$
f_i (x_1, \ldots, x_n)  =  \alpha x_i + \beta \sum_{j=1}^n w_{ij} x_j =  \alpha x_i + \beta (W_Gx)_i, 
 $$
that is,  $f = \alpha \mbox{id}_n + \beta W_G$.  \\
(ii) When $G$ is regular, we have $L_G = v_W \mbox{id}_n - W_G$ and so  $<\mathrm{id}_n,\, W_G>  = < \mathrm{id}_n,\, L_G>$. 
\end{proof}

\section{Balanced partitions and synchrony 
in the class of the coupled cell systems with input additive structure}
\label{sec_bal}

Following ~\cite{GSP03,GST05}, a polydiagonal  $\Delta$ is a {\it synchrony subspace} of a weighted network $G$ when it is left invariant under the flow of every $G$-admissible coupled cell system with additive input structure.  

Recall that, given a weighted network $G$ and a balanced standard partition $\mathcal{P}$ on its set of cells $C$, we denote by $\Delta_{\mathcal{P}}$, the associated polydiagonal subspace, and by $G_{\mathcal{P}}$, the quotient network of $G$ by $\mathcal{P}$.

\begin{thm}[\cite{ADF17}] \label{thm:SSBR}
Let  $G$ be an $n$-cell weighted network. Consider the admissible coupled cell systems for $G$ with additive input structure, for any given choice of total phase space $\left(\R^k\right)^n$. Then:\\
(i) A polydiagonal subspace $\Delta_{\mathcal{P}}$ associated with a standard partition $\mathcal{P}$ is a synchrony subspace for $G$ if and only if the partition $\mathcal{P}$ is balanced on the set of cells of $G$. \\
(ii) Let $\mathcal{P}$ be a standard balanced partition on the set of cells of $G$. Then:\\
(ii.a) The restriction to $\Delta_{\mathcal{P}}$ of a $G$-admissible coupled cell system with additive input structure  is a $G_{\mathcal{P}}$-admissible coupled cell system with additive input structure.\\
(ii.b) Every $G_{\mathcal{P}}$-admissible  coupled cell system with additive input structure is the restriction to $\Delta_{\mathcal{P}}$ of a $G$-admissible coupled cell system with additive input structure. 
\end{thm}

Let $G$ be a weighted network and $W_G$ the corresponding weighted adjacency matrix. Let $\mathcal{P}$ be a balanced
 standard partition of the set of cells of $G$ with parts $P_1, \ldots, P_p$ and consider the corresponding block structure (\ref{eq:oddbf}) of $W_G$. 
  Denote coordinates on $\Delta_{\mathcal{P}}$ by $(y_1, \ldots, y_p)$ where $y_j = x_k$ for (all) $ k \in P_j$, where $j=1, \ldots, p$. The restriction of (\ref{eq:EDOsystem}) to the polydiagonal space $\Delta_{\mathcal{P}}$ is admissible for the quotient $G_{\mathcal{P}}$ with adjacency matrix $[q_{ij}]$ given by:
\begin{equation} 
\dot{y}_j=g(y_j) +\sum_{i=1}^p {q_{ji}h\left(y_j,y_i\right)} \quad \left(j=1, \ldots, p\right)\, .
\label{eq:restEDO}
\end{equation}

\begin{rem}\normalfont 
In $I_{G,0}$, we have that in (\ref{eq:EDOsystem}) and (\ref{eq:restEDO}), the terms $a_{jj} h(x_j,x_j)$ and $q_{jj} h(y_j,y_j)$ vanish, respectively. Thus 
$$
I_{G,0} \subsetneqq I_G\, . 
$$
It follows then that in $I_{G,0}$, for a polydiagonal subspace $\Delta_{\mathcal{P}}$ to be a synchrony subspace for $G$, that is, to be left invariant under the flow of any system of the form  (\ref{eq:EDOsystem}) where $f \in I_{G,0}$, we expect less restrictions to be imposed on $\mathcal{P}$. In fact, to be precise, we can relax the condition of a partition to be balanced by dropping down the conditions on the blocks $Q_{jj}$  that have constant row sum. That is, the standard partition $\mathcal{P}$ must be exo-balanced as we show  in the next section. 
\hfill $\Diamond$
\end{rem}

\section{Exo-balanced partitions and synchrony in 
the class of the  exo-input-additive coupled cell systems}
\label{sec_exo_bal}

In this section we enlarge the set of synchrony subspaces of a network by restricting to coupled cell systems that are exo-input-additive. 
Let $G$ be an $n$-cell weighted network and $W_G$ the corresponding weighted adjacency matrix. When $f \in I_{G,0}$, equations (\ref{eq:EDOsystem}) for the 
input additive coupled cell systems admissible by $G$ simplify to: 
 \begin{equation} 
\dot{x}_j=g(x_j) +\sum_{i=1,
 i\not= j}^n w_{ji} h\left(x_j,x_i\right) \quad \left(j=1,\ldots,n\right)\, .
\label{eq:2EDO}
\end{equation}

The following result is an extension, to weighted coupled cell networks and input additive coupled cell systems, of Theorem 3.13 in Neuberger {\it et al.}~\cite{NSS19}.

\begin{prop} \normalfont 
Let $G$ be an $n$-cell weighted network and $\mathcal{P}$ a standard partition of its set of cells. The partition $\mathcal{P}$  is exo-balanced for $G$ if and only if  
$\Delta_{\mathcal{P}}$ is left invariant under the flow of every system in $I_{G,0}$, for any given choice of total phase space $\left(\R^k\right)^n$. 
\end{prop}

\begin{proof} 
It follows from Remark~\ref{rmk:bal_G_L}, that a partition $\mathcal{P}$ is exo-balanced for $G$ if and only it it is balanced for the network $G_{-L}$  with adjacency matrix $-L_G$, where $L_G=[l_{ij}]$ is the Laplacian matrix of $G$. 
By Theorem~\ref{thm:SSBR}, this is equivalent to the polydiagonal subspace $\Delta_{\mathcal{P}}$ being a synchrony subspace for $G_{-L}$ which is equivalent to the polydiagonal subspace $\Delta_{\mathcal{P}}$ being left invariant under the flow of every system in $I_{G_{-L}}$. For the systems in $I_{G_{-L}}$, the equation $\dot x_j = f_j(x)$ associated with cell $j$ has the form:
\begin{equation} 
\dot{x}_j=g(x_j)  + \sum_{i=1}^n \left(-l_{ji}\right) h\left(x_j,x_i\right) \quad \left(j =1, \ldots, n\right)\, .
\label{eq:sys_Lag}
\end{equation}
Given that $w_{ji} = -l_{ji}$, for $i \ne j$, when $h(u,u)=0$, we have that the equations in (\ref{eq:2EDO}) and (\ref{eq:sys_Lag}) are the same. That is, $I_{G_{-L},0}$ and $I_{G,0}$ coincide. The result then follows.
\end{proof}

 Let $\mathcal{P}$ be a strict exo-balanced
 standard partition on the set of cells of $G$ with parts $P_1, \ldots, P_p$ and consider an enumeration of the network set of cells adapted to $\mathcal{P}$ so that $W_G$ has a 
 block structure (\ref{eq:oddbf}).  Recall Definition~\ref{def:quo:_exo} where it is described what we call the quotient network $G_{\footnotesize{\mathcal{P}}}$ which has the $p \times p$ adjacency matrix $W_{G_{\footnotesize{\mathcal{P}}}} = [q_{ij}]$ with $q_{ii} = 0$ and $q_{ij} = v_{Q_{ij}}$, for $i \ne j$. 

Equations (\ref{eq:2EDO}), when restricted to $\Delta_{\mathcal{P}}$ are given by: 
\begin{equation} 
\dot{y}_j=g(y_j) +\sum_{i=1, i\not=j}^p {q_{ji}h\left(y_j,y_i\right)} \quad \left( j=1, \ldots, p\right) 
\label{eq:2restEDO}
\end{equation}
which  are admissible by the network with adjacency matrix 
$$\left( 
 \begin{array}{c|c|c|c}
0 & q_{12} &  \cdots &q_{1p}\\
\hline 
 \vdots& \vdots & \cdots &\vdots\\
 \hline 
q_{p1} & q_{p2} & \cdots &0
 \end{array}
 \right)\, .$$
 In particular, these restricted equations are also the restriction to $\Delta_{\mathcal{P}}$ of equations (\ref{eq:2EDO}),  for the network 
 with adjacency matrix 
 \begin{equation}
\left( 
 \begin{array}{c|c|c|c}
0_{11} & Q_{12} &  \cdots &Q_{1p}\\
\hline 
 \vdots& \vdots & \cdots &\vdots\\
 \hline 
Q_{p1} & Q_{p2} & \cdots &0_{pp}
 \end{array}
 \right)\, . 
 \label{eq:2zerodiagbf}
 \end{equation}

		\begin{exam} \normalfont 
Recall the three-cell network $G$ at the left of Figure~\ref{fig:odd3cell}. The standard partition $\mathcal{P} = \left\{P_1 = \{1\},\, P_2 = \{2,3\} \right\}$ is balanced and exo-balanced. A coupled cell system of the form  (\ref{eq:EDOsystem}) for $G$ where $f \in I_{G}$ or $f \in I_{G,0}$ takes the form
$$
\left\{ 
\begin{array}{l}
\dot{x}_1 = g(x_1) + h(x_1,x_2) + h(x_1, x_3)\\
\dot{x}_2 = g(x_2) + 2h(x_2,x_1)\\
\dot{x}_3 = g(x_3) + 2h(x_3,x_1)
\end{array}
\right. \, .
$$
Restricting any such system to  $\Delta_{\mathcal{P}} = \{x:\, x_2 = x_3\}$, we obtain
$$
\left\{ 
\begin{array}{l}
\dot{x}_1 = g(x_1) + 2h(x_1,x_2)\\
\dot{x}_2 = g(x_2) + 2h(x_2,x_1)
\end{array}
\right. \, .
$$
This system is  admissible for the quotient network $Q_2$ at the right of Figure~\ref{fig:odd3cell}. In fact, if $f \in I_{G,0}$, then this restricted system is in $I_{Q_{2},0} \subset I_{Q_{2}}$. The network $Q_2$ is a two-cell bidirectional ring network where the edges have weight two. 
\hfill $\Diamond$
\end{exam}

\begin{exam} \normalfont  Consider the four-cell network $G$ with weighted adjacency matrix 
$$
W_{G} =
	\left( 
 \begin{array}{cccc}
0 & -3 & -1& -2  \\
-1 & 0 & -1& -1  \\
-3 & 0 & 0 & -1  \\
-1 & -1 & -1& 0  
 \end{array}
 \right)
$$
and note that $\mathcal{P} = \left\{P_1 = \{1,2,4\},\, P_2 = \{3\} \right\}$ is a strict exo-balanced standard partition  for $G$. An exo-input-additive coupled cell system for $G$, that is, in $I_{G,0}$, takes the form
$$
\left\{ 
\begin{array}{l}
\dot{x}_1 = g(x_1) -3 h(x_1,x_2) - h(x_1, x_3) - 2h(x_1, x_4)\\
\dot{x}_2 = g(x_2) - h(x_2,x_1) - h(x_2, x_3) - h(x_2, x_4)\\
\dot{x}_3 = g(x_3) - 3h(x_3,x_1) - h(x_3, x_4)\\
\dot{x}_4 = g(x_4) - h(x_4,x_1) - h(x_4, x_2) - h(x_4, x_3)
\end{array}
\right. \, .
$$
Restricting any such system to  $\Delta_{\mathcal{P}} = \{x:\, x_1=x_2 = x_4\}$, given that $h(u,u)=0$, we get the system
$$
\left\{ 
\begin{array}{l}
\dot{x}_1 = g(x_1) -h(x_1,x_3)\\
\dot{x}_3 = g(x_3) -4h(x_3,x_1)
\end{array}
\right. \, .
$$
This system is admissible for the quotient network of $G$ by the exo-balanced partition $\mathcal{P}$ with adjacency matrix 
$
\left( 
\begin{array}{cc}
0 & -1 \\
-4 & 0
\end{array}
\right)
$ (recall Definition~\ref{def:quo:_exo}). 
\hfill $\Diamond$
\end{exam}

\section{Odd-balanced partitions and anti-synchrony  in 
the class of the odd-input-additive coupled cell systems} \label{sec_odd_bal}

A  non-standard generalized polydiagonal  left invariant under the flow of every odd-input-additive coupled cell system admissible by a weighted network $G$ is an {\it anti-synchrony subspace} of $G$. We show next that these anti-synchrony subspaces of $G$ are the non-standard generalized polydiagonals associated with the odd-balanced tagged partitions of $G$. This result is an extension, to weighted coupled cell networks and input additive coupled cell systems, of Theorem 4.14 in Neuberger {\it et al.}~\cite{NSS19}.

\begin{prop} \normalfont \label{prop:odd_iff}
Let $G$ be a weighted network and $\mathcal{P}$ a tagged partition of its set of cells 
which is not standard.
The tagged partition $\mathcal{P}$  is odd-balanced for $G$ if and only if  the generalized polydiagonal $\Delta_{\mathcal{P}}$ is left invariant under the flow of every system in $I_{G,odd}$, for any given choice of total phase space  $\left(\R^k\right)^n$. 
\end{prop}

\begin{proof} 
Let $G$ be an $n$-cell weighted network with set of cells $C$ and adjacency matrix $W_G$.  Consider a  tagged partition $\mathcal{P}$ of $C$  
formed by parts $P_1, P_2, \ldots, P_p$,  counterparts  $\overline{P}_1, \overline{P}_2, \ldots, \overline{P}_q$ and zero part $P_0$.
Consider an enumeration of the cells adapted to $\mathcal{P}$  so that $W_G$ has a  block form (\ref{eq:oddbf}).

Equation (\ref{eq:EDOsystem}) for the input additive coupled cells systems admissible by $G$ can be rewritten as 
{\tiny 
\begin{equation} \label{eq:EDOsystem_odd}
\dot{x}_j=g(x_j) + \sum_{t=1}^q \left( \sum_{i \in P_t} {w_{ji}h\left(x_j,x_i\right)} +  \sum_{i \in \overline{P}_t} {w_{ji}h\left(x_j,x_i\right)} \right) +  
\sum_{t=q+1}^p \left( \sum_{i \in P_t} w_{ji}h \left(x_j,x_i\right) \right) 
+ \sum_{i \in P_0} {w_{ji}h\left(x_j,x_i\right)},  
\end{equation}}
for $j=1, \ldots, n$.

For any given choice of total phase space $\left(\R^k\right)^n$, assume that the generalized polydiagonal subspace $\Delta_{\mathcal{P}}$ is left invariant under the flow of every system in $I_{G,odd}$. Take $k=1$ and $h(x,y) = x-y$ and note that, by Proposition~\ref{prop:gen_form_h} (ii), in $I_{G,odd}$ we have $h(u,u)=0$. Then, in the restriction to $\Delta_{\mathcal{P}}$, we have the following: \\
\noindent (i) For $j,k \in P_r$, for $r \in \{1, \ldots, p\}$, we have $\dot{x}_j = \dot{x}_k$. Thus, since $h(u,u)=0$, we have $ \sum_{i \in P_t} w_{ji} = \sum_{i \in P_t} w_{ki}$ and $\sum_{i \in \overline{P}_s} w_{ji} = \sum_{i \in \overline{P}_s} w_{ki}$, for $t \ne r$, 
$1\leq t \leq p$, $1\leq s \leq q$, and $ \sum_{i \in P_0} w_{ji} = \sum_{i \in P_0} w_{ki}$. That is,  
the block matrices, $Q_{ij}$ if $i \not= j$, $R_{ij}$ and $Z_{i0}$, in (\ref{eq:oddbf}), are regular.\\
\noindent (ii) Analogously, taking  $j,k \in \overline{P}_r$, for  $r \in \{1, \ldots, q\}$, we conclude that the block matrices, 
$\overline{Q}_{ij}$ if $i \not= j$, $\overline{R}_{ij}$ and $\overline{Z}_{i0}$, in (\ref{eq:oddbf}), are regular.\\
\noindent (iii) For $j \in P_r$ and $k \in \overline{P}_r$, for $r \in \{1, \ldots, q\}$, we have $\dot{x}_j = - \dot{x}_k$. Thus, since $h$ is odd, we have $\sum_{i \in P_t} w_{ji} = \sum_{i \in \overline{P}_t} w_{ki}$ and $\sum_{i \in \overline{P}_t} w_{ji} = \sum_{i \in P_t} w_{ki}$, for $t \ne r$.  That is, for $i,j \in \{1, \ldots, q \}$, we have ${\mathrm rs}\left( Q_{ij} \right) = {\mathrm rs}\left( \overline{Q}_{ij} \right)$ if  $i \ne j$, and ${\mathrm rs}\left( R_{ij} \right) = {\mathrm rs}\left( \overline{R}_{ij} \right)$.\\
\noindent (iv) For $j \in P_0$, we have $\dot{x}_j = 0$. Thus, since $h$ is odd, we have $\sum_{i \in P_t} w_{ji} = \sum_{i \in \overline{P}_t} w_{ji}$, 
that is, ${\mathrm rs}\left(Z_{0t} \right) = {\mathrm rs}\left( \overline{Z}_{0t} \right)$ for $1\leq t \leq q$, and 
$\sum_{i \in P_t} w_{ji} = 0$, that is, ${\mathrm rs}\left(Z_{0t} \right) =0$  for $q+1 \leq t \leq p$.

We conclude that, if the generalized polydiagonal subspace $\Delta_{\mathcal{P}}$ is left invariant under the flow of every system in $I_{G,odd}$ then the tagged partition $\mathcal{P}$ is odd-balanced for $G$.

Now, assume that the tagged partition $\mathcal{P}$ is odd-balanced for $G$ and consider the input additive coupled cells systems admissible by $G$ in $I_{G,0}$. We can assume that the equations are in the form given in (\ref{eq:EDOsystem_odd}).  We have the following: \\
\noindent Conditions (a)-(b) in Definition~\ref{def:odd} of odd-balanced partition imposing the regularity of all the blocks except $Q_{ii}, \overline{Q}_{ii}, Z_{0j}, \overline{0j}, Z_{00}$ and that, each pair of blocks of the type, $Q_{ij},\, \overline{Q}_{ij}$ if $i \ne j$, $R_{ij},\, \overline{R}_{ij}$, and $Z_{i0},\, \overline{Z}_{i0}$ are both regular of the same valency, since $h(u,u)=0$, imply that, given an initial condition in $\Delta_{\mathcal{P}}$, the equations for cells  in the same part $P_r$ for $1\leq r \leq p$, or $\overline{P}_s$ for $1 \leq s \leq q$, are equal. Moreover, they imply that the equations for cells in a part $P_s$ are symmetric to the  equations for cells in its counterpart $\overline{P}_s$ for $1 \leq s \leq q$,  with the additional condition of $g$ and $h$ being odd.\\
\noindent Conditions (c)-(d)  in Definition~\ref{def:odd} of odd-balanced partition imposing that the blocks $Z_{0j},\, \overline{Z}_{0j}$ satisfy ${\mathrm rs}(Z_{0j}) = {\mathrm rs}\left(\overline{Z}_{0j}\right)$, for $0 < j \leq q$, and  ${\mathrm rs}(Z_{0j})=0$ for $q+1 \leq j \leq p$, imply that, given an initial condition in $\Delta_{\mathcal{P}}$, the equations for cells  in the part $P_0$ are null, with the additional condition of $g$ and $h$ being odd.\\
We conclude then that, if a tagged partition $\mathcal{P}$  is odd-balanced for $G$ then $\Delta_{\mathcal{P}}$ is left invariant under the flow of every system in $I_{G,odd}$.
\end{proof}

\begin{exam}\normalfont   Returning to the network on the left of Figure~\ref{fig:odd3cell},  we have that equations (\ref{eq:EDOsystem}) for $G$ where $f \in I_{G,odd}$ take the form
$$
\left\{ 
\begin{array}{l}
\dot{x}_1 = g(x_1) + h(x_1,x_2) + h(x_1, x_3)\\
\dot{x}_2 = g(x_2) + 2h(x_2,x_1)\\
\dot{x}_3 = g(x_3) + 2h(x_3,x_1)
\end{array}
\right. 
$$
where $g,h$ are odd and $h(x,x) = 0$. In Example~\ref{ex:simples}, we have seen that the tagged partition $\mathcal{P} = \left\{[1] = P_1 = \{1\},\, -[1] = \overline{P}_1 = \{2,3\} \right\}$ is odd-balanced. Restricting any such system to  the generalized polydiagonal $\Delta_{\mathcal{P}} = \{x:\, x_2 = -x_1,\, x_3  = -x_1\}$, we obtain
$$
\dot{x}_1 = g(x_1) + 2h(x_1,-x_1)\, .
$$
The symbolic network $Q_1$ at the center of Figure~\ref{fig:odd3cell}, as described in Definition~\ref{def:quo_odd}, represents this restricted system where the cell $-[1]$ represents the negative state of the cell $[1]$. 
\hfill $\Diamond$
\end{exam}

From  Proposition~\ref{prop:odd_iff} and using the symbolic quotient defined in Definition~\ref{def:quo_odd} for  an odd-balanced tagged partition, it follows the following proposition:
	
\begin{prop}
Given  an $n$-cell network $G$, an odd-balanced tagged partition $\mathcal{P}$ on the network set of cells, and an enumeration of cells adapted to $\mathcal{P}$ providing a block structure (\ref{eq:oddbf}) of the adjacency matrix $W_G$, we have that any coupled cell system in $I_{G,odd}$ restricted to the generalized polydiagonal $\Delta_{\mathcal{P}}$ is consistent with the symbolic quotient defined in Definition~\ref{def:quo_odd} where cells representing the classes $\overline{P}_i \equiv -P_i$ correspond to the negative states of the cells representing the classes $P_i$. Moreover, the cell representing the class $P_0$  corresponds to the zero state. More precisely, it has the following  the form. Denoting coordinates on $\Delta_{\mathcal{P}}$ by $(y_1, \ldots, y_p)$ where $y_j = x_k$ for (all) $ k \in P_j$, where $j=1, \ldots, p$, the restriction of (\ref{eq:EDOsystem}) to $\Delta_{\mathcal{P}}$ where $f \in I_{G,odd}$  is given by:
\begin{equation} 
\dot{y}_j=g(y_j) +\sum_{i=1, i \not= j}^p q_{ji}h\left(y_j,y_i\right) + \sum_{i=1}^{q} r_{ji}h\left(y_j,-y_i\right) +  z_{j0}h\left(y_j,0\right)  \quad \left( j=1, \ldots, p\right)\, .
\label{eq:oddrestEDO}
\end{equation}
Here, $q_{ji}$ (resp. $r_{ji}$) represents the valency of the regular matrix $Q_{ji}$ (resp. $R_{ji}$) and $z_{j0}$ the valency of the regular matrix $Z_{j0}$. 
\end{prop}

\section{Linear-balanced partitions and anti-synchrony  in the class of the linear-input-additive coupled cell systems}\label{sec_lin_bal}

A non-standard generalized polydiagonal left invariant under the flow of every  linear-input-additive coupled cell system admissible by a weighted network $G$ is an {\it anti-synchrony subspace} of $G$. We show next that these  anti-synchrony subspaces of $G$ are the 
non-standard  generalized polydiagonals associated with the linear-balanced tagged partitions of $G$. This result is an extension,
to weighted coupled cell networks and input additive coupled cell systems, of Theorem 4.21 in Neuberger {\it et al.}~\cite{NSS19}. 

Recall that in $I_{G,l}$, from Proposition~\ref{prop:gen_form_h} (iii), we have  for $a,b \in \R^k$, 
		$$ h(a,a) = 0,\quad  h(a,-a) = 2 h(a,0),\quad h(a, \pm b) = h(a,0) \overline{+} h(b,0)\, . $$
		
\begin{prop} \normalfont \label{prop:IGL}
Let $G$ be a weighted network and $\mathcal{P}$ a tagged partition of its set of cells 
which is not standard.  
The tagged partition $\mathcal{P}$  is linear-balanced for $G$ if and only if the generalized polydiagonal $\Delta_{\mathcal{P}}$ is left invariant under the flow of every system in $I_{G,l}$, for any given choice of total phase space $\left(\R^k\right)^n$.
\end{prop}

\begin{proof}  
Let $G$ be an $n$-cell weighted network with set of cells $C$, adjacency matrix $W_G$ and Laplacian $L_G$.  Consider a  tagged partition $\mathcal{P}$ of $C$ with parts $P_1, P_2, \ldots, P_p,$  counterparts $\overline{P}_1, \overline{P}_2, \ldots, \overline{P}_q$, zero part $P_0$ and the corresponding generalized polydiagonal $\Delta_{\mathcal{P}}$. 

Assume  $\Delta_{\mathcal{P}}$ is left invariant under the flow of every system in $I_{G,l}$. In  (\ref{eq:EDOsystem}), assume   $k=1$. By Proposition~\ref{prop:reg_nreg_L_W}, the space $\Delta_{\mathcal{P}}$ is left invariant under $L_G$. By Definition~\ref{def:linear}, we have that $\mathcal{P}$ is linear-balanced.

Assume now that the tagged partition $\mathcal{P}$ is linear-balanced for $G$ and consider an enumeration of the cells of $G$ adapted to $\mathcal{P}$ so that 
the adjacency matrix $W_G$ of $G$ has  a block structure (\ref{eq:oddbf}). By Definition~\ref{def:linear}, for $k=1$ the space 
$\Delta_{\mathcal{P}}$ is  left invariant by the matrix $L_G = D_G - W_G$, which is equivalent to  the entries of $W_G$ satisfy the conditions in Corollary~\ref{cor:mainLaplacian}. 
Consider an  additive coupled cell system in $I_{G,l}$, with equations
\begin{equation} 
{\tiny 
\begin{array}{rcl}
\dot{x}_i & = & g(x_i) + \displaystyle \sum_{t=1}^q \left( \sum_{j \in P_t} {w_{ij}h\left(x_i,x_j\right)} +  \sum_{m \in \overline{P}_t} {w_{im}h\left(x_i,x_m\right)} \right) + \displaystyle \sum_{t=q+1}^p \left( \sum_{j \in P_t} w_{ij} h\left(x_i,x_j\right)\right) + \sum_{j \in P_0} {w_{ij}h\left(x_i,x_j\right)}, 
\end{array}}
\label{eq:EDOsystem_Lag_lin}
\end{equation}
for $i=1,\ldots,n$, where $g$ is odd and $h$ is linear.
Consider coordinates $\left(y_1, \ldots, y_p\right)$ in $\Delta_{\mathcal{P}}$ where: for $1 \leq t \leq q$, we take $y_t = x_j = -x_m$ for all $j \in P_t$ and $m \in \overline{P}_t$; 
for $q+1 \leq t  \leq p$, we have  $y_t = x_j$ for all $j \in P_t$; also, $x_j = 0$ for all $j \in P_0$. We have so $h(y_t,y_t) = 0$ 
for all $1 \leq t \leq p$ and $h(y_t,-y_t) = 2 h(y_t,0)$ for $1 \leq t \leq q$; also,  if $l\not= t$, we have 
$h(y_l,\pm y_t) = h(y_l,0) \overline{+}  h(y_t,0)$ and $h(-y_l,\pm y_t) = - h(y_l,0) \overline{+} h(y_t,0)$.

In (\ref{eq:EDOsystem_Lag_lin}), if $i \in P_l$ for $1\leq l \leq q$ and $x \in \Delta_{\mathcal{P}}$, using  conditions (i)-(ii) in Corollary~\ref{cor:mainLaplacian} and corresponding notation, we obtain:\\
$
{\tiny 
\begin{array}{rcl}
& & g(x_i) + \displaystyle \sum_{t=1}^q \left( \sum_{j \in P_t} {w_{ij}h\left(x_i,x_j\right)} +  \sum_{m \in \overline{P}_t} {w_{im}h\left(x_i,x_m\right)} \right) + \displaystyle \sum_{t=q+1}^p \left( \sum_{j \in P_t} w_{ij} h\left(x_i,x_j\right)\right) + \sum_{j \in P_0} {w_{ij}h\left(x_i,x_j\right)}
\end{array}}
$\\
$
{\tiny 
\begin{array}{rcl}
& = & g(y_l) + \displaystyle \sum_{t=1, t\not= l}^q \left( h\left(y_l,y_t\right) \sum_{j \in P_t} w_{ij}+  h\left(y_l,-y_t\right) \sum_{m \in \overline{P}_t} w_{im} \right) \\
 & & \\
 & & \quad      \displaystyle    +  h\left(y_l,y_l\right) \sum_{j \in P_l} w_{ij} +  h\left(y_l,-y_l\right) \sum_{m \in \overline{P}_l} w_{im}   +   
\displaystyle \sum_{t=q+1}^p  h(y_l, y_t) \left( \sum_{j \in P_t} w_{ij} \right) +
 h\left(y_l,0\right) \sum_{j \in P_0} w_{ij} 
 \end{array}}
 $\\
 $
 {\tiny 
 \begin{array}{rcl}
 & = & g(y_l) + \displaystyle \sum_{t=1, t\not=l}^q \left( \left( h(y_l,0) -  h(y_t,0)\right) \sum_{j \in P_t} w_{ij} +  \left( h(y_l,0) + h(y_t,0) \right) \sum_{m \in \overline{P}_t} w_{im}    \right) \\
 & & \\
  & & \quad     +  \displaystyle 2 h\left(y_l,0\right)  \sum_{m \in \overline{P}_l} w_{im} +  \displaystyle \sum_{t=q+1}^p  \left( h(y_l,0) -  h(y_t,0)\right) \left( \sum_{j \in P_t} w_{ij} \right) +     h\left(y_l,0\right) \sum_{j \in P_0} w_{ij} 
  \end{array}}
  $\\
  $
  {\tiny 
  \begin{array}{rcl}
  & = & g(y_l) + h(y_l,0) \displaystyle \left( \sum_{t=1, t\not=l}^q  \left( \sum_{j \in P_t} w_{ij}  + \sum_{m \in \overline{P}_t} w_{im}\right)  + 2 \sum_{m \in \overline{P}_l} w_{im}  
\displaystyle + \sum_{t=q+1}^p   \left( \sum_{j \in P_t} w_{ij} \right)  +  \sum_{j \in P_0} w_{ij} \right)   \\
  & & \\
   & & \quad +  \displaystyle \sum_{t=1, t\not=l}^q h(y_t,0) \left( - \sum_{j \in P_t} w_{ij}  + \sum_{m \in \overline{P}_t} w_{im} \right)  + \displaystyle \sum_{t=q+1}^p  \left(   h(y_t,0)\right) \left(  - \sum_{j \in P_t} w_{ij} \right)   
   \end{array}}
   $\\
   $
   {\tiny 
   \begin{array}{rcl}
    & = &  g(y_l) + h(y_l,0)  \displaystyle \left( \sum_{t=1, t\not= l}^{q} \left[ {\mathrm rs}\left(Q_{lt}\right) +  {\mathrm rs}\left(R_{lt}\right) \right] +  2  {\mathrm rs}\left(R_{ll}\right)  + 
  \sum_{t=q+1}^{p}   {\mathrm rs}\left(Q_{lt}\right)   +    {\mathrm rs}\left( Z_{l0} \right) \right)_i\\
    & & \\
    & & \quad + \displaystyle \sum_{t=1, t\not=l}^q h(y_t,0)  \left(- {\mathrm rs}(Q_{lt}) +  {\mathrm rs}(R_{lt}) \right)_i  
    +  \displaystyle \sum_{t=q+1}^p h(y_t,0)  \left(- {\mathrm rs}(Q_{lt})\right)_i   \\
     & & \\
  & = & g(y_l) +  h(y_l,0) r_l +  \displaystyle \sum_{t=1, t\not= l}^p h(y_t,0) q_{lt} \, .
 \end{array}}
$\\

Recall that,  for $1\leq l \leq q$, the column matrices  $- {\mathrm rs}(Q_{lt}) +  {\mathrm rs}(R_{lt})$ for $t=1, \ldots, q$, $t\not= l$ and  $-{\mathrm rs}(Q_{lt})$, for $t=q+1, \ldots, p$ are 
regular of valency $q_{lt}$. Also,  $\sum_{t=1, t \not= l}^{p}   {\mathrm rs}\left(Q_{lt}\right) +  \sum_{t=1, t\not= l}^{q}  {\mathrm rs}\left(R_{lt}\right) +  2  {\mathrm rs}\left(R_{ll}\right)  + 
   {\mathrm rs}\left( Z_{l0} \right)$ is regular of valency $r_l$. 
    
Similarly, in (\ref{eq:EDOsystem_Lag_lin}), if $i \in \overline{P}_l$ for $1\leq l \leq q$ and $x \in \Delta_{\mathcal{P}}$, using  conditions (i)-(ii) in Corollary~\ref{cor:mainLaplacian} and corresponding notation, we obtain:\\
$
{\tiny 
\begin{array}{l}
 g(x_i) + \displaystyle \sum_{t=1}^q \left( \sum_{j \in P_t} {w_{ij}h\left(x_i,x_j\right)} +  \sum_{m \in \overline{P}_t} {w_{im}h\left(x_i,x_m\right)} \right)  + \displaystyle \sum_{t=q+1}^p \left( \sum_{j \in P_t} w_{ij} h\left(x_i,x_j\right)\right) +  \sum_{j \in P_0} {w_{ij}h\left(x_i,x_j\right)}\\
\end{array}}
$\\
$
{\tiny 
\begin{array}{rcl}
  & = & \displaystyle  -g(y_l) -  h(y_l,0) r_l -  \sum_{t=1, t\not= l}^p  h(y_t,0) q_{lt}\, .
   \end{array}}
$

In (\ref{eq:EDOsystem_Lag_lin}), if $i \in P_l$ for $ l > q$ and $x \in \Delta_{\mathcal{P}}$, using  conditions (iii)-(iv) in Corollary~\ref{cor:mainLaplacian} and corresponding notation, we obtain:\\
$
{\tiny 
\begin{array}{l}
 g(x_i) + \displaystyle \sum_{t=1}^q \left( \sum_{j \in P_t} {w_{ij}h\left(x_i,x_j\right)} +  \sum_{m \in \overline{P}_t} {w_{im}h\left(x_i,x_m\right)} \right) + \displaystyle \sum_{t=q+1}^p \left( \sum_{j \in P_t} w_{ij} h\left(x_i,x_j\right)\right) +  \sum_{j \in P_0} {w_{ij}h\left(x_i,x_j\right)}
\end{array}}
$\\
$
{\tiny 
\begin{array}{rcl}
& = & g(y_l) + \displaystyle \sum_{t=1}^q \left( h\left(y_l,y_t\right) \sum_{j \in P_t} w_{ij}+  h\left(y_l,-y_t\right) \sum_{m \in \overline{P}_t} w_{im} \right) \\
 & & \\
 & & \quad      \displaystyle    +  h\left(y_l,y_l\right) \sum_{j \in P_l} w_{ij} + 
\displaystyle \sum_{t=q+1, t\not=l}^p  h(y_l, y_t) \left( \sum_{j \in P_t} w_{ij} \right) +  
 h\left(y_l,0\right) \sum_{j \in P_0} w_{ij} 
 \end{array}}
 $\\
 $
 {\tiny 
 \begin{array}{rcl}
 & = & g(y_l) + \displaystyle \sum_{t=1}^q \left( \left( h(y_l,0) -  h(y_t,0)\right) \sum_{j \in P_t} w_{ij} +  \left( h(y_l,0) + h(y_t,0) \right) \sum_{m \in \overline{P}_t} w_{im}    \right) \\
 & & \\
  & & \quad     +   \displaystyle \sum_{t=q+1, t\not= l}^p  \left( h(y_l,0) -  h(y_t,0)\right) \left( \sum_{j \in P_t} w_{ij} \right) +                             h\left(y_l,0\right) \sum_{j \in P_0} w_{ij} 
  \end{array}}
  $\\
  $
  {\tiny 
  \begin{array}{rcl}
  & = & g(y_l) + h(y_l,0) \displaystyle \left( \sum_{t=1}^q  \left( \sum_{j \in P_t} w_{ij}  + \sum_{m \in \overline{P}_t} w_{im}\right)  
 \displaystyle + \sum_{t=q+1, t\not=l}^p   \left( \sum_{j \in P_t} w_{ij} \right) +  \sum_{j \in P_0} w_{ij} \right)   \\
  & & \\
   & & \quad +  \displaystyle \sum_{t=1}^q h(y_t,0) \left( - \sum_{j \in P_t} w_{ij}  + \sum_{m \in \overline{P}_t} w_{im} \right)  + \displaystyle \sum_{t=q+1, t\not=l}^p  \left(   h(y_t,0)\right) \left(  - \sum_{j \in P_t} w_{ij} \right)   
   \end{array}}
   $\\
   $
   {\tiny 
   \begin{array}{rcl}
    & = &  g(y_l) + h(y_l,0)  \displaystyle \left( \sum_{t=1}^{q} \left[ {\mathrm rs}\left(Q_{lt}\right) +  {\mathrm rs}\left(R_{lt}\right) \right] + 
   \sum_{t=q+1, t\not=l}^{p}   {\mathrm rs}\left(Q_{lt}\right)   +    {\mathrm rs}\left( Z_{l0} \right) \right)_i\\
    & & \\
    & & \quad + \displaystyle \sum_{t=1}^q h(y_t,0)  \left(- {\mathrm rs}(Q_{lt}) +  {\mathrm rs}(R_{lt}) \right)_i  
    +  \displaystyle \sum_{t=q+1, t\not=l}^p h(y_t,0)  \left(- {\mathrm rs}(Q_{lt})\right)_i   \\
     & & \\
  & = & g(y_l) +  h(y_l,0) r_l +  \displaystyle \sum_{t=1, t\not= l}^p h(y_t,0) q_{lt} \, .
 \end{array}}
$\\
Recall that,  for $p \geq l >q$, the column matrices  $- {\mathrm rs}(Q_{lt}) +  {\mathrm rs}(R_{lt})$ for $t=1, \ldots, q$, and  $-{\mathrm rs}(Q_{lt})$, for $t=q+1, \ldots, p$, where $t \not= l$, are 
regular of valency $q_{lt}$. Also,  $\sum_{t=1, t \not= l}^{p}   {\mathrm rs}\left(Q_{lt}\right) +  \sum_{t=1}^{q}  {\mathrm rs}\left(R_{lt}\right)  
    +    {\mathrm rs}\left( Z_{l0} \right)$ is regular of valency $r_l$.

In (\ref{eq:EDOsystem_Lag_lin}), if $i \in P_0$ and $x \in \Delta_{\mathcal{P}}$, using  conditions (v)-(vi) in Corollary~\ref{cor:mainLaplacian} and corresponding notation, 
recalling that $h(0,0)=0$ and $g(0) =0$, we obtain:\\
$
{\tiny 
\begin{array}{rcl}
& & g(x_i) + \displaystyle \sum_{t=1}^q \left( \sum_{j \in P_t} {w_{ij} h\left(x_i,x_j\right)} +  \sum_{m \in \overline{P}_t} {w_{im}h\left(x_i,x_m\right)} \right) 
  \displaystyle + \sum_{t=q+1}^p \left(\sum_{j \in P_t} w_{ij} h\left(x_i,x_j\right)\right)  +  \sum_{j \in P_0} w_{ij} h\left(x_i,x_j\right) 
\end{array}}
$

\noindent 
$
{\tiny 
\begin{array}{rcl}
& = & \displaystyle  \sum_{t=1}^q \left( h\left(0,y_t\right) \sum_{j \in P_t} w_{ij}+  h\left(0,-y_t\right) \sum_{m \in \overline{P}_t} w_{im} \right)     \displaystyle + \sum_{t=q+1}^p h\left(0,y_t\right) \left(\sum_{j \in P_t} w_{ij} \right)  
\end{array} }
$\\
$
{\tiny 
\begin{array}{rcl}
& = &   \displaystyle  \sum_{t=1}^p h\left(0,y_t\right) \left( \sum_{j \in P_t} w_{ij}-  \sum_{m \in \overline{P}_t} w_{im} \right)  
+ \displaystyle  \sum_{t=q+1}^p h\left(0,y_t\right) \left( \sum_{j \in P_t} w_{ij} \right)\\
 \end{array}}
 $\\
 $
{\tiny 
\begin{array}{rcl}
 & = & \displaystyle  \sum_{t=1}^q h\left(0,y_t\right)  \left( {\mathrm rs}(Z_{0t})  - {\mathrm rs}\left(\overline{Z}_{0t}\right)\right)_i   
+  \sum_{t=q+1}^p h\left(0,y_t\right)  \left( {\mathrm rs}(Z_{0t}) \right)_i      =  0\, .
    \end{array}}
$

We have so that $\Delta_{\mathcal{P}}$ is invariant under the flow of the additive coupled cell system with equations (\ref{eq:EDOsystem_Lag_lin}).  That is, $\Delta_{\mathcal{P}}$ is left invariant under the flow of every system in $I_{G,l}$.
\end{proof}

From  Proposition~\ref{prop:IGL} and using  the notation of  the symbolic quotient of $G$ by a  linear-balanced tagged  partition  determined by the matrix (\ref{eq:quolin}) in Definition~\ref{def:quo_odd},  it follows the following proposition:

\begin{prop}
Let $G$ be  an $n$-cell network, $\mathcal{P}$ a linear-balanced tagged  partition on the set of cells of $G$ with parts $P_1, \ldots, P_p$, counterparts $\overline{P}_1, \ldots, \overline{P}_q$ and zero part $P_0$, and consider an enumeration of the set of cells adapted to  $\mathcal{P}$ providing a block structure (\ref{eq:oddbf}) of the adjacency matrix $W_G$. Consider the symbolic quotient of $G$ by $\mathcal{P}$ determined by the matrix (\ref{eq:quolin}) in Definition~\ref{def:quo_odd}. Denoting coordinates on $\Delta_{\mathcal{P}}$ by $(y_1, \ldots, y_p)$ where $y_i = x_k$ for (all) $ k \in P_i$, the restriction of (\ref{eq:EDOsystem}) to $\Delta_{\mathcal{P}}$ where $f \in I_{G,l}$  is given by:
\begin{equation} 
\dot{y}_i=g(y_i) +\sum_{j=1, j \not= i}^p q_{ij}h\left(y_j,0\right) + r_{i}h\left(y_i,0\right)  \quad \left( i=1,\ldots, p\right)\, .
\label{eq:linrestEDO}
\end{equation}
\end{prop}
		
\begin{exams} \normalfont  Consider the isomorphic six-cell networks in Figure~\ref{f:2ndsixSwift} and the linear-balanced partitions of Examples~\ref{exs:linear}. For the linear-balanced  partition of the network set of cells in Examples~\ref{exs:linear}~(i) with parts $P_1 =\{1,2\}, \, P_2 = \{ 3\}$ and counterparts $\overline{P}_1 =\{4,5\}, \, \overline{P}_2 = \{ 6\}$,  we have that any coupled cell system in $I_{G,l}$ restricted to 
  $$
  \Delta_{\mathcal{P}} = \{ (y_1,y_1, y_2, -y_1, -y_1, -y_2):\, y_1, y_2 \in \R^k\} 
  $$ 
  has the following  the form:
\begin{equation} 
\left\{
\begin{array}{l}
\dot{y}_1  =g(y_1) +  2 h(y_1,0)\\
\\
\dot{y}_2  =g(y_2) +  2 h(y_2,0)
\end{array}\, .
\right.
\label{eq:linexemplo}
\end{equation} \\
Consider now the linear-balanced partition of the network set of cells in Examples~\ref{exs:linear}~(ii) with one part $P_1 =\{1,2\}$, one counterpart $ \overline{P}_1 = \{ 3, 4\}$ and the zero part $P_0 =\{5,6\}$.  We have that any coupled cell system in $I_{G,l}$ restricted to 
  $$
  \Delta_{\mathcal{P}} = \{ (y_1,y_1,- y_1, -y_1, 0,0):\, y_1 \in \R^k\} 
  $$ 
  has the following  the form:
\begin{equation} 
\begin{array}{l}
\dot{y}_1  =g(y_1) +  2 h(y_1,0)\, .
\end{array}
\label{eq:2linexemplo}
\end{equation}
\hfill $\Diamond$ 
\end{exams}

\section{Even-odd-balanced partitions and anti-synchrony  in the class of the even-odd-input-additive coupled cell systems}\label{sec_eo_bal}

A non-standard generalized polydiagonal left invariant under the flow of every  even-odd-input-additive  coupled cell system admissible by a weighted network $G$ is an {\it anti-synchrony subspace} of $G$. We show next that these  anti-synchrony subspaces of $G$ are the non-standard  generalized polydiagonals associated with the even-odd-balanced partitions of $G$. 

Recall that in $I_{G,eo}$, 	$g$ is odd and the coupling function $h(x,y)$ is even in $x$ and odd in $y$. It follows in particular that $h(x,0)=0$ for all $x$. Also, we have that 
$h(-x,-y) = h(x,-y) = -h(x,y) = -h(-x,y)$ for all $x,y$. 
		
\begin{prop} \normalfont \label{prop:eoia}
Let $G$ be a weighted network and $\mathcal{P}$ a tagged partition of its set of cells 
which is not standard.  
The partition $\mathcal{P}$  is even-odd-balanced for $G$ if and only if the generalized polydiagonal $\Delta_{\mathcal{P}}$ is left invariant under the flow of every system in $I_{G,eo}$, for any given choice of total phase space $\left(\R^k\right)^n$.
\end{prop}		
		
\begin{proof}  Let $G$ be an $n$-cell weighted network with set of cells $C$ and adjacency matrix $W_G$. Consider a  tagged partition $\mathcal{P}$ of $C$ with parts $P_1, P_2, \ldots, P_p,$  counterparts $\overline{P}_1, \overline{P}_2, \ldots, \overline{P}_q$, zero part $P_0$ and the corresponding generalized polydiagonal $\Delta_{\mathcal{P}}$.

Assume  $\Delta_{\mathcal{P}}$ is left invariant under the flow of every system in $I_{G,eo}$. In particular, by Proposition~\ref{prop:reg_nreg_L_W}, the coupled cell system $\dot{x} = W_G \, x$ where $x \in \R^n$, is even-odd-input-aditive. Thus, the space $\Delta_{\mathcal{P}}$ is left invariant under $W_G$. By Definition~\ref{def:linear}, we have that $\mathcal{P}$ is even-odd-balanced.

Assume now that the partition $\mathcal{P}$ is even-odd-balanced for $G$ and consider an enumeration of the cells of $G$ adapted to $\mathcal{P}$ so that 
the adjacency matrix $W_G$ of $G$ has  a block structure (\ref{eq:oddbf}). By Definition~\ref{def:linear}, for $k=1$ the space 
$\Delta_{\mathcal{P}}$ is  left invariant by the matrix $W_G$, which is equivalent to  the entries of $W_G$ satisfying the conditions in Proposition~\ref{thm:mainLaplacian}. 
Consider an  additive coupled cell system in $I_{G,eo}$, with equations
\begin{equation} 
{\tiny 
\begin{array}{rcl}
\dot{x}_i & = & g(x_i) + \displaystyle \sum_{t=1}^q \left( \sum_{j \in P_t} {w_{ij}h\left(x_i,x_j\right)} +  \sum_{m \in \overline{P}_t} {w_{im}h\left(x_i,x_m\right)} \right)  + \displaystyle \sum_{t=q+1}^p \left( \sum_{j \in P_t} w_{ij} h\left(x_i,x_j\right)\right) + \sum_{j \in P_0} {w_{ij}h\left(x_i,x_j\right)}, 
\end{array}}
\label{eq:EDOsystem_Lag_eo}
\end{equation}
for $i=1,\ldots,n$, where $g$ is odd and $h$ is even in $x$ and odd in $y$. 
Consider coordinates $\left(y_1, \ldots, y_p\right)$ in $\Delta_{\mathcal{P}}$ where: for $1 \leq t \leq q$, we take $y_t = x_j = -x_m$ for all $j \in P_t$ and $m \in \overline{P}_t$;  for $q+1 \leq t  \leq p$, we have  $y_t = x_j$ for all $j \in P_t$; also, $x_j = 0$ for all $j \in P_0$. We have so $h(y_t,0) = 0$. The proof now follows as in the proof of Proposition~\ref{prop:IGL}. As an  example, note that in (\ref{eq:EDOsystem_Lag_eo}), if $i \in P_l$ for $1\leq l \leq q$ and $x \in \Delta_{\mathcal{P}}$, we obtain:\\
$
{\tiny 
\begin{array}{rcl}
& & g(x_i) + \displaystyle \sum_{t=1}^q \left( \sum_{j \in P_t} {w_{ij}h\left(x_i,x_j\right)} +  \sum_{m \in \overline{P}_t} {w_{im}h\left(x_i,x_m\right)} \right) + \displaystyle \sum_{t=q+1}^p \left( \sum_{j \in P_t} w_{ij} h\left(x_i,x_j\right)\right) + \sum_{j \in P_0} {w_{ij}h\left(x_i,x_j\right)}
\end{array}}
$\\
$
{\tiny 
\begin{array}{rcl}
& = & g(y_l) + \displaystyle \sum_{t=1}^q \left( h\left(y_l,y_t\right) \sum_{j \in P_t} w_{ij}+  h\left(y_l,-y_t\right) \sum_{m \in \overline{P}_t} w_{im} \right)   +   
  \displaystyle \sum_{t=q+1}^p  \left( h(y_l, y_t) \sum_{j \in P_t} w_{ij} \right) + h\left(y_l,0\right) \sum_{j \in P_0} w_{ij} 
 \end{array}}
 $\\
 $
{\tiny 
\begin{array}{rcl}
& = & g(y_l) + \displaystyle \sum_{t=1}^q \left( h\left(y_l,y_t\right) \sum_{j \in P_t} w_{ij} -  h\left(y_l,y_t\right) \sum_{m \in \overline{P}_t} w_{im} \right)    +    
 \displaystyle \sum_{t=q+1}^p   \left( h(y_l, y_t) \  \sum_{j \in P_t} w_{ij} \right)
 \end{array}}
 $\\
 $
{\tiny 
\begin{array}{rcl}
& = & g(y_l) + \displaystyle \sum_{t=1}^q h\left(y_l,y_t\right) \left( \sum_{j \in P_t} w_{ij} -  \sum_{m \in \overline{P}_t} w_{im} \right)  +    
 \displaystyle \sum_{t=q+1}^p  \left( h(y_l, y_t) \sum_{j \in P_t} w_{ij} \right) 
 \end{array}}
 $\\
 $
{\tiny 
\begin{array}{rcl}
& = & g(y_l) + \displaystyle \sum_{t=1}^q h\left(y_l,y_t\right) \left(   {\mathrm rs}\left(Q_{lt}\right) -  {\mathrm rs}\left(R_{lt}\right)  \right)_i +    
 \displaystyle \sum_{t=q+1}^p  h(y_l, y_t) \left( {\mathrm rs}\left(Q_{lt}\right) \right)_i 
 \end{array}}\\
 $
Similarly, in (\ref{eq:EDOsystem_Lag_eo}), if $i \in \overline{P}_l$ for $1\leq l \leq q$ and $x \in \Delta_{\mathcal{P}}$, we obtain:\\
$
{\tiny 
\begin{array}{rcl}
& & g(x_i) + \displaystyle \sum_{t=1}^q \left( \sum_{j \in P_t} {w_{ij}h\left(x_i,x_j\right)} +  \sum_{m \in \overline{P}_t} {w_{im}h\left(x_i,x_m\right)} \right) + \displaystyle \sum_{t=q+1}^p \left( \sum_{j \in P_t} w_{ij} h\left(x_i,x_j\right)\right) + \sum_{j \in P_0} {w_{ij}h\left(x_i,x_j\right)}
\end{array}}
$\\
 $
{\tiny 
\begin{array}{rcl}
& = & -g(y_l) - \displaystyle \sum_{t=1}^q h\left(y_l,y_t\right) \left(    {\mathrm rs}\left(\overline{Q}_{lt}\right)       - {\mathrm rs}\left(\overline{R}_{lt}\right)   \right)_i    
-    \displaystyle \sum_{t=q+1}^p  h(y_l, y_t) \left( -{\mathrm rs}\left(\overline{R}_{lt}\right) \right)_i 
 \end{array}}
 $\\
The rest of the proof follows in a similar way as in the proof of Proposition \ref{prop:IGL} using the conditions in Proposition~\ref{thm:mainLaplacian} for $W_G$. 
\end{proof}

The following proposition  describes the restrictions of coupled cell systems which are even-odd-input-additive to  anti-synchrony spaces.

\begin{prop}
Let $G$ be  an $n$-cell network, $\mathcal{P}$ an even-odd--balanced partition on the set of cells of $G$  which is not standard, 
with parts $P_1, \ldots, P_p$, counterparts $\overline{P}_1, \ldots, \overline{P}_q$ and zero part $P_0$, and consider an enumeration of the set of cells adapted to  $\mathcal{P}$ providing a block structure (\ref{eq:oddbf}) of the adjacency matrix $W_G$. Consider the symbolic quotient of $G$ by $\mathcal{P}$ determined by the matrix (\ref{eq:quoeo}) in Definition~\ref{def:quo_odd}. Denoting coordinates on $\Delta_{\mathcal{P}}$ by $(y_1, \ldots, y_p)$ where $y_i = x_k$ for (all) $ k \in P_i$, the restriction of (\ref{eq:EDOsystem}) to $\Delta_{\mathcal{P}}$ where $f \in I_{G,eo}$  is given by:
\begin{equation} 
\dot{y}_i=g(y_i) +\sum_{j=1}^p q_{ij}h\left(y_i,y_j\right)  \quad \left( i=1,\ldots, p\right)\, .
\label{eq:eorestEDO}
\end{equation}
\end{prop}

\section{The set of synchrony and anti-synchrony subspaces of a weighted network} \label{sec:algm}

In \cite{AD18}, Aguiar and Dias, extend previous results on the coupled cell networks formalism of Golubitsky, Stewart and collaborators to the setup of weighted coupled cell networks considering input additive coupled cell systems. Some of those results have to do with the polydiagonal subspaces of the network phase space, assuming one dimensional cell dynamics, that are left invariant by the network weighted adjacency matrix. These correspond to the polydiagonal subspaces that are flow invariant by all the input additive coupled cell systems that are admissible by the network structure. That is, they correspond to the synchrony subspaces of the weighted network that are given by the balanced 
 partitions  of   the network set of cells.

In \cite{AD18}, taking the results in Stewart~\cite{S07}, Aguiar and Dias conclude that the set  of the synchrony subspaces of a weighted coupled cell network, in one-to-one correspondence with the  balanced partitions of  the  set of cells of the network, is a lattice with the partial order given by inclusion and the meet operation given by intersection. Moreover, they conclude that both the characterization and the algorithm obtained in Aguiar and Dias~\cite{AD14}, where the lattice of synchrony subspaces of a network can be obtained using the eigenvalue and eigenvector structure of its adjacency matrix, apply to the weighted setup.

Here, we enlarge the set of the polydiagonal subspaces by considering the generalized polydiagonal subspaces that are left invariant by the adjacency matrix and/or the Laplacian matrix of a network. 
The 
{\it synchrony subspaces} of  a network correspond to the polydiagonal subspaces that are given by the exo-balanced partitions on the network set of cells. These are flow invariant by the exo-input-additive coupled cell systems. The subset of the synchrony subspaces that are given by the balanced 
partitions are flow invariant by the larger space of the input-additive coupled cell systems. 
The  
{\it anti-synchrony subspaces} 
of a network correspond to the generalized polydiagonal subspaces that are given by the linear-balanced and even-odd-balanced  partitions on the network set of cells. These are flow invariant by the linear-input-additive and even-odd-input-additive  coupled cell systems, respectively. 
Recall that the linear-input-additive coupled cell systems are the coupled cell systems with input additive structure where the internal function 
$g$ is odd and the coupling function $h$ is linear (and odd) and satisfies $h(x,x) \equiv 0$. The even-odd-input-additive  coupled cell systems are  the coupled cell systems with input additive structure where the internal function $g$ is odd and the coupling function $h$ is even in the first variable and odd in the second variable.
The subset of the anti-synchrony subspaces that are given by the odd-balanced 
partitions
are flow invariant by the  
space of the odd-input-additive coupled cell systems. 
Recall that the space of the odd-input-additive coupled cell systems is larger than the space  the linear-input-additive coupled cell systems as the coupling 
function does not have to be necessarily linear.

\begin{Def}\normalfont 
Given an $n$-cell weighted network $G$, denote by ${\mathcal L}_{W_G}$ and by ${\mathcal L}_{L_G}$, the set of the generalized polydiagonal subspaces that are left invariant by the adjacency matrix $W_G$ and the Laplacian matrix $L_G$ of $G$, respectively.
\hfill $\Diamond$
\end{Def}

From the results in the previous sections,  ${\mathcal L}_{W_G} \cup {\mathcal L}_{L_G}$ corresponds to the set of the synchrony and anti-synchrony subspaces of a network $G$. Moreover, we have the following result.

\begin{thm} \label{thm:final}
Let $G$ be a weighted network and consider the set  ${\mathcal L}_{W_G} \cup {\mathcal L}_{L_G}$ of the  synchrony and anti-synchrony subspaces of  $G$. 
Let $\Delta_{\footnotesize{\mathcal{P}}}$ be a subspace in ${\mathcal L}_{W_G} \cup {\mathcal L}_{L_G}$ and $\mathcal{P}$ the associated tagged partition. We have that, \\
(i) $\Delta_{\footnotesize{\mathcal{P}}}$ is a synchrony subspace for every system in \\
\hspace{0.1in} (i.a) $I_G$, $I_{G,eo}$  if and only if $\mathcal{P}$ is balanced; \\
\hspace{0.1in}(i.b) $I_{G,0}$, $I_{G,odd}$, or $I_{G,l}$ if and only if $\mathcal{P}$ is exo-balanced. \\
(ii) $\Delta_{\footnotesize{\mathcal{P}}}$ is an anti-synchrony subspace for every system in \\
(ii.a) $I_{G,odd}$ if and only if $\mathcal{P}$ is odd-balanced; \\
(ii.b) $I_{G,l}$ if and only if $\mathcal{P}$ is linear-balanced; \\
(ii.c) $I_{G,eo}$ if and only if $\mathcal{P}$ is even-odd-balanced.
\end{thm}

\begin{rem}
(i) Given an $n$-cell weighted network $G$, we have that ${\mathcal L}_{W_G}$ and ${\mathcal L}_{L_G}$, are lattices with the partial order and meet operations given by the inclusion and intersection, respectively, as happens for the lattice of polydiagonal subspaces that are left invariant by the adjacency matrix of a network (synchrony subspaces given by the balanced standard partitions). In fact, observe that the intersection of a synchrony subspace with a synchrony is a synchrony space, and the intersection of a synchrony space with an anti-synchrony subspace, or the intersection of  two anti-synchrony spaces, is an anti-synchrony subspace. \\
(ii) The set of subspaces invariant under a linear map  forms a complete lattice under the relation of inclusion. Moreover, this lattice can be described using the Jordan subspaces, the irreducible invariant subspaces having a unique eigenvector (up to multiplication by a scalar). In this lattice the meet operation is the intersection and the join operation is the sum. We can apply this to the set of all the spaces that are invariant under the network adjacency matrix and the network Laplacian matrix, to conclude that they are complete lattices  where the meet  operation is the intersection and the join operation is the sum, and both can be obtained from the corresponding Jordan subspaces. \\ 
(iii) As the Laplacian matrix $L_G$ is regular, we have that the one-dimensional diagonal space where all cell coordinates are equal belongs to ${\mathcal L}_{L_G}$. 
 However, the bottom of ${\mathcal L}_{L_G}$ is the zero space as it is always invariant under the Laplacian matrix. The same holds for  ${\mathcal L}_{W_G}$. Equivalently, any  linear-input-additive coupled cell system and any even-odd-input-additive has the zero equilibrium.\\
(iv) As mentioned in (ii), the join operation for the lattice of the subspaces which are left invariant under the network adjacency matrix or Laplacian matrix is given by the usual sum. However, analogously to happens for synchrony subspaces, the sum of two synchrony or anti-synchrony subspaces may not be a generalized polydiagonal subspace and so the join operation for the lattices ${\mathcal L}_{W_G}$ and ${\mathcal L}_{L_G}$  is not the sum. Moreover, there is no explicit form of describing the join of two synchrony or anti-synchrony subspaces. 
Thus both lattices ${\mathcal L}_{W_G}$ and ${\mathcal L}_{L_G}$ are subsets of the lattices of the invariant subspaces under the network adjacency matrix and the network Laplacian matrix, respectively, but are not sublattices.
\hfill $\Diamond$ 
\end{rem}

From Proposition~\ref{prop:subset}, when $G$ is a regular network, ${\mathcal L}_{W_G} = {\mathcal L}_{L_G}$. If $G$ is not regular, then in general, 
${\mathcal L}_{W_G} \not= {\mathcal L}_{L_G}$, moreover, neither ${\mathcal L}_{W_G},\ {\mathcal L}_{L_G}$ is strictly included in the other.

%
%

The work in Aguiar and Dias~\cite{AD14} extends in a natural way, by considering generalized polydiagonal subspaces and using the eigenvalue and eigenvector structures of the adjacency matrix $W_G$ and the Laplacian matrix $L_G$, to obtain the lattices ${\mathcal L}_{W_G}$ and ${\mathcal L}_{L_G}$ and, thus, to obtain the set of the synchrony and anti-synchrony  subspaces of a network $G$.
Although the lattices  join operation is not given by the sum, as in \cite{AD14}, we can conclude that, for each lattice, there is a subset of synchrony and  anti-synchrony subspaces, called {\it minimal}, with the property of every remaining synchrony or anti-synchrony subspace in the lattices ${\mathcal L}_{W_G}$ and ${\mathcal L}_{L_G}$ being a sum of subspaces in that minimal subset. Each minimal synchrony or anti-synchrony subspace of ${\mathcal L}_{W_G}$ (${\mathcal L}_{L_G}$) is associated to an eigenvector or a set of generalized eigenvectors of $W_G$ ($L_G$).  
We have then that the Algorithm 6.5 in \cite{AD14}, for networks with only one edge-type, can be easily adapted to find the lattices ${\mathcal L}_{W_G}$ and ${\mathcal L}_{L_G}$ for a weighted network $G$ and, thus, to find the set of synchrony and anti-synchrony subspaces of $G$.  In fact, the only step of the algorithm which needs adaptation is the first one where, for each eigenvalue $\lambda$ of the matrix, the table that is constructed, besides the polydiagonal subspaces, must also contain all the generalized polydiagonal  subspaces, for the eigenvectors and Jordan chains in the generalized eigenspace for $\lambda$. This step relies on Lemma 6.1 in  \cite{AD14}, which generalizes easily for the case where, besides conditions of the form $x_{l_1} = x_{l_2}$, we have also equalities of the from $x_{l_1} = -x_{l_2}$ or $x_{l_1} = 0$.

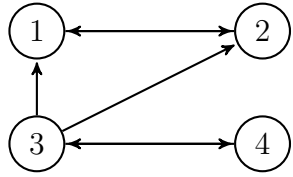
\begin{figure}[h!]
\begin{center}
\begin{tikzpicture}[->,>=stealth',shorten >=1pt,auto, node distance=1.5cm, thick,node/.style={circle,draw}]
                           \node[node]	(1) at (-6cm, 4cm) [fill=white] {$1$};
			\node[node]	(2) at (-3cm, 4cm)  [fill=white] {$2$};
			\node[node]	(3) at (-6cm, 2.5cm) [fill=white] {$3$};
			\node[node]	(4) at (-3cm, 2.5cm)  [fill=white] {$4$};

                                                   \path                                                   
			
				(1) edge node {} (2)
				(2) edge node {} (1)
				(3) edge node {} (4)
				(4) edge node {} (3)
				(3) edge node {} (1)
				(3) edge node {} (2);    
		\end{tikzpicture}	
		\end{center}
		\caption{A four-cell network.}
		\label{contra_exemplo}
	\end{figure}

\begin{exam} \label{ex:qutro}
Consider the four-cell network $G$ in Figure~\ref{contra_exemplo}. Considering the exposition above, we compute the lattices ${\mathcal L}_{W_G}$ and ${\mathcal L}_{L_G}$  of $G$ by considering the referred adaptation of Algorithm 6.5 in  Aguiar and Dias~\cite{AD14}. 

We have that
$$
W_{G} =
	\left( 
 \begin{array}{cccc}
0 &  1 & 1 & 0 \\
1 & 0 & 1  & 0 \\
0 & 0 & 0 & 1 \\
0 & 0 & 1 & 0  
 \end{array}
 \right) \quad \mbox{ and } \quad 
 L_{G} =
	\left( 
 \begin{array}{cccc}
2 &  -1 & -1 & 0 \\
-1 & 2 & -1  & 0 \\
0 & 0 & 1 & -1 \\
0 & 0 & -1 & 1  
 \end{array}
 \right)\, .
$$

The  eigenvalues of $W_G$ are $-1,1$ and the corresponding eigenspaces and generalized eigenspaces are
{\small 
$$
E_{-1} = <(1,-1,0,0),(1,1,-2,2)>,\ E_1 = <(1,1,0,0)>,\ G_1 = <(1,1,0,0), (1,1,1,1)>\, .
$$
}
We start by identifying the generalized polydiagonals associated with the (generalized) eigenvectors of $W_G$. That is, for each (generalized) eigenvector we consider the generalized polydiagonal with less dimension containing the eigenvector.
These are,
$$
\begin{array}{ll}
\Delta_{\footnotesize{\mathcal{P}_1}} = \{ {\bf x}:\ x_1=-x_2, x_3=x_4=0 \},  \qquad &\Delta_{\footnotesize{\mathcal{P}_2}} = \{ {\bf x}:\ x_1=x_2, x_3=-x_4 \},  \\
 \\
\Delta_{\footnotesize{\mathcal{P}_3}} = \{ {\bf x}:\ x_2=x_4=-x_3, x_1=0 \}, &  \Delta_{\footnotesize{\mathcal{P}_4}} = \{ {\bf x}:\ x_1=x_4=-x_3, x_2=0 \}, \\
\\
\Delta_{\footnotesize{\mathcal{P}_5}}  = \{ {\bf x}:\ x_1=x_2, x_3=x_4=0 \}, & \Delta_{\footnotesize{\mathcal{P}_6}}  = \{ {\bf x}:\ x_1=x_2, x_3=x_4 \}.
\end{array}
$$ 
Next, checking whether or not these generalized polydiagonal subspaces have an eigenvector basis, we conclude that, since
$$
\begin{array}{lll}
\Delta_{\footnotesize{\mathcal{P}_1}}  = <(1,-1,0,0)>, \quad & \Delta_{\footnotesize{\mathcal{P}_2}} = <(1,1,-2,2)> \oplus E_{1}, \quad & \Delta_{\footnotesize{\mathcal{P}_3}} =  <(0,-2,2,-2)>, \\
\\
\Delta_{\footnotesize{\mathcal{P}_4}}= <(2,0,-2,2)>, \quad &  \Delta_{\footnotesize{\mathcal{P}_5}} = E_1, \quad & \Delta_{\footnotesize{\mathcal{P}_6}} = G_1, 
\end{array}
$$
they are all left invariant by the matrix $W_G$. All these subspaces are anti-synchrony  subspaces, with the exception of $\Delta_{\footnotesize{\mathcal{P}_6}}$ that is a synchrony subspace. By the results in  \cite{AD14}, they form a sum-dense set for the lattice ${\mathcal L}_{W_G}$. 
Considering the possible sums of two or more of these subspaces, we get one more synchrony and two more anti-synchrony subspaces for $G$ in ${\mathcal L}_{W_G}$,
$$
\begin{array}{ll}
\Delta_{\footnotesize{\mathcal{P}_7}}  = \Delta_{\footnotesize{\mathcal{P}_1}}  \oplus \Delta_{\footnotesize{\mathcal{P}_2}} = \{ {\bf x}:\  x_4=-x_3 \} , \quad & \Delta_{\footnotesize{\mathcal{P}_8}}   = \Delta_{\footnotesize{\mathcal{P}_2}}  \oplus \Delta_{\footnotesize{\mathcal{P}_6}} = \{ {\bf x}:\  x_1=x_2 \}, \\
 \Delta_{\footnotesize{\mathcal{P}_9}}  = \Delta_{\footnotesize{\mathcal{P}_1}} \oplus \Delta_{\footnotesize{\mathcal{P}_5}}  = \{ {\bf x}:\  x_3=x_4=0 \} .
\end{array}
$$
Considering the intersection  of all the subspaces, we get one more element, the anti-synchrony null subspace $\{ 0_{\R^4} \}$ which corresponds to the bottom of the lattice ${\mathcal L}_{W_G}$. We have, then, ${\mathcal L}_{W_G}  = \{ \Delta_{\footnotesize{\mathcal{P}_i}},\ i=1,\ldots, 9 \} \cup P \cup \{ 0_{\R^4} \}$.

The  eigenvalues of $L_G$ are $0,1,2,3$ and the corresponding eigenspaces are
{\small 
$$
E_0 = <(1,1,1,1)>,\ E_1 = <(1,1,0,0)>,\ E_2 = <(1,1,-1,1)>,\ E_3 = <(1,-1,0,0)>\, .
$$
}
The generalized polydiagonals associated with the eigenvectors of $L_G$ are, besides $\Delta_{\footnotesize{\mathcal{P}_1}}$ and $\Delta_{\footnotesize{\mathcal{P}_5}}$,
$$
\begin{array}{ll}
 \Delta_{\footnotesize{\mathcal{P}_{10}}} = \{ {\bf x}:\ x_1=x_2=x_3=x_4 \},  \qquad & \Delta_{\footnotesize{\mathcal{P}_{11}}} = \{ {\bf x}:\ x_1=x_2=x_4 =- x_3\}.
\end{array}
$$ 
We have that $\Delta_{\footnotesize{\mathcal{P}_{10}}} = E_0$ and $\Delta_{\footnotesize{\mathcal{P}_{11}}}=E_2$ are left invariant by the Laplacian matrix $L_G$. Thus, they are a synchrony and an anti-synchrony subspace for $G$, respectively. By the results in  \cite{AD14},  $\Delta_{\footnotesize{\mathcal{P}_1}}$, $\Delta_{\footnotesize{\mathcal{P}_5}}$, $\Delta_{\footnotesize{\mathcal{P}_{10}}}$ and $\Delta_{\footnotesize{\mathcal{P}_{11}}}$ form a sum-dense set for the lattice ${\mathcal L}_{L_G}$. 
Considering the possible sums of two or more of these subspaces, we get the following synchrony and  anti-synchrony subspaces for $G$ in ${\mathcal L}_{L_G}$,
{\tiny 
$$
\Delta_{\footnotesize{\mathcal{P}_2}} =  \Delta_{\footnotesize{\mathcal{P}_5}} \oplus \Delta_{\footnotesize{\mathcal{P}_{11}}}, \quad  
\Delta_{\footnotesize{\mathcal{P}_6}} =  \Delta_{\footnotesize{\mathcal{P}_5}} \oplus \Delta_{\footnotesize{\mathcal{P}_{10}}}, \quad 
\Delta_{\footnotesize{\mathcal{P}_9}}  = \Delta_{\footnotesize{\mathcal{P}_1}} \oplus \Delta_{\footnotesize{\mathcal{P}_5}}, \quad 
\Delta_{\footnotesize{\mathcal{P}_{12}}}  = \Delta_{\footnotesize{\mathcal{P}_{10}}} \oplus \Delta_{\footnotesize{\mathcal{P}_{11}}}  = \{ {\bf x}:\  x_1=x_2=x_4 \}, 
$$
}
{\tiny 
$$
\Delta_{\footnotesize{\mathcal{P}_{7}}}  = \Delta_{\footnotesize{\mathcal{P}_1}} \oplus \Delta_{\footnotesize{\mathcal{P}_{5}}} \oplus \Delta_{\footnotesize{\mathcal{P}_{11}}}, \quad 
\Delta_{\footnotesize{\mathcal{P}_8}}  = \Delta_{\footnotesize{\mathcal{P}_5}} \oplus \Delta_{\footnotesize{\mathcal{P}_{10}}} \oplus \Delta_{\footnotesize{\mathcal{P}_{11}}}, \quad 
\Delta_{\footnotesize{\mathcal{P}_{13}}}  = \Delta_{\footnotesize{\mathcal{P}_1}} \oplus \Delta_{\footnotesize{\mathcal{P}_{5}}} \oplus \Delta_{\footnotesize{\mathcal{P}_{10}}} = \{ {\bf x}:\  x_3=x_4 \}\, .
$$
}

Considering the intersection  of all the subspaces, we get one more element, the anti-synchrony null subspace $\{ 0_{\R^4} \}$ which corresponds to the bottom of the lattice ${\mathcal L}_{L_G}$. Thus, ${\mathcal L}_{L_G}  = \{ \Delta_{\footnotesize{\mathcal{P}_i}},  i\in \{1,2, 5,\ldots,13\} \} \cup P \cup \{ 0_{\R^4} \}$. 
We have, then, that the set of the synchrony and anti-synchrony subspaces for $G$ is ${\mathcal L}_{W_G} \cup {\mathcal L}_{L_G} = \{ \Delta_{\footnotesize{\mathcal{P}_i}},\ i=1,\ldots, 13 \} \cup P \cup \{ 0_{\R^4} \}$.

The partitions $\mathcal{P}_6$ and $\mathcal{P}_8$ are balanced and the partitions $\mathcal{P}_{10}, \mathcal{P}_{12}, \mathcal{P}_{13}$ are strictly exo-balanced. The generalized partitions $\mathcal{P}_1$, $\mathcal{P}_2$, $\mathcal{P}_5$, $\mathcal{P}_7$ and $\mathcal{P}_9$ are odd-balanced, the generalized partitions $\mathcal{P}_1$, $\mathcal{P}_2$, $\mathcal{P}_5$, $\mathcal{P}_7$, $\mathcal{P}_9$ and $\mathcal{P}_{11}$ are linear-balanced, and the generalized partitions $\mathcal{P}_1$, $\mathcal{P}_2$, $\mathcal{P}_3$, $\mathcal{P}_4$,  $\mathcal{P}_5$, $\mathcal{P}_7$ and $\mathcal{P}_{9}$ are even-odd-balanced.

Thus, by Theorem~\ref{thm:final}, we have:\\
(i) the synchrony subspaces for the admissible systems in $I_G$, $I_{G,eo}$ are $\Delta_{\footnotesize{\mathcal{P}_{6}}}$ and $\Delta_{\footnotesize{\mathcal{P}_{8}}}$; \\
(ii) the synchrony subspaces for the admissible systems in $I_{G,0}$, $I_{G,odd}$, $I_{G,l}$  are those in (i) together with $\Delta_{\footnotesize{\mathcal{P}_{10}}}$, $\Delta_{\footnotesize{\mathcal{P}_{12}}}$ and $\Delta_{\footnotesize{\mathcal{P}_{13}}}$; \\
(iii) the anti-synchrony subspaces for the admissible systems in $I_{G,odd}$ are $\Delta_{\footnotesize{\mathcal{P}_{1}}}$, $\Delta_{\footnotesize{\mathcal{P}_{2}}}$, 
$\Delta_{\footnotesize{\mathcal{P}_{5}}}$, $\Delta_{\footnotesize{\mathcal{P}_{7}}}$  and $\Delta_{\footnotesize{\mathcal{P}_{9}}}$; \\
(iv) the anti-synchrony subspaces for the admissible systems in $I_{G,l}$ are those in (iii) together with $\Delta_{\footnotesize{\mathcal{P}_{11}}}$; \\
(v) the anti-synchrony subspaces for the admissible systems in $I_{G,eo}$ are those in (iii) together with $\Delta_{\footnotesize{\mathcal{P}_{3}}}$ and  $\Delta_{\footnotesize{\mathcal{P}_{4}}}$. 
\hfill $\Diamond$
\end{exam}

\begin{rem}\normalfont 
As illustrated by Example~\ref{ex:qutro}, even for a non-regular network $G$ the intersection ${\mathcal L}_{W_G} \cap {\mathcal L}_{L_G}$ can be non-trivial. 
%
\hfill $\Diamond$
\end{rem}

\begin{rem}\normalfont \label{rmk:reg} 
For the particular case of regular networks, ${\mathcal L}_{W_G} = {\mathcal L}_{L_G}$ and so  the network set of synchrony and anti-synchrony subspaces can be obtained using either the eigenvalue and eigenvector structure of the network adjacency matrix or of the Laplacian matrix.
\hfill $\Diamond$
\end{rem}

\section{Conclusions} \label{sec:conclu}		

In this paper, we caracterize the set ${\mathcal L}_{W_G} \cup {\mathcal L}_{L_G} $ of the synchrony and anti-synchrony subspaces of a general weighted network $G$, which corresponds to the generalized polydiagonals invariant under the adjacency and/or Laplacian matrices of $G$. More precisely, the set ${\mathcal L}_{W_G}$ of  synchrony and anti-synchrony subspaces of a general weighted network $G$ corresponds to the generalized polydiagonals that are flow-invariant by any coupled cell system with input additive structure that are even-odd-balanced.  These are in correspondence with the generalized polydiagonals invariant under the network adjacency matrix. The set ${\mathcal L}_{L_G}$ of  synchrony and anti-synchrony subspaces of a general weighted network $G$ corresponds to the generalized polydiagonals that are flow-invariant by any coupled cell system with input additive structure that are linear-balanced. These are in correspondence with the generalized polydiagonals invariant under the network Laplacian matrix.

Much of this work is motivated by the work presented by Neuberger, Sieben, and Swift in \cite{NSS19}, which we extend in several aspects. In \cite{NSS19}, the authors consider undirected networks without weights on the connections. Here, we consider weighted directed networks. In \cite{NSS19}, the associated admissible systems are difference-coupled vector fields, a special class of the input-additive vector fields that we consider here. In our setting we have a more general definition of anti-synchrony subspace in the sense that for the associated tagged partition a part and its counterpart may have a different number of cell elements. Moreover, there can be parts with no counterparts. We have also that, contrary to what happens in \cite{NSS19}, an anti-synchrony subspace can correspond to a generalized polydiagonal susbpace that is invariant by the adjacency matrix of the network and not by its Laplacian matrix. In \cite{NSS19}, the set of anti-synchrony subspaces corresponds to the matched polydiagonals that are left invariant by the Laplacian matrix of the network.
				
\vspace{5mm}

\noindent {\bf Acknowledgments} \\
The authors were partially supported by CMUP, which is financed by national funds through FCT-- Funda\c{c}\~ao para a Ci\^encia e a Tecnologia, I.P., under the project with reference UIDB/00144/2020.

\end{document}